\documentclass[a4paper,11pt]{amsart}

\usepackage{amsmath, amssymb, amsthm}
\usepackage{graphicx}
\usepackage{subfigure}
\usepackage{color}
\usepackage[round]{natbib}
\usepackage[hyperindex,breaklinks]{hyperref}
\usepackage{xspace}
\usepackage{framed}
\usepackage{algorithm}
\usepackage[noend]{algpseudocode}

\usepackage[a4paper]{geometry} \newtheorem{theorem}{Theorem}

\newtheorem{assumption}{Assumption}

\newtheorem{corollary}{Corollary}
\newtheorem{property}{Property}

\newtheorem{example}{Example}
\newtheorem{lemma}[theorem]{Lemma}

\newtheorem{proposition}[theorem]{Proposition}
\newtheorem{remark}{Remark}

\hypersetup{colorlinks = true, urlcolor = blue, linkcolor = blue,
  citecolor = blue }

\newcommand{\xmath}[1]{\ensuremath{#1}\xspace}

\newcommand{\borel}{\xmath{\mathcal{B}}}
\newcommand{\spt}{\xmath{\textnormal{Spt}}}

\newcommand{\psirob}{\xmath{\psi_{\mathnormal{rob}}}}
\newcommand{\psild}{\xmath{\psi_{_{\mathnormal{LD}}}}}
\renewcommand{\epsilon}{\xmath{\varepsilon}}
\renewcommand{\emptyset}{\xmath{\varnothing}}
\newcommand{\Mu}{\xmath{\mathfrak{m}_{\mathcal{U}}(S;\overline{\mathbb{R}})}}

\title[Quantifying model risk]{Quantifying Distributional Model Risk
  via Optimal Transport} \author[Blanchet and Murthy]{{\large
    J\MakeLowercase{ose} B\MakeLowercase{lanchet}} \hspace{30pt}
  {\large K\MakeLowercase{arthyek} M\MakeLowercase{urthy}}\\\\
  \textit{\large C\MakeLowercase{olumbia} U\MakeLowercase{niversity}}}

\address{Columbia University, Department of Industrial Engineering \&
  Operations Research, S. W. Mudd Building, 500 W 120th St., New York,
  NY 10027, United States.}  \email{\{jose.blanchet,
  karthyek.murthy\}@columbia.edu} \date{}
\thanks{The authors gratefully acknowledge support from Norges Bank
  Investment Management and NSF grant CMMI 1436700.}

\begin{document}

\maketitle
\ \vspace{-40pt}
\begin{abstract}
 This paper deals with the problem of quantifying the impact of model
  misspecification when computing general expected values of
  interest. The methodology that we propose is applicable in great
  generality, in particular, we provide examples involving
  path-dependent expectations of stochastic processes. Our approach
  consists in computing bounds for the expectation of interest
  regardless of the probability measure used, as long as the measure
  lies within a prescribed tolerance measured in terms of a flexible
  class of distances from a suitable baseline model. These distances,
  based on optimal transportation between probability measures,
  include Wasserstein's distances as particular cases. The proposed
  methodology is well-suited for risk analysis, as we demonstrate with
  a number of applications. We also discuss how to estimate the
  tolerance region non-parametrically using Skorokhod-type embeddings
  in some of these applications.
\end{abstract}

\section{Introduction.}
\noindent One of the most ubiquitous applications of probability is
the evaluation of performance or risk by means of computations such as
$Ef\left( X\right) =\int fdP$ for a given probability measure $P$, a
function $f$, and a random element $X$. In classical applications,
such as in finance, insurance, or queueing analysis, $X$ is a
stochastic process and $f$ is a path dependent functional.

The work of the modeler is to choose a probability model $P$ which is
both descriptive and, yet, tractable. However, a substantial
challenge, which arises virtually always, is that the model used, $P$,
will in general be subject to misspecification; either because of lack
of data or by the choice of specific parametric distributions.
    
The goal of this paper is to develop a framework to assess the impact
of model misspecification when computing $Ef\left( X\right)$.

The importance of constructing systematic methods that provide
performance estimates which are robust to model misspecification is
emphasized in the academic response to the Basel Accord 3.5, 
\cite{risks2010025}: In its broadest sense,
robustness has to do with (in)sensitivity to underlying model
deviations and/or data changes. Furthermore, here, a whole new field
of research is opening up; at the moment, it is difficult to point to
the right approach.

The results of this paper allow the modeler to provide a bound for the
expectation of interest regardless of the probability measure used as
long as such measure remains within a prescribed tolerance of a
suitable baseline model.

The bounds that we obtain are applicable in great generality, allowing
us to use our results in settings which include continuous-time
stochastic processes and path dependent expectations in various
domains of applications, including insurance and queueing.

Let us describe the approach that we consider more precisely. Our
starting point is a technique that has been actively pursued in the
literature (to be reviewed momentarily), which involves considering a
family of plausible models $\mathcal{P}$, and computing
distributionally robust bounds as the solution to the optimization
problem
\begin{equation}
\sup_{P\in\mathcal{P}}\int fdP, \label{Obj-Paper}%
\end{equation}
over all probability measures in $\mathcal{P}$. The motivation, as
mentioned earlier, is to measure the highest possible risk regardless
of the probability measure in $\mathcal{P}$ used. A canonical object
that appears naturally in specifying the family $\mathcal{P}$ is the
neighborhood $\{ P : d(\mu,P) \leq \delta\},$ where $\mu$ is a chosen
baseline model and $\delta$ is a non-negative tolerance level. Here,
$d$ is a metric that measures discrepancy between probability
measures, and the tolerance $\delta$ interpolates between no ambiguity
$(\delta=0)$ and high levels of model uncertainty ($\delta$
large). The family of plausible models, $\mathcal{P}$, can then be
specified as a single neighborhood (or) a collection of such plausible
neighborhoods. In this paper, we choose $d$ in terms of a transport
cost (defined precisely in Section \ref{Sec-Gen-Results}), and analyze
the solvability of (\ref{Obj-Paper}) and discuss various implications.

Being a flexible class of distances that include the popular
Wasserstein distances as a special case, transport costs allow easy
interpretation in terms of minimum cost associated with transporting
mass between probability measures, and have been widely used in
probability theory and its applications (see, for example,
\citet{rachev1998mass, rachev1998massII, villani2008optimal,
  ambrosio2003optimal} for a massive collection of classical
applications and \citet{Barbour_Xia, Gozlan, nguyen2013,
  wang2012supervised, Fournier2014, canas2012learning,
  solomon2014wasserstein, Wasserstein15} for a
sample of growing list of new applications).


Relative entropy (or) Kullback-Liebler (KL) divergence, despite not
being a proper metric, has been the most popular choice for $d$,
thanks to the tractability of (\ref{Obj-Paper}) when $d$ is chosen as
relative entropy or other likelihood based discrepancy measures (see
\citet{MAFI:MAFI12050, lam2013robust, Glasserman_Xu, MR3299141} for
the use in robust performance analysis, and \citet{Hans_Sarg,
  iyengar2005robust, NE:05, lim2007relative, Jain2010, Ben_Tal,
  Wang2015, Jiang2015, hu2012kullback, doi:10.1287/educ.2015.0134} and
references therein for applications to distributionally robust
optimization). Our study in this paper is directly motivated to
address the shortcoming that many of these earlier works acknowledge:
the absolute continuity requirement of relative entropy; there can be
two probability measures $\mu$ and $\nu$ that produce the same samples
with high probability, despite having the relative entropy between
$\mu$ and $\nu$ as infinite. Such absolute continuity requirements
could be very limiting, particularly so when the models of interest
are stochastic processes defined over time. For instance, Brownian
motion, because of its tractability in computing path-dependent
expectations, is used as approximation of piecewise linear or
piecewise constant processes (such as random walk). However, the use
of the relative entropy would not be appropriate to accommodate these
settings because the likelihood ratio between Brownian motion and any
piecewise differentiable process is not well defined.

Another instance which shows the limitations of the KL divergence
arises when using an It\^{o} diffusion as a baseline model. In such
case, the region $\mathcal{P}$ contains only It\^{o} diffusions with
the same volatility parameter as the underlying baseline model, thus
failing to model volatility uncertainty.

The use of a distance $d$ based on an optimal transport cost (or)
Wasserstein's distance, as we do here, also allows to incorporate the
use of tractable surrogate models such as Brownian motion. For
instance, consider a classical insurance risk problem in which the
modeler has sufficient information on claim sizes (for example, in car
insurance) to build a non-parametric reserve model. But the modeler is
interested in path-dependent calculations of the reserve process (such
as ruin probabilities), so she might decide to use Brownian motion
based approximation as a tractable surrogate model which allows to
compute $Ef\left( X\right )$ easily. A key feature of the
Wasserstein's distances, is that it is computed in terms of a
so-called optimal coupling.  So, the modeler can use any good coupling
to provide a valid bound for the tolerance parameter $\delta$.

The use of tractable surrogate models, such as Brownian motion,
diffusions, and reflected Brownian motion, etc., has enabled the
analysis of otherwise intractable complex stochastic systems. As we
shall illustrate, the bounds that we obtain have the benefit of being
computable directly in terms of the underlying tractable surrogate
model. In addition, the calibration of the tolerance parameter takes
advantage of well studied coupling techniques (such as Skorokhod
embedding). So, we believe that the approach we propose naturally
builds on the knowledge that has been developed by the applied
probability community.

We summarize our main contributions in this paper below:

a) Assuming that $X$ takes values in a Polish space, and using a wide
range of optimal transport costs (that include the popular
Wasserstein's distances as special cases) we arrive at a dual
formulation for the optimization problem in \eqref{Obj-Paper}, and
prove strong duality -- see Theorem \ref{THM-STRONG-DUALITY}.

b) Despite the infinite dimensional nature of the optimization problem
in \eqref{Obj-Paper}, we show that the dual problem admits a one
dimensional reformulation which is easy to work with. We provide
sufficient conditions for the existence of an optimizer $P^\ast$ that
attains the supremum in \eqref{Obj-Paper} -- see Section
\ref{Sec-Primal-Opt-Existence}.  Using the result in a), for upper
semicontinuous $f$, we show how to characterize an optimizer
$P^{\ast}$ in terms of a coupling involving $\mu$ and the chosen
transport cost -- see Remark \ref{Rem-Opt-Measure-Structure} and
Section \ref{Eg-Rob-Probabilities}.

c) We apply our results to various problems such as robust evaluation
of ruin probabilities using Brownian motion as a tractable surrogate
(see Section \ref{Sec-App-LC}), general multidimensional first-passage
time probabilities (see Section \ref{Sec-App-FPP}), and optimal
decision making in the presence of model ambiguity (see Section
\ref{Sec-App-St-Opt}).

d) We discuss a non-parametric method, based on a Skorokhod-type
embedding, which allows to choose $\delta$ (these are particularly
useful in tractable surrogate settings discussed earlier). The
specific discussion of the Skorokhod embedding for calibration is
given in Appendix \ref{Appendix-Embeddings}. We also discuss how
$\delta$ can be chosen in a model misspecification context, see the
discussion at the end of Section \ref{Sec-App-FPP}.

In a recent paper, \citet{esfahani2015data}, the authors also consider
distributionally robust optimization using a very specific
Wasserstein's distance. Similar formulations using Wasserstein
distance based ambiguity sets have been considered in
\citet{MR2354780, MR2874755} and \citet{Zhao_Guan} as well. The form
of the optimization problem that we consider here is basically the
same as that considered in \citet{esfahani2015data}, but there are
important differences that are worth highlighting. In
\citet{esfahani2015data}, the authors concentrate on the specific case
of baseline measure being empirical distribution of samples obtained
from a distribution supported in $\mathbb{R}^d,$ and the function $f$
possessing a special structure. In contrast, we allow the baseline
measure $\mu$ to be supported on general Polish spaces and study a
wide class of functions $f$ (upper semicontinuous and integrable with
respect to $\mu$).  The use of the Skorokhod embeddings that we
consider here, for calibration of the feasible region, is also novel
and not studied in the present literature.




Another recent paper, \citet{Gao_Kleyweget}, (which is an independent
contribution made public a few days after we posted the first version
of this paper in arXiv), presents a very similar version of the strong
duality result shown in this paper. One difference that is immediately
apparent is that the authors in \citet{Gao_Kleyweget} concentrate on
the case in which a specific cost function $c(\cdot,\cdot)$ used to
define transport cost is of the form $c(x,y)=d(x,y)^p,$ for some
$p \geq 1$ and metric $d(\cdot,\cdot),$ whereas we only impose that
$c(\cdot,\cdot)$ is lower semicontinuous. {Allowing for general lower
  semicontinuous cost functions that are different from $d(x,y)^p$ is
  useful, as demonstrated in applications towards distributionally
  robust optimization and machine learning in
  \citet{blanchet2016robust}.}  Another important difference is that,
as far as we understand, the proof given in \citet{Gao_Kleyweget}
appears to implicitly assume that the space $S$ in which the random
element of interest $X$ takes values is locally compact. For example,
the proof of Lemma 2 in \citet{Gao_Kleyweget} and other portions of
the technical development appear to use repeatedly that a closed norm
ball is compact. However, this does not hold in any infinite
dimensional topological vector space\footnote{see, for example,
  Example 3.100 and Theorem 5.26 in \citet{aliprantis1999infinite}},
thus excluding important function spaces like $C[0,T]$ and $D[0,T]$
(denoting respectively the space of continuous and c\`{a}dl\`{a}g
functions on interval $[0,T]$), that are at the center of our
applications. As mentioned earlier, our focus in this paper is on
stochastic process applications, whereas the authors in
\citet{Gao_Kleyweget} put special emphasis on stochastic optimization
in $\mathbb{R}^d.$

The rest of the paper is organized as follows: In Section
\ref{Sec-Gen-Results} we shall introduce the assumptions and discuss
our main result, whose proof is given in Section
\ref{Sec-Duality-Proof}. The proof is technical due to the level of
generality that is considered. More precisely, it is the fact that the
cost functions used to define optimal transport costs need not be
continuous, and the random elements that we consider need not take
values in a locally compact space (like $\mathbb{R}^d$) which
introduces technical complications, including issues like
measurability of a key functional appearing in the dual
formulation. In preparation to the proof, and to help the reader gain
some intuition of the result, we present a one dimensional example
first, in Section \ref{Sec-App-LC}.  Then, after providing the proof
of our main result in Section \ref{Sec-Duality-Proof} and conditions
for existence of the worst-case probability distribution that attains
the supremum in \eqref{Obj-Paper} in Section
\ref{Sec-Primal-Opt-Existence}, we discuss additional examples in
Section \ref{Sec-App}.


\section{Our Main Result.}
\label{Sec-Gen-Results}
\noindent In order to state our main result, we need to introduce some
notation and define the optimal transport cost between probability
measures.
\subsection{Notation and Definitions.}
\label{Sec-Notn-Assump} 
For a given Polish space $S,$ we use $\mathcal{B}(S)$ to denote the
associated Borel $\sigma$-algebra. Let us write $P(S)$ and $M(S),$
respectively, to denote the set of all probability measures and finite
signed measures on $(S,\mathcal{B}(S))$. {For any $\mu \in P(S),$ let
  $\mathcal{B}_\mu(S)$ denote the completion of $\mathcal{B}(S)$ with
  respect to $\mu;$ the unique extension of $\mu$ to
  $\mathcal{B}_\mu(S)$ that is a probability measure is also denoted
  by $\mu,$ and the measure $\mu$ in $\int \phi d\mu$ is to be
  interpreted as this extension defined on $\mathcal{B}_\mu(S)$
  whenever
  $\phi:(S,\mathcal{B}(S)) \rightarrow (\mathbb{R},
  \mathcal{B}(\mathbb{R}))$ is not measurable, but instead,
  $\phi: (S,\mathcal{B}_\mu(S)) \rightarrow (\mathbb{R},
  \mathcal{B}(\mathbb{R}))$ is measurable.} We shall use $C_{b}(S)$ to
denote the space of bounded continuous functions from $S$ to
$\mathbb{R},$ and $\spt(\mu)$ to denote the support of a probability
measure $\mu.$ For any $\mu \in P(S)$ and $p \geq 1,$ $L^{p}(d\mu)$
denotes the collection of Borel measurable functions
$h:S \rightarrow \mathbb{R}$ such that $\int|h|^{p} d\mu<\infty.$ {The
  universal $\sigma-$algebra is defined by
  $\mathcal{U}(S) = \cap_{\mu \in P(S)} B_\mu(S).$ We use
  $\overline{\mathbb{R}} = \mathbb{R} \cup \{-\infty, \infty\}$ to
  denote the extended real line, and $\Mu$ to denote the collection of
  measurable functions
  $\phi: (S,\mathcal{U}(S)) \rightarrow
  (\overline{\mathbb{R}},\mathcal{B}(\overline{\mathbb{R}})).$ As
  $\mathcal{U}(S) \subseteq \mathcal{B}_\mu(S)$ for every
  $\mu \in P(S),$ any $\phi \in \Mu$ is also measurable when $S$ and
  $\overline{\mathbb{R}}$ are equipped, respectively, with the
  $\sigma$-algebras $\mathcal{B}_\mu(S)$ and
  $\mathcal{B}(\overline{\mathbb{R}}).$ Consequently, the integral
  $\int \phi d\mu$ is well-defined for any non-negative
  $\phi \in \Mu.$} In addition, for any $\mu \in P(S),$ we say that
$\mu$ is \textit{concentrated} on a set $A \in B_\mu(S)$ if
$\mu(A) = 1.$\vspace{-10pt}
  \subsubsection*{Optimal transport cost.}  For any two probability
  measures $\mu_1$ and $\mu_2$ in $P(S)$, let $\Pi (\mu_1,\mu_2)$
  denote the set of all joint distributions with $\mu_1$ and $\mu_2$
  as respective marginals. In other words, the set $\Pi(\mu_1,\mu_2)$
  represents the set of all couplings (also called transport plans)
  between $\mu_1$ and $\mu_2.$ Throughout the paper, we assume that
\begin{assumption}[A1]
  $c:S\times S\rightarrow\mathbb{R}_+$ is a nonnegative lower
  semicontinuous function satisfying $c(x,y)=0$ if and only if
  $x = y.$
\end{assumption}
Then the \textit{optimal transport cost} associated with the cost
function $c$ is defined as,
\begin{align}
  d_{c}(\mu_1, \mu_2):=\inf\left\{ \int cd\pi:\pi\in\Pi(\mu_1,\mu_2)\right\} ,
  \quad \mu_1, \mu_2 \in P(S).
\label{Defn-Transport-Cost}
\end{align}
Intuitively, the quantity $c(x,y)$ specifies the cost of transporting
unit mass from $x$ in $S$ to another element $y$ of $S.$ Given a lower
semicontinuous cost function $c(x,y)$ and a coupling
$\pi \in \Pi(\mu_1, \mu_2),$ the integral $\int cd\pi$ represents the
expected cost associated with the coupling (or transport plan) $\pi.$
For any non-negative lower semi-continuous cost function $c,$ it is
known that the `optimal transport plan' that attains the infimum in
the above definition exists (see Theorem 4.1 of
\citet{villani2008optimal}), and therefore, the optimal transport cost
$d_c(\mu_1,\mu_2)$ corresponds to the lowest transport cost that is
attainable among all couplings between $\mu_1$ and $\mu_2.$

If the cost function $c$ is symmetric (that is, $c(y,x) = c(x,y)$ for
all $x,y$), and it satisfies triangle inequality, one can show that
the minimum transportation cost $d_c(\cdot,\cdot)$ defines a metric on
the space of probability measures. For example, if $S$ is a Polish
space equipped with metric $d,$ then taking the cost function to be
$c(x,y) = d(x,y),$ renders the transport cost $d_c(\mu,\nu)$ to be
simply the Wasserstein distance of first order between $\mu$ and
$\nu.$ 
Unlike the Kullback-Liebler divergence (or) other likelihood based
divergence measures, the Wasserstein distance is a proper metric on
the space of probability measures. More importantly, Wasserstein
distances do not restrict all the probability measures in the
neighborhoods such as
  $\left\{ \nu \in P(S) : d_c(\mu,\nu) \leq \delta \right\}$
  to share the same support as that of $\mu$ 
(see, for example, Chapter 6 in
\citet{villani2008optimal} for properties of Wasserstein distances).

\subsection{Primal Problem.}
\label{Sec-primal-prob}
Underlying our discussion we have a Polish space $S$, which is the
space where the random elements of the given probability model
$\mu \in P(S)$ takes values. Given $\delta > 0,$ the objective, as
mentioned in the Introduction, is to evaluate
\begin{align*}
  \sup \left\{ \int f d\nu: d_c(\mu,\nu) \leq \delta \right\},
\end{align*}
for any function $f$ that satisfies the assumption that
\begin{assumption}[A2]
  $f\in L^{1}(d\mu)$ is upper semicontinuous.
\end{assumption} 
{As the integral $\int f d\nu$ may equal $\infty - \infty$ for some
  $\nu$ satisfying $d_c(\mu,\nu) \leq \delta,$ only for the purposes
  of interpreting the supremum above, we let
  $\sup\{\infty, \infty - \infty\} = \infty$\footnote{{This is
      because, under the assumption that $f \in L^1(d\mu),$ for every
      probability measure $\nu$ such that $d_c(\mu,\nu) \leq \delta$
      and $\int f^- d\nu = \infty,$ one can identify a probability
      measure $\nu'$ satisfying $d_c(\mu,\nu') \leq \delta,$
      $\int f^+ d\nu' \geq \int f^+ d\nu,$ and
      $\int f^-d\nu' < \infty = \int f^-d\nu;$ see Corollary
      \ref{Cor-Notn-Clar} in Appendix \ref{Appendix-Proofs} for a
      simple construction of such a measure $\nu'$ from measure $\nu.$
      Consequently, when computing the supremum of $\int fd\nu,$ it is
      meaningful to restrict our attention to probability measures
      $\nu$ satisfying $\int f^-d\nu < \infty,$ and interpret
      $\sup \left\{ \int f d\nu: d_c(\mu,\nu) \leq \delta \right\}$ as
      $\sup\{\int f d\nu: d_c(\mu,\nu) \leq \delta, \int f^-d\nu <
      \infty\}.$ The notation,
      $\sup\{\infty, \infty-\infty\} = \infty,$ simply facilitates
      this interpretation in this context. 
    }
  }.}  The value of this optimization problem provides a bound to the
expectation of $f,$ regardless of the probability measure used, as
long as the measure lies within $\delta$ distance (measured in terms
of $d_c$) from the baseline probability measure $\mu.$ In applied
settings, $\mu$ is the probability measure chosen by the modeler as
the baseline distribution, and the function $f$ corresponds to a risk
functional (or) performance measure of interest, for example, expected
losses, probability of ruin, etc.

As the infimum in the definition of the optimal transport cost $d_c$
is attained for any given non-negative lower semicontinuous cost
function $c$ (see Theorem 4.1 of \citet{villani2008optimal}), we
rewrite the quantity of interest as below:
\begin{align*}
  I    :=\sup\left\{  \int fd\nu: \ d_{c}(\mu,\nu)\leq\delta\right\}
       =\ \sup\bigg\{  \int f(y)d\pi(x,y): \pi\in \bigcup_{\nu \in
  P(S)}  \hspace{-6pt}\Pi(\mu,\nu),\ \int cd\pi\leq
       \delta\bigg\},\nonumber
\end{align*}
which, in turn, is an optimization problem with linear objective
function and linear constraints. If we let
\[I(\pi):=\int f(y)d\pi(x,y) \quad \text{ and } \quad 
 \Phi_{\mu,\delta}:= \bigg\{\pi\in \bigcup_{\nu \in
  P(S)} \hspace{-6pt} \Pi(\mu,\nu): \int cd\pi\leq\delta \bigg\}\] for
brevity, then
\begin{equation}
I=\sup\left\{  I(\pi):\pi\in\Phi_{\mu,\delta}\right\}
\label{EQ-Primal-Problem}
\end{equation}
is the quantity of interest.

\subsubsection{The Dual Problem and Weak Duality.}
\label{Sec-Dual-Problem}
Define $\Lambda_{c,f}$ to be the collection of all pairs
$(\lambda,\phi )$ such that $\lambda$ is a non-negative real number,
{$\phi \in \Mu$} and
\begin{equation}
  \phi(x)+\lambda c(x,y)\geq f(y),\text{ for all }x,y.
\label{Dual-Condition}
\end{equation}
For every such $(\lambda, \phi) \in \Lambda_{c,f}$ consider
\[
J(\lambda,\phi):=\lambda\delta+\int\phi d\mu.
\]
As $\int fd\mu$ is finite and $\phi \geq f$ whenever $\phi$ satisfies
\eqref{Dual-Condition}, the integral in the definition of
$J(\lambda, \phi)$ avoids ambiguities such as $\infty - \infty$ for
any $(\lambda,\phi) \in \Lambda_{c,f}.$ Next, for any
$\pi\in\Phi_{\mu,\delta}$ and $(\lambda,\phi)\in\Lambda_{c,f},$ see
that
\begin{align*}
J(\lambda,\phi)  &  =\lambda\delta+\int\phi(x)d\pi(x,y)\\
&  \geq\lambda\delta+\int\big(f(y)-\lambda c(x,y)\big)d\pi(x,y)\\
&  =\int f(y)d\pi(x,y)+\lambda\left(  \delta-\int c(x,y)d\pi(x,y)\right) \\
&  \geq\int f(y)d\pi(x,y)\\
&  =I(\pi).
\end{align*}
Consequently, we have
\begin{equation}
J:=\inf\big\{J(\lambda,\phi):(\lambda,\phi)\in\Lambda_{c,f}\big\}\geq I.
\label{EQ-WEAK-DUALITY}
\end{equation}
Following the tradition in optimization theory, we refer to the above
infimum problem that solves for $J$ in \eqref{EQ-WEAK-DUALITY} as the
dual to the problem that solves for $I$ in \eqref{EQ-Primal-Problem},
which we address as the primal problem. Our objective in the next
section is to identify whether the primal and the dual problems have
same value (that is, do we have that $I=J$?).

\subsection{Main Result: Strong Duality Holds.}

Recall that the feasible sets for the primal and dual problems, respectively,
are:
\begin{subequations}
\begin{align}
  \Phi_{\mu,\delta}  &  := \bigg\{  \pi\in \bigcup_{\nu \in P(S)}
                       \hspace{-4pt}\Pi(\mu, \nu): \ \int c 
                       d\pi\leq\delta\bigg\}  \text{ and }\label{Primal-Feas-Set}\\
  \Lambda_{c,f}  &  := \big\{ (\lambda, \phi): \ \lambda\geq0, \ \phi
                   \in \Mu, \ \phi(x) + \lambda c(x,y) \geq f(y) \text{ for all } x,y \in
                   S\big\}. \label{Dual-Feas-Set}%
\end{align}
The corresponding primal and dual problems are
\end{subequations}
\begin{subequations}
\begin{align*}
I := \sup\left\{  \int f(y) d\pi(x,y) : \pi\in\Phi_{\mu,\delta} \right\}
\text{ and } J := \inf\left\{  \lambda\delta+ \int\phi d\mu: (\lambda, \phi)
\in\Lambda_{c,f}\right\}  
\end{align*}
For brevity, we have identified the primal and dual objective
functions as $I(\pi)$ and $J(\lambda, \phi)$ respectively. As the
identified primal and dual problems are infinite dimensional, it is
not immediate whether they have same value (that is, is $I = J$?). The
objective of the following theorem is to verify that, for a broad
class of performance measures $f,$ indeed $I$ equals
$J.$ 
\end{subequations}
\begin{theorem}
  Under the Assumptions (A1) and (A2), \ \newline(a) $I=J$. In other
  words,
\[
\sup\big\{I(\pi):\pi\in\Phi_{\mu,\delta}\big\}=\inf\big\{J(\lambda
,\phi):(\lambda,\phi)\in\Lambda_{c,f}\big\}.
\]
(b) For any $\lambda\geq0,$ define $\phi_{_{\lambda}}:S\rightarrow
\mathbb{R}\cup\{\infty\}$ as follows:
\[
\phi_{_{\lambda}}(x):=\sup_{y\in S}\big\{f(y)-\lambda c(x,y)\big\}. 
\]
There exists a dual optimizer of the form
$(\lambda, \phi_{_{\lambda}}),$ for some $\lambda \geq 0.$ In
addition, any feasible $\pi ^{\ast}\in\Phi_{\mu,\delta}$ and
$(\lambda^{\ast},\phi_{_{\lambda^{\ast}}})\in\Lambda_{c,f}$ are primal
and dual optimizers, satisfying $I(\pi^\ast) = J(\lambda^\ast,
\phi_{_{\lambda^\ast}})$, if and only if 
\begin{subequations}
\begin{align}
&  f(y)-\lambda^{\ast}c(x,y)=\sup_{z\in S}\big\{f(z)-\lambda^{\ast
}c(x,z)\big\},\quad\pi^{\ast}a.s.,\text{ and }\label{Comp-Slack-1}\\
&  \lambda^{\ast}\left(  \int c(x,y)d\pi^{\ast}(x,y)-\delta\right)  =0.
\label{Comp-Slack-2}%
\end{align}
\end{subequations}
\label{THM-STRONG-DUALITY}
\end{theorem}
\vspace{-5pt}
\noindent As the measurability of the function $\phi_{_{\lambda}}$ is
not immediate, we establish that $\phi_{_{\lambda}} \in \Mu$ in
Section \ref{Sec-Duality-Proof}, where the proof of Theorem
\ref{THM-STRONG-DUALITY} is also presented. Sufficient conditions for
the existence of a primal optimizer $\pi^\ast \in \Phi_{\mu,\delta},$
satisfying $I(\pi^\ast) = I,$ are presented in Section
\ref{Sec-Primal-Opt-Existence}.  For now, we are content discussing
the insights that can be obtained from Theorem
\ref{THM-STRONG-DUALITY}.  \bigskip
\begin{remark}[on the value of the dual problem]
  \textnormal{First, we point out the following useful characterization of the
  optimal value, $I,$ as a consequence of Theorem
  \ref{THM-STRONG-DUALITY}:
\begin{equation}
  I=\inf_{\lambda\geq0}\left\{ \lambda\delta+E_{\mu}\left[ \sup_{y\in
        S}\big\{f(y)-\lambda c(X,y)\big\}\right] \right\},
  \label{EQ-Duality-Implication}
  \end{equation}
  where the right hand side\footnote{As $\phi_{\lambda}(x)\geq f(x),$
    there is no ambiguity, such as the form $\infty-\infty,$ in the
    definition of integral $\int\phi_{\lambda}(x)d\mu(x)$} is simply a
  univariate reformulation of the dual problem.  This characterization
  of $I$ follows from the strong duality in Theorem
  \ref{THM-STRONG-DUALITY} and the observation that
  $J(\lambda, \phi_{_\lambda}) \leq J(\lambda, \phi),$ for every
  $(\lambda, \phi) \in \Lambda_{c,f}.$ The significance of this
  observation lies in the fact that the only probability measure
  involved in the right-hand side of \eqref{EQ-Duality-Implication} is
  the baseline measure $\mu,$ which is completely characterized, and
  is usually chosen in a way that it is easy to work with (or) draw
  samples from. In effect, the result indicates that the infinite
  dimensional optimization problem in \eqref{EQ-Primal-Problem} is
  easily solved by working on the univariate reformulation in the
  right hand side of
  \eqref{EQ-Duality-Implication}.} \label{Rem-Duality-Usefulness}
\end{remark}
\smallskip
\begin{remark}[on the structure of primal optimal transport
  plan\footnote{In the literature of optimal transportation 
    of probability measures, it is common to refer joint probability
    measures alternatively as transport plans, and the integral
    $\int cd\pi$ as the cost of transport plan $\pi.$ We follow this
    convention to allow easy interpretations.}] 
  \textnormal{It is evident from the characterisation of an optimal measure
  $\pi^{\ast}$ in Theorem \ref{THM-STRONG-DUALITY}(b) that {if
    $\pi^\ast$ exists,} it is concentrated on
  $ \left\{ (x,y)\in S\times S: y \in \arg \max_{z\in
      S}\{f(z)-\lambda^{\ast}c(x,z)\}\right\} .  $ Thus, a worst-case
  joint probability measure $\pi^{\ast}$ gets identified with a
  transport plan that transports mass from $x$ to the optimizer(s) of
  the local optimization problem
  $\sup_{z\in S}\big\{f(z)-\lambda^{\ast}c(x,z)\big\}.$ According to
  the complementary slackness conditions \eqref{Comp-Slack-1} and
  \eqref{Comp-Slack-2}, such a transport plan $\pi^\ast$ satisfies one
  of the following two cases:\\
  \textsc{Case 1: $\lambda^\ast > 0:$} the transport plan $\pi^\ast$
  necessarily costs $\int cd\pi^\ast = \delta,$\\
  \textsc{Case 2: $\lambda^\ast = 0:$} {the transport plan $\pi^\ast$
    satisfies $\int cd\pi^\ast \leq \delta$ (follows from primal
    feasibility of $\pi^\ast$) and $f(y)$ equals the constant
    $\sup_{z \in S} f(z),$ $\pi^\ast$ almost surely (follows from
    \eqref{Comp-Slack-1} assuming that $\arg \max_{z \in S} f(z)$
    exists). Recall from \eqref{EQ-Duality-Implication} that
    $I = \lambda^\ast \delta + \sup_{z \in S}\{f(z) - \lambda^\ast
    c(x,z)\} = \sup_{z \in S} f(z),$ which is in agreement with the
    structure described for the primal optimizer here when
    $\lambda^\ast = 0.$ The interpretation is that the budget for
    quantifying ambiguity, $\delta,$ is sufficiently large to move all
    the probability mass to $\arg \max_{z \in S} f(z),$ thus making
    $I$ as large as possible.}}
\label{Rem-Opt-Measure-Structure}
\end{remark}
\smallskip

\begin{remark}[On the uniqueness of a primal optimal transport plan] 
  \textnormal{Suppose that there exists a primal optimizer
  $\pi^\ast \in \Phi_{\mu,\delta}$ and a dual optimizer
  $(\lambda^\ast, \phi_{_{\lambda^\ast}}) \in \Lambda_{c,f}$
  satisfying
  $I(\pi^\ast) = J(\lambda^\ast, \phi_{_{\lambda^\ast}}) < \infty.$ In
  addition, suppose that for $\mu-$almost every $x \in S,$ there is a
  unique $y \in S$ that attains the supremum in
  $\sup_{y \in S}\{ f(y) - \lambda^\ast c(x,y)\}.$ Then the primal
  optimizer $\pi^\ast$ is unique because of the following reasoning:
  For every $x \in S,$ if we let $T(x)$ denote the unique maximizer,
  $\arg \max\{f(y) - \lambda^\ast c(x,y)\},$ then it follows from
  Proposition 7.50(b) of \cite{bertsekas1978stochastic} that the map
  $T: (S, \mathcal{U}(S)) \rightarrow (S, \mathcal{B}(S))$ is
  measurable. If $(X,Y)$ is a pair jointly distributed according to
  $\pi^\ast$, then due to complementarity slackness condition
  \eqref{Comp-Slack-1}, we must have that $Y = T(X),$ almost
  surely. As any primal optimizer must necessarily assign entire
  probability mass to the set $\{(x,y): y = T(x)\}$ (due to
  complementarity slackness condition \eqref{Comp-Slack-1}), the
  primal optimizer is unique.}
  \label{Rem-Uniqueness-Primal-Opt}
\end{remark}

\begin{remark}[on $\epsilon-$optimal transport plans]
  \textnormal{Given $\epsilon > 0$ and a dual optimal pair
    $(\lambda^\ast,
    \phi_{_{\lambda^\ast}}),$ 
any
$\pi_\epsilon \in \Phi_{\mu,\delta}$ is $\epsilon-$primal optimal if
and only if
\begin{align}
  \int \left( \phi_{\lambda^\ast}(x) -\big( f(y) - \lambda^\ast c(x,y)
  \big)\right) d\pi_{\epsilon}(x,y) \quad +  \quad \lambda^\ast \left( \delta -
  \int cd\pi_\epsilon\right) \ \leq \ \epsilon,
\label{Eps-Opt-Char-1}
\end{align}
where both the summands in the above expression are necessarily
nonnegative.  This follows trivially by substituting the following
expressions for $I$ and $I(\pi_\epsilon)$ in
$I - I(\pi_{\epsilon}) \leq \epsilon.$
\begin{align*}
  I = J(\lambda^\ast, \phi_{_\lambda^\ast}) &= \lambda^\ast\delta + \int
  \phi_{_{\lambda^\ast}}(x)d\pi_{\epsilon}(x,y)  \text{ and }\\
  I(\pi_\epsilon)
  = \int f(y)d\pi_{\epsilon}(x,y) &= \int  \big(f(y) -\lambda^\ast
  c(x,y) \big) d\pi_{\epsilon}(x,y) + \lambda^\ast  \int 
  c(x,y)d\pi_{\epsilon}(x,y).
\end{align*}
As both the summands in \eqref{Eps-Opt-Char-1} are nonnegative, for
any $\pi_\epsilon \in \Phi_{\mu,\delta}$ such that
$I(\pi_\epsilon) = \pi_\epsilon(S \times A) \geq I - \epsilon,$ we have,
  \begin{align}
    \int \big( \phi_{_{\lambda^\ast}}(x) - f(y) -
    \lambda^\ast c(x,y) \big)d\pi_\epsilon(x,y) \ \leq\  \epsilon \quad \text{ and }
    \quad  \left(\delta - \frac{\epsilon}{\lambda^\ast }\right)^+ \leq 
    \int cd\pi_\epsilon  \leq \delta, 
    \label{Inter-Eps-Opt}
  \end{align} 
where $a^+ := \max\{a, 0\}$ for any $a \in \mathbb{R}.$} 
\label{Rem-Eps-Opt-Measure-Structure}
\end{remark}




\begin{remark}
  \textnormal{{If $\sup_{y \in S} f(y)/(1+c(x,y)) = \infty$ (for
      example, if $f(y)$ grows to $\infty$ at a rate faster than the
      rate at which the transport cost function $c(x,y)$ grows) for
      every $x$ in a set $A \subseteq S$ such that $\mu(A) > 0,$ then
      $I = J = \infty.$ This is because, for every $\lambda \geq 0$
      and $x \in A,$
      $\phi_\lambda(x) = \sup_{y \in S}\{ f(y) -\lambda c(x,y)\} =
      \infty,$ and consequently,
      $\lambda \delta + \int \phi_{\lambda}d\mu = \infty$ for very
      $\lambda \geq 0;$ therefore $I = J = \infty$ as a consequence of
      Theorem \ref{THM-STRONG-DUALITY}. Requiring the objective $f$ to
      not grow faster than the transport cost $c$ may be useful from a
      modeling viewpoint as it offers guidance in understanding
      choices of transport cost functions that necessarily yield
      $\infty$ as the robust estimate
      $\sup\{\int fd\nu: d_c(\mu,\nu) \leq \delta\}$.}}
  \label{Rem-Growth-Condition}
\end{remark}
\smallskip

We next discuss an important special case of \textnormal{Theorem
  \ref{THM-STRONG-DUALITY}, namely, computing worst case
  probabilities.} The applications of this special case, in the
context of ruin probabilities, is presented in Section
\ref{Sec-App-LC} and Section \ref{Sec-App-FPP}. Other applications of
Theorem \ref{THM-STRONG-DUALITY}, broadly in the context of
distributionally robust optimization, are available in
\citet{esfahani2015data, Zhao_Guan}, \citet{Gao_Kleyweget} and
\citet{blanchet2016robust}, and as well in Example
\ref{Eg-Choosing-Reinsurance-Prop} in Section \ref{Sec-App-St-Opt} of
this paper.  
\subsection{Application of Theorem \ref{THM-STRONG-DUALITY} for
  computing worst-case probabilities.}
 \label{Eg-Rob-Probabilities}
 Suppose that we are interested in computing,
\begin{equation}
  I=\mathnormal{\sup\{P(A):d_{c}(\mu, P)\leq\delta\},} \label{Eq_Worst_Prob}
\end{equation}
where $A$ is a nonempty closed subset of the Polish space $S.$ Since
$A$ is closed, the indicator function $f(x) =\mathbf{1}_A(x) $ is
upper semicontinuous and therefore we can apply Theorem
{\ref{THM-STRONG-DUALITY} to address (\ref{Eq_Worst_Prob}). In order
  to apply Theorem {\ref{THM-STRONG-DUALITY}, we first observe that
\[
\sup_{y\in S}\left\{  \mathbf{1}_{A}(y)-\lambda c(x,y)\right\}
=\big(1-\lambda c(x,A)\big)^{+},
\]
where $c(x,A):=\inf\{c(x,y):y\in A\}$ is the lowest cost possible in
transporting unit mass from $x$ to some $y$ in the set $A,$ and
$a^{+}$ denotes the positive part of the real number $a.$
Consequently, Theorem \ref{THM-STRONG-DUALITY} guarantees that the
quantity of interest, $\mathnormal{I}$, in (\ref{Eq_Worst_Prob}), can
be computed by simply solving, 
\begin{align}
I=\inf_{\lambda\geq0}\left\{  \lambda\delta+E_{\mu}\left[  \big(1-\lambda
c(X,A)\big)^{+}\right]  \right\}.
\label{Probability-1d-Problem}
\end{align}
If the infimum in the above expression for $I$ is attained
at $\lambda^\ast = 0,$ then merely by substituting $\lambda = 0$ in
$\lambda\delta+E_{\mu}[ (1-\lambda c(X,A))^{+}]$, we obtain $I = 1.$
In addition, due to complementary slackness conditions
\eqref{Comp-Slack-1} and \eqref{Comp-Slack-2}, we have
$\pi^\ast(\mathbf{1}_A(y) = 1) = 1,$ and
$\int cd\pi^\ast \leq \delta,$ for any optimal transport plan
$\pi^\ast \in \Phi_{\mu,\delta}$ satisfying $I = I(\pi^\ast).$ Here,
as $f(x) = \mathbf{1}_A(x),$ we have 
\begin{equation*}
  I(\pi) = \int \mathbf{1}_A(y)d\pi(x,y) = \pi(S
  \times A), \quad\quad \pi \in \Phi_{\mu,\delta}. 
\end{equation*} 
Next, as we turn our attention towards the structure of an optimal
transport plan 
when $\lambda^\ast > 0,$ let us assume, for
ease of discussion, that
$Proj_A(x) := \{ y \in A: c(x,y) = c(x,A)\} \neq \emptyset$ for every
$x \in S.$
Unless indicated otherwise, let us assume that $\lambda^\ast > 0$ in
the rest of this discussion. For every $x \in S$ and
$\lambda^\ast > 0,$ observe that
\begin{align*}
 \arg\max_{y \in S} \big\{ \mathbf{1}_A(y) - \lambda^\ast c(x,y)
  \big\}    = 
  \begin{cases}
    Proj_A(x) &\text{ if } 0 \leq c(x,A) <
    \frac{1}{\lambda^\ast},\\
    Proj_A(x) \cup \{x\}  &\text{ if }  c(x,A) =
    \frac{1}{\lambda^\ast},\\ 
   \{ x \} & \text{ otherwise.}
  \end{cases}
\end{align*}
Then, Part (b) of Theorem \ref{THM-STRONG-DUALITY} allows us to
conclude that, an optimal transport plan
$\pi^\ast \in \Phi_{\mu,\delta}$ satisfying
$\pi^\ast(S \times A) = I,$ if it exists, is concentrated on
\begin{align*}
  S^\ast 
  &= \left\{ (x,y):c(x,A)\leq
  \frac{1}{\lambda^{\ast}},\ y\in\mathnormal{Proj}\xspace_{A}(x)\right\}
\cup\left\{ (x,x):c(x,A) \geq \frac{1}{\lambda^{\ast}}\right\}. 
\end{align*}
{Next, let $\Pi_{S^\ast}(\mu)$ denote the set of probability measures
  $\pi$ with $\mu$ as marginal for the first component, and satisfying
  $\pi(S^\ast) = 1.$ In other words,
\begin{align*}
  \Pi_{S^\ast}(\mu) := \big\{ \pi \in P(S \times S): \ \pi(A \times S) =
  \mu(A) \text{ for all } A \in \mathcal{B}(S), \ \pi(S^\ast) = 1
  \big\}.  
\end{align*}
Recall that $c(x,x) = 0$ for any $x \in S,$ and $c(x,y) = c(x,A)$ for
any $y \in Proj_A(x).$ Therefore, for a pair $(X,Y)$ distributed
jointly according to some $\pi \in \Pi_{S^\ast}(\mu),$ it follows from
the definition of the collection $\Pi_{S^\ast}(\mu)$ that
\begin{align*}
  E_\pi\left[ c(X,Y)\ | \  X \right] =
  \begin{cases}
    E_\pi\left[ c(X,A)\ | \  X \right] = c(X,A) \quad\quad&\text{ if } c(X,A) < 1/\lambda^\ast,\\
    E_\pi\left[ c(X,X)\ | \  X \right]  = 0 &\text{ if } c(X,A) > 1/\lambda^\ast,\\
    E_{\pi} \big[ I\left(Y \in Proj_A(X)\right)| \ X
    \big]c(X,A) &\text{ if } c(X,A) = 1/\lambda^\ast,
  \end{cases}
\end{align*}
almost surely. Then, as $c(\cdot,\cdot)$ is non-negative,
$E_\pi\big[c(X,Y)\big] = E_\pi\left[ E_\pi\left[ c(X,Y)\ | \ X
  \right]\right]$ satisfies,
\[ E_{\pi}\left[c(X,A); c(X,A) < \frac{1}{\lambda^\ast} \right] \leq
  E_\pi\big[c(X,Y)\big] \leq E_{\pi}\left[c(X,A); c(X,A) <
    \frac{1}{\lambda^\ast} \right] + E_{\pi} \left[c(X,A); c(X,A) =
    \frac{1}{\lambda^\ast} \right].\] Since the marginal distribution
of $X$ is $\mu$ (refer the definiton of $\Pi_{S^\ast}(\mu)$ above), it
follows that $\int cd\pi = E_\pi[c(X,Y)]$ satisfies
$\underline{c} \leq \int cd\pi \leq \overline{c},$ where
\begin{align}
  \underline{c} := \int_{\{ x: c(x,A) < \frac{1}{\lambda^\ast}\}}
  \hspace{-10pt}c(x,A) d\mu(x) 
  \quad \text{ and } \quad 
  \overline{c} := \int_{\{ x: c(x,A) \leq \frac{1}{\lambda^\ast} \}}
  \hspace{-10pt}c(x,A) d\mu(x).
\label{Defn-Cbars}
\end{align}
Further, as any optimal measure $\pi^\ast \in \Phi_{\mu,\delta}$
satisfying $\pi^\ast(S \times A) = I$ is a member of
$\Pi_{S^\ast}(\mu),$ it follows from the complementary slackness
condition \eqref{Comp-Slack-2} that $\int cd\pi^\ast$ has to equal
$\delta$ whenever $\lambda^\ast > 0,$ and consequently,
  $\underline{c} \leq \delta \leq \overline{c},$
  whenever a primal optimizer $\pi^\ast \in \Phi_{\mu,\delta}$
  satisfying $\pi^\ast(S \times A) = I$ exists. The observation that
  $\underline{c} \leq \delta \leq \overline{c}$ holds regardless of
  whether a primal optimizer exists or not, and this is the content of
  Lemma \ref{Lem-Optimality-lambda-Prob} below,  whose proof is
  provided towards the end of this section in Subsection
  \ref{Sec-Proof-Probab-Case}. 
  \begin{lemma}
    Suppose that Assumption (A1) is in force, and $A$ is a nonempty
    closed subset of the Polish space $S.$ In addition, suppose that
    $\lambda^\ast \in (0,\infty)$ attains the infimum in
    \eqref{Probability-1d-Problem}. Then,
    $\underline{c} \leq \delta \leq \overline{c}.$ On the other hand,
    if the infimum in \eqref{Probability-1d-Problem} is attained at
    $\lambda^\ast = 0,$ then
    $\delta \geq \overline{c} = \underline{c}.$
    \label{Lem-Optimality-lambda-Prob}
  \end{lemma}

  If $\lambda^\ast > 0$ and $\underline{c} = \overline{c} = \delta,$
  any coupling in $\Pi_{S^\ast}(\mu)$ is primal optimal (because it
  satisfies the complementary slackness conditions
  \eqref{Comp-Slack-1} and \eqref{Comp-Slack-2} in addition to the
  primal feasibility condition that $\pi^\ast \in
  \Phi_{\mu,\delta}$). In particular, one can describe a convenient
  optimal coupling $\pi^\ast \in \Pi_{S^\ast}(\mu)$ satisfying
  $I(\pi^\ast) := \pi^\ast(S \times A) = I$ as follows: First, sample
  $X$ with distribution $\mu,$ and let $\xi : S \rightarrow A$ be any
  universally measurable map such that $\xi (x) \in Proj_A(x),$
  $\mu$-almost surely. We then write,
\begin{align*}
  Y^\ast = \xi (X) \cdot I\left( c(X, A) \leq
  \frac{1}{\lambda^\ast}\right) + X \cdot I \left( c(X,A) >
  \frac{1}{\lambda^\ast}\right), 
\end{align*} 
to obtain $(X,Y^\ast)$ distributed according to $\pi^\ast.$ Such a
universally measurable map $\xi(\cdot) $ always exists, assuming that
$c(\cdot,\cdot) $ is lower semicontinuous and that
\textnormal{$\mathnormal{Proj}\xspace_{A}(x)$ is not
  empty\footnote{The assumption that
    $\mathnormal{Proj}\xspace_{A}(x) \neq \emptyset$ holds, for
    example, when $A$ is compact and nonempty (as $c$ is lower
    semi-continuous), or when $c(x,\cdot)$ has compact sub-level sets
    for each $x$ (recall that $A$ is a closed set in the discussion)}
  for $\mu$-almost every }$x$ (see, for example, Proposition 7.50(b)
of \citet{bertsekas1978stochastic}). In this case,
\begin{align*}
  I 
  = \pi^\ast(S \times A) = \pi^\ast \left\{(x,y) \in S \times S:
  c(x,A) \leq \frac{1}{\lambda^\ast}\right\}
  =\mu\left\{x \in S:c(x,A)\leq\frac{1}{\lambda^{\ast}}\right\}.  
\end{align*}
The second equality follows from the observation that for the
described coupling $(X,Y^\ast)$ distributed according to $\pi^\ast,$
we have, $Y^\ast \in A$ if and only if $c(X,A) \leq 1/\lambda^\ast,$
almost surely.

The reformulation that $I = \mu(x: c(x,A) \leq 1/\lambda^\ast)$ is
extremely useful, as it re-expresses the worst-case probability of
interest in terms of the probability of a suitably inflated
neighborhood, $\{x \in S: c(x,A) \leq 1/\lambda^\ast\},$ evaluated
under the reference measure $\mu,$ which is often tractable (see
Section \ref{Sec-App-LC} and \ref{Sec-App-FPP} for some
applications). However, as the reasoning that led to this
reformulation relies on the existence of a primal optimal transport
plan, we use $\epsilon-$optimal transport plans in Theorem
\ref{Thm-Prob-Reform} below to arrive at the same conclusion. The
proof of Theorem \ref{Thm-Prob-Reform} is presented towards the end of
this section in Subsection \ref{Sec-Proof-Probab-Case}.



\begin{theorem}
  Suppose that Assumption (A1) is in force, and $A$ is a nonempty
  closed subset of the Polish space $S.$ In addition, suppose that
  $\lambda^\ast \in [0,\infty)$ attains the infimum in
  \eqref{Probability-1d-Problem}, and the quantities
  $\underline{c},\overline{c},$ defined in \eqref{Defn-Cbars} in terms
  of $\lambda^\ast$, are such that $\underline{c} = \overline{c}.$
  Then,
  \begin{equation}
    \sup\big\{  P(A):\ d_{c}(\mu, P)\leq\delta\big\}  =\mu\left\{
      x:c(x,A)\leq\frac{1}{\lambda^{\ast}}\right\}. 
  \label{Set-Modified-Support}
\end{equation}
\label{Thm-Prob-Reform}
\end{theorem}

Recall from Lemma \ref{Lem-Optimality-lambda-Prob} that
$\lambda^\ast,$ if it is strictly positive, is such that
$\delta \in (\underline{c}, \overline{c}).$ To arrive at the
characterization \eqref{Set-Modified-Support}, we had assumed that the
integral
\begin{align}
  h(u) := \int_{\{x:c(x,A) \leq u\}} c(x,A)d\mu(x) 
\label{h(u)-Integral}
\end{align}
is continuous at $u = 1/\lambda^\ast,$ and consequently
$\underline{c} = \overline{c}.$ However, if this is not the case, for
example, because $\mu$ is atomic, then
$\underline{c} < \delta \leq \overline{c}.$ In this scenario, one can
identify an optimal coupling by randomizing between the extreme cases
$\underline{c}$ and $\overline{c}.$ Remark
\ref{Rem-Primal-Opt-Cbar-Not-equal} below provides a characterization
of the primal optimizer when $\underline{c} < \delta \leq
\overline{c}.$ 

\begin{remark}
  Suppose that $\lambda^\ast$ is such that
  $\underline{c} < \delta \leq \overline{c},$ and that all the
  assumptions in Theorem \ref{Thm-Prob-Reform}, except the condition
  that $\underline{c} = \overline{c} = \delta,$ are satisfied.  Let
  $Z$ be an independent Bernoulli random variable with success
  probability,
    $p := {(\delta - \underline{c})}/{(\overline{c} -
      \underline{c})}.$
    As before, sample $X$ with distribution $\mu$ and let
    $\xi : S \rightarrow A$ be any universally measurable map such
    that $\xi (x) \in Proj_A(x),$ $\mu$-almost surely. Then,
\begin{align*}
  Y^\ast &:= \xi (X) \cdot I\left( c(X, A) <
           \frac{1}{\lambda^\ast}\right) + X \cdot I \left( c(X,A) >
           \frac{1}{\lambda^\ast}\right) 
           + \left( Z\cdot\xi (X) + (1-Z) \cdot X
           \right) \cdot I\left( c(X, A) = 
           \frac{1}{\lambda^\ast}\right) 
\end{align*}
is such that $\Pr((X,Y^\ast) \in S^\ast) = 1$ (thus satisfying
complementary slackness condition \eqref{Comp-Slack-1}), and
\begin{align*}
  E\big[c(X,Y^\ast) \big] = \int_{\{c(x,A) < \frac{1}{\lambda^\ast}\}} \hspace{-20pt}
                            c(x,A) d\mu(x) \ +\  \Pr( Z
                            =1 ) \int_{\{ c(x,A) =
                            \frac{1}{\lambda^\ast}\}} \hspace{-20pt} c(x,A) d\mu(x) \
                            + \  0
                          =\underline{c} + p \big( \overline{c} -
                            \underline{c}\big) = \delta.  
\end{align*}
This verifies the complementary condition \eqref{Comp-Slack-2} as
well, and hence the marginal distribution of $Y^\ast$ in the coupling
$(X,Y^\ast)$ attains the supremum in
$\sup\{ P(A) : d_c(\mu, P) \leq \delta \}.$ 
  \label{Rem-Primal-Opt-Cbar-Not-equal}
\end{remark}

All the examples considered in this paper have the function $h(u)$ as
continuous, and hence we have the characterization
\eqref{Set-Modified-Support} which offers an useful reformulation for
computing worst-case probabilities.  For instance, let us take the
cost function $c$ as a distance metric $d$ (as in Wasserstein
distances), and let the baseline measure $\mu$ and the set $A$ be such
that $h(\cdot),$ defined in \eqref{h(u)-Integral}, is continuous.
Then \eqref{Set-Modified-Support} provides an interpretation that the
worst-case probability $P(A)$ in the neighborhood of $\mu$ is just the
same as $\mu(A_{1/\lambda^{\ast}}),$ where
$A_{\epsilon}:= \big\{y\in S:d(x,y)\leq\epsilon\text{ for some }x\in A
\big\}$ denotes the $\epsilon$-neighborhood of the set $A.$ Thus, the
problem of finding a probability measure with worst-case probability
$P(A)$ in the neighborhood of measure $\mu$ amounts to simply
searching for a suitably inflated neighborhood of the set $A$
itself. For example, if $S=\mathbb{R},\ c(x,y)=|x-y|,$ then 
\[
\sup\big\{P[a,\infty):d_{c}(\mu,P)\leq\delta\big\}=\mu\left[  a-\frac
{1}{\lambda^{\ast}},\infty\right),
\]
for any $\mu\in\mathcal{B}(\mathbb{R}),$ where $\lambda^\ast$ is the
solution to the univariate optimization problem in
\eqref{Probability-1d-Problem}. Alternatively, due to Lemma
\ref{Lem-Optimality-lambda-Prob}, 
one can characterize $1/\lambda^\ast$ as
$h^{-1}(\delta) := \inf\{u \geq 0: h(u) \geq \delta\},$ where $h(u)$
is the monotonically increasing right-continuous function (with left
limits) defined in \eqref{h(u)-Integral}.  If $h(u)$ does not admit a
closed-form expression, one approach is to obtain samples of $X$ under
the reference measure $\mu,$ and either solve the sampled version of
\eqref{Probability-1d-Problem}, or compute a Monte Carlo approximation
of the integral $h(u)$ to identify the level $1/\lambda^\ast$ as
$h^{-1}(\delta) = \inf\{u: h(u) \geq \delta\}.$

\subsubsection{Proofs of Lemma \ref{Lem-Optimality-lambda-Prob} and
  Theorem \ref{Thm-Prob-Reform}.}
\label{Sec-Proof-Probab-Case}
We conclude this section with proofs for Lemma
\ref{Lem-Optimality-lambda-Prob} and Theorem
\ref{Thm-Prob-Reform}. 
  Given $\lambda^\ast \geq 0$
  and $n > 1,$ define $C_n:= C_n^{(1)} \cup C_n^{(2)},$ where
\begin{align*}
  C_n^{(1)} := \left\{ (x,y) \in S \times S: c(x,A) \leq 
              \frac{1}{\lambda^\ast}\left(1+\frac{1}{n}\right),  \ y \in A,
              \ c(x,y) < c(x,A) + \frac{1}{\lambda^\ast n}\right\} 
              \text{ and }\\ 
  C_n^{(2)} := \left\{ (x,y) \in S \times S: c(x,A) > 
              \frac{1}{\lambda^\ast}\left( 1- \frac{1}{n}\right),  \ y
              \notin A, \ c(x,y) < \frac{1}{\lambda^\ast n}\right\}. 
\end{align*}
In addition, define 
  \begin{align*}
    D_n^{(1)} := \left\{ (x,y) \in C_n: c(x,A) \leq
    (1-1/n)/\lambda^\ast\right\}, \quad 
    D_n^{(2)} := \left\{ (x,y) \in C_n: c(x,A) > 
    (1+1/n)/\lambda^\ast\right\},  
  \end{align*}
  and
  $D_n^{(3)} := C_n \setminus \left(D_n^{(1)} \cup D_n^{(2)}\right).$
  The above definitions yield, $C_n^{(1)} = D_n^{(1)} = S \times A,$
  and $C_n^{(2)} = D_n^{(2)} = D_n^{(3)} = \emptyset,$ when
  $\lambda^\ast = 0.$ 

\begin{lemma} 
  Suppose that Assumption (A1) is in force, and $A$ is a nonempty
  closed subset of the Polish space $S.$ In addition, suppose that
  $\lambda^\ast \in [0,\infty)$ attains the infimum in
  \eqref{Probability-1d-Problem}. Then, there exists a collection of
  probability measures
  $\{\pi_n : n > 1\} \subseteq \Phi_{\mu,\delta}$ such that
  $\pi_n(C_n) \geq 1 - 1/n,$
  $\int_{(S \times S) \setminus C_n} cd\pi_n = 0,$ and
  $I(\pi_n) := \pi_n(S \times A) \geq I - 2/n.$
\label{Lem-Inter-Probs}
\end{lemma}

\begin{proof}
  \noindent \ \textit{Proof of Lemma \ref{Lem-Inter-Probs}.}  Since
  $I := \sup \{ \pi(S \times A): \pi \in \Phi_{\mu,\delta}\},$ we
  consider a collection
  $\{ \tilde{\pi}_n: n > 1\} \subseteq \Phi_{\mu,\delta}$ such that
  $I(\tilde{\pi}_n) \geq I - 1/n^2.$ For every $n > 1,$ if we let
  \[B_n := \left\{ (x,y) \in S \times S: \mathbf{1}_A(y) - \lambda^\ast
    c(x,y) > \big(1-\lambda^\ast c(x,A)\big)^+ - 1/n \right\},\] as a consequence of
  Markov's inequality and the characterization \eqref{Inter-Eps-Opt}
  in Remark \ref{Rem-Eps-Opt-Measure-Structure}, we obtain,
  \[ \tilde{\pi}_n(B_n) \geq 1 - \frac{1}{n} \int \big(
    \phi_{_{\lambda^\ast}}(x) - \mathbf{1}_A(y) - \lambda^\ast c(x,y)
    \big)d\tilde{\pi}_n(x,y) \geq 1 - \frac{1/n^2}{1/n} =
    1-\frac{1}{n}. \] For every $n > 1,$ given a pair
  $(\tilde{X}_n,\tilde{Y}_n)$ jointly distributed with law
  $\tilde{\pi}_n,$ we define a new jointly distributed pair
  $(X_n,Y_n)$ defined as follows,
  \begin{align*}
    (X_n, Y_n) :=
    \begin{cases}
      (\tilde{X}_n, \tilde{Y})_n \quad &\text{ if } (X_n,Y_n) \in B_n,\\
      (\tilde{X}_n, \tilde{X}_n) &\text{ otherwise.} 
    \end{cases}
  \end{align*}
  Our objective now is to show that the collection $\{\pi_n: n > 1\},$
  where $\pi_n := \text{Law}(X_n,Y_n),$ satisfies the desired
  properties.  We begin by verifying that
  $\pi_n \in \Phi_{\mu,\delta}$ here: As
  $\tilde{\pi}_n \in \Phi_{\mu,\delta}$ and $\tilde{X}_n := X_n,$ it
  is immediate that $\text{Law}(X_n) = \mu.$ In addition, as
  $\int cd\pi_n = \int_{B_n}cd\tilde{\pi}_n + 0,$ because
  $c(x,x) = 0,$ it follows from the non-negativity of $c(\cdot,\cdot)$
  and $\int cd\tilde{\pi}_n \leq \delta$ that
  $\pi_n \in \Phi_{\mu,\delta}.$ Next, as $B_n \subseteq C_n$ for
  every $n > 1,$ we have that
  $\pi_n(C_n) \geq \pi_n(B_n) \geq 1- 1/n,$ and
  $\int_{(S \times S) \setminus C_n} c(x,y)d\pi_n(x,y) = \int_{(S
    \times S) \setminus C_n} c(x,x) d\pi_n(x,y) = 0.$ Finally, for
  every $n \geq 1,$ $I(\pi_n) \geq I - 2/n$ is also immediate once we
  observe that
  \begin{align*}
    I(\pi_n) &= \pi_n(S \times A) \geq
               \pi_n((S \times A) \cap B_n) = \tilde{\pi}_n((S \times A) \cap B_n)\\
             &\geq \tilde{\pi}_n(S \times A) - \tilde{\pi}_n((S \times S)
               \setminus B_n) \geq I - \frac{1}{n^2} - \frac{1}{n},
  \end{align*}
  thus verifying all the desired poperties of the collection
  $\{\pi_n : n > 1\}.$ 
\end{proof}
\bigskip
\begin{proof}
  \noindent \ \textit{Proof of Lemma
    \ref{Lem-Optimality-lambda-Prob}.}  Consider a collection
  $\{\pi_n : n > 1\} \subseteq \Phi_{\mu,\delta}$ such that
  $\pi_n(C_n) \geq 1 - 1/n,$
  $\int_{(S \times S) \setminus C_n} cd\pi_n = 0,$ and
  $I(\pi_n) := \pi_n(S \times A) \geq I - 2/n.$ Such a collection
  exists because of Lemma \ref{Lem-Inter-Probs}.  We first observe
  that $\int cd\pi_n = \int_{C_n} cd\pi_n,$ because
  $\int_{(S \times S) \setminus C_n}cd\pi_n = 0.$ Next, recalling the
  definitions of subsets $D_n^{(i)}, i = 1,2,3$ introduced before
  stating Lemma \ref{Lem-Inter-Probs}, we use
  $D_n^{(i)} \subseteq C_n^{(i)},$ $i = 1,2,$ to observe that,  
  \begin{align*}
    c(x,A) \leq c(x,y) < c(x,A)+ 1/(n\lambda^\ast) \quad &\text{ if }
                                                           (x,y) \in D_n^{(1)},\\ 
    0 \leq c(x,y) < 1/(n\lambda^\ast) \quad &\text{ if }
                                              (x,y)  \in D_n^{(2)}, \text{ and } \\
    0 \leq c(x,y) < c(x,A) + 1/(n\lambda^\ast)
    \quad &\text{ if }
                                       (x,y) \in D_n^{(3)}.
  \end{align*}
  \textsc{Case 1: $\lambda^\ast > 0.$} Let us first restrict ourself to
  the case where $\lambda^\ast > 0.$ Then
  $\int cd\pi_n = \int_{C_n} cd\pi_n$ can be bounded from above and
  below as follows:
  \begin{align*}
    \int_{D_n^{(1)}} \hspace{-5pt} c(x,A)d\pi_n(x,y) \leq \int cd\pi_n \leq
    \int_{D_n^{(1)} \cup D_n^{(3)}} \hspace{-5pt} \left(c(x,A) + \frac{1}{n\lambda^\ast}  \right)  
    d\pi_n(x,y)  + 
    \frac{\pi_n(D_n^{(2)})}{n\lambda^\ast}. 
  \end{align*}
  Next, as $\lambda^\ast > 0$ and
  $I(\pi_n) = \pi_n(S \times A) \geq I - 2/n,$ we use the second part 
  of \eqref{Inter-Eps-Opt} to reason that 
  \begin{align}
    \label{Inter-Probab-LB}    
    \int_{D_n^{(1)}} \hspace{-5pt} c(x,A)d\pi_n(x,y) &\leq \int cd\pi_n
    \leq \delta,  \text{ and }\\
    \delta - \frac{1}{2n\lambda^\ast} \leq &\int cd\pi_n \leq 
    \int_{D_n^{(1)} \cup D_n^{(3)}} \hspace{-5pt} \left(c(x,A) + \frac{1}{n\lambda^\ast}  \right)  
    d\pi_n(x,y)  + 
    \frac{\pi_n(D_n^{(2)})}{n\lambda^\ast}, 
    \label{Inter-Probab-UB}    
  \end{align}
  for every $n > 1.$ The next few steps are dedicated towards
  re-expressing the integrals in the left hand side of
  \eqref{Inter-Probab-LB} and right hand side of
  \eqref{Inter-Probab-UB} in terms of $\underline{c}$ and
  $\overline{c}$ to obtain,
  \begin{align}
    \varlimsup_{n \rightarrow \infty} \int_{D_n^{(1)} \cup D_n^{(3)}} \hspace{-5pt}
    \left(c(x,A) + \frac{1}{n\lambda^\ast}  \right) d\pi_n(x,y) \leq
    \overline{c} \quad \text{ and } \quad     \underline{c} \ \leq \ 
    \varliminf_{n \rightarrow \infty} \int_{D_n^{(1)}} c(x,A)d\pi_n(x,y).
    \label{Lim-Simplifying-Integrals}
  \end{align}
  As every $(x,y) \in D_n^{(1)} \cup D_n^{(3)}$ satisfies
  $c(x,A) \leq (1+1/n)/\lambda^\ast,$ we have
  \begin{align*}
    \int_{D_n^{(1)} \cup D_n^{(3)}} \hspace{-5pt} \left(c(x,A) + \frac{1}{n\lambda^\ast}  \right)  
    d\pi_n(x,y) \leq \int_{\{x:c(x,A) \leq (1+1/n)/\lambda^\ast\}}
    \hspace{-40pt} c(x,A)d\pi_n(x,y) = \int_{\{x:c(x,A) \leq (1+1/n)/\lambda^\ast\}}
    \hspace{-40pt} c(x,A)d\mu(x), 
  \end{align*}
  thus establishing the first inequality in
  \eqref{Lim-Simplifying-Integrals} as a consequence of bounded
  convergence theorem.  Next, as
  $\{(x,y) \notin D_n^{(1)}: c(x,A) < 1/\lambda^\ast\}$ is contained
  in the union of $(S \times S) \setminus C_n$ and
  $\{(x,y) \in C_n: (1-1/n)/\lambda^\ast < c(x,A) < 1/\lambda^\ast\},$
  we obtain,
  \begin{align*}
    \underline{c} &:= \int_{\{x: c(x,A) < 1/\lambda^\ast\}}
    \hspace{-30pt} c(x,A)d\pi_n(x,y) \leq
    \int_{D_n^{(1)}}c(x,A)d\pi_n(x,y) + \frac{1}{\lambda^\ast}\pi_n\big((x,y) \notin
    D_n^{(1)}: c(x,A) < 1/\lambda^\ast \big)\\ 
    &\leq \int_{D_n^{(1)}}c(x,A)d\pi_n(x,y) + \frac{1}{\lambda^\ast}
      \left( \pi_n \big((x,y) \in C_n: (1-1/n)/\lambda^\ast \leq c(x,A) <
      1/\lambda^\ast\big) + \pi_n\big( (S \times S) \setminus
      C_n\big)\right)\\
    &\leq \int_{D_n^{(1)}}c(x,A)d\pi_n(x,y) + \frac{1}{\lambda^\ast}
      \left( \mu\big(x \in S: (1-1/n)/\lambda^\ast \leq c(x,A) <
      1/\lambda^\ast\big) + \frac{1}{n}\right),
  \end{align*}
  thus verifying the second inequality in
  \eqref{Lim-Simplifying-Integrals}.  Now that we have verified both
  the inequalities in \eqref{Lim-Simplifying-Integrals}, the
  conclusion that $\underline{c} \leq \delta \leq \overline{c}$ is
  automatic if we send $n \rightarrow \infty$ in
  \eqref{Inter-Probab-LB}, \eqref{Inter-Probab-UB} and use
  \eqref{Lim-Simplifying-Integrals}. 

  \textsc{Case 2: When $\lambda^\ast = 0.$} It follows from the
  definitions of sets $C_n^{(i)}$ and $D_n^{(i)}$ that
  $C_n^{(1)} = D_n^{(1)} = S \times A,$ and
  $C_n^{(2)} = D_n^{(2)} = D_n^{(3)} = \emptyset$ when
  $\lambda^\ast = 0.$ Recall that $C_n := C_n^{(1)} \cup C_n^{(2)}$
  and $\pi_n \in \Phi_{\mu,\delta}$ are such that
  $\pi_n(C_n) = \pi_n(S \times A) \geq 1-1/n.$ Let
  $h^-(u):= \int_{\{x:c(x,A) < u\}}c(x,A)d\mu(x)$ and
  $u_n := \inf\{u \geq 0: \mu(x:c(x,A) \leq u) \geq 1-1/n\}.$ As
  $\pi_n(S \times A) \geq 1-1/n$ and $c(x,y) \geq c(x,A)$ for any
  $x \in A,$ we have $\int cd\pi_n \geq \int_{S \times
    A}c(x,A)d\pi_n(x,y) \geq h^-(u_n).$ Next, as
  $\mu(x: c(x,A) \leq \sup_n u_n) = 1,$ we have
  \begin{align}
    \varliminf_{n \rightarrow \infty} \int cd\pi_n \geq \varliminf_{n
    \rightarrow \infty} h^-(u_n) = \lim_{n \rightarrow \infty}
    \int_{\{x: c(x,A) < u_n\}} c(x,A) (x)d\mu(x) = \int c(x,A)d\mu(x),    
    \label{Inter-lambda-Eq-0}
  \end{align}
  as a consequence of monotone convergence theorem. Further, as
  $\int cd\pi_n \leq \delta$ for every $n,$ it follows from
  \eqref{Inter-lambda-Eq-0} that
  $\overline{c} = \underline{c} = \int c(x,A)d\mu(x) \leq \delta.$
  This concludes the proof. 
\end{proof}
\bigskip

\begin{proof}
  \noindent \ \textit{Proof of Theorem \ref{Thm-Prob-Reform}.}
  Consider a collection
  $\{\pi_n : n > 1\} \subseteq \Phi_{\mu,\delta}$ such that
  $\pi_n(C_n) \geq 1 - 1/n,$
  $\int_{(S \times S) \setminus C_n} cd\pi_n = 0,$ and
  $I(\pi_n) := \pi_n(S \times A) \geq I - 2/n.$ Such a collection
  exists because of Lemma \ref{Lem-Inter-Probs}. Similar to the proof
  of Lemma \ref{Lem-Optimality-lambda-Prob}, we recall the definitions
  of subsets $D_n^{(i)}, i = 1,2,3$ introduced before stating Lemma
  \ref{Lem-Inter-Probs}, to observe that
  $D_n^{(i)} \subseteq C_n^{(i)},$ $i = 1,2,$ and consequently, 
  \begin{align*}
    \mathbf{1}_A(y) = 1 \text{ if }   (x,y) \in D_n^{(1)}, \quad
    \text{ and } \mathbf{1}_A (y) = 0 \text{ if } (x,y)  \in D_n^{(2)}.
  \end{align*}
  As a result,
  $\pi_n\left( D_n^{(1)}\right) \leq \pi_n(S \times A) \leq \pi_n
  \left( (S \times S) \setminus D_n^{(2)}\right).$ Combining this with the
  observation that $I - 2/n \leq \pi_n(S \times A) \leq I,$ we obtain,  
  \begin{align}
    \pi_n\left( D_n^{(1)}\right) \leq I \quad \text{ and } \quad  \pi_n
  \left((S \times S) \setminus D_n^{(2)} \right) \geq I - \frac{2}{n},
    \label{Inter-Thm-Prob-1}
  \end{align}
  for every $n > 1.$ Similar to the proof of Lemma
  \ref{Lem-Optimality-lambda-Prob}, the rest of the proof is dedicated
  towards proving
  \begin{align}
    \varlimsup_n \pi_n\left( (S \times S) \setminus D_n^{(2)} \right) \ \leq
    \ \mu\big( x \in S:
    c(x,A) \leq 1/\lambda^\ast\big) \  \leq \ \varliminf_n \pi_n
  \left( D_n^{(1)} \right),
    \label{Inter-Thm-Prob-2}
  \end{align}
  when $\underline{c} = \overline{c}.$ The first inequality is
  immediate when we observe that every $(x,y) \notin D_n^{(2)}$ is
  such that $c(x,A) \leq (1+1/n)/\lambda^\ast,$ and therefore,
  \begin{align*}
    \varlimsup_n \pi_n \left( (S \times S) \setminus D_n^{(2)} \right) \leq \ 
    \varlimsup_n \pi_n \big( (x,y)
    : c(x,A) \leq (1+1/n)/\lambda^\ast\big) = \lim_n \mu  \big( x
    : c(x,A) \leq (1+1/n)/\lambda^\ast\big), 
  \end{align*}
  which equals $\mu(x:c(x,A) \leq 1/\lambda^\ast),$ due to bounded 
  convergence theorem. This verifies the first inequality in
  \eqref{Inter-Thm-Prob-1}.  Next, as
  $D_n^{(1)} := \{(x,y) \in C_n: c(x,A) \leq (1-1/n)/\lambda^\ast\},$
  the probability
  $\pi_n \left( \{(x,y): c(x,A) \leq 1/\lambda^\ast\}\right)$ equals,
  \begin{align*}
    \pi_n\left( D_n^{(1)} \right)  + \pi_n \left( \big\{(x,y)
    \in C_n: (1-1/n)/\lambda^\ast <  c(x,A) \leq
    1/\lambda^\ast\big\}\right) + \pi_n \left( \big\{ (x,y) \notin
    C_n: c(x,A) \leq 1/\lambda^\ast \big\}\right).   
  \end{align*}
  Since $\pi_n(C_n) \geq 1-1/n$ and
  $\mu(x: c(x,A) \leq 1/\lambda^\ast) = \pi_n \left( \{(x,y): c(x,A)
    \leq 1/\lambda^\ast\}\right),$ we obtain that
  \begin{align*}
    \pi_n\left(D_n^{(1)} \right) \geq \ \mu \big(x: c(x,A) \leq
    1/\lambda^\ast \big) - \mu \left( \big\{x
    : (1-1/n)/\lambda^\ast <  c(x,A) \leq
    1/\lambda^\ast\big\}\right) - 1/n. 
  \end{align*}
  Further, as $\underline{c} = \overline{c},$ we have
  $\mu(x: c(x,y) = 1/\lambda^\ast) = 0,$ and therefore,
  $ \varliminf_n \pi_n\left(D_n^{(1)} \right) \geq \ \mu \big(x:
  c(x,A) \leq 1/\lambda^\ast \big),$ thus verifying the inequality in
  the right hand side of \eqref{Inter-Thm-Prob-2}. Now that both the
  inequalities in \eqref{Inter-Thm-Prob-2} are verified, we combine
  \eqref{Inter-Thm-Prob-1} with \eqref{Inter-Thm-Prob-2} to obtain
  that $I = \mu(x \in S: c(x,A) \leq 1/\lambda^\ast).$ 
\end{proof}



\section{Computing ruin probabilities: A first example.}
\label{Sec-App-LC}
In this section, we consider an example with the objective of
computing a worst-case estimate of ruin probabilities in a ruin model
whose underlying probability measure is not completely specified. As
mentioned in the Introduction, such a scenario can arise for various
reasons, including the lack of data necessary to pin down a
satisfactory model. A useful example to keep in mind would be the
problem of pricing an exotic insurance contract (for example, a
contract covering extreme climate events) in a way the probability of
ruin is kept below a tolerance level.
\begin{example}
  \textnormal{ In this example, we consider the celebrated
    Cramer-Lundberg model, where a compound Poisson process is used to
    calculate ruin probabilities of an insurance risk/reserve process.
    The model is specified by 4
    primitives: initial reserve $u,$ safety loading $\eta > 0,$ the
    rate $\nu$ at which claims arrive, and the distribution of
    claim sizes $F$ with first and second moments $m_1$ and $m_2$
    respectively. Then the stochastic process
    \begin{align*}
      R(t) = u + (1 + \eta) \nu m_1 t - \sum_{i=1}^{N_t} X_i,
    \end{align*}
    specifies a model for the reserve available at time $t.$ Here,
    $X_n$ denotes the size of the $n$-th claim, and the collection
    $\{X_1,X_2,\ldots\}$ is assumed to be independent samples from the
    distribution $F.$ Further, $\tilde{p} := (1+\eta) \nu m_1$ is the
    rate at which a premium is received, and $N_t$ is taken to be a
    Poisson process with rate $\lambda.$ Then one of the important
    problems in risk theory is to calculate the probability
    \[\psi (u,T) := \Pr\left\{\inf_{t \in [0,T]} R(t) \leq 0 \right\}\]
    that the insurance firm runs into bankruptcy before a specified
    duration $T.$}

  \textnormal{Despite the simplicity of the model, existing results in
    the literature do not admit simple methods for the computation of
    $\psi(u,T)$ (see \citet{MR2766220}, \citet{MR1458613},
    \citet{MR1680267} and references therein for a comprehensive
    collection of results). In addition, if the historical data is not
    adequately available to choose an appropriate distribution for
    claim sizes, as is typically the case in an exotic insurance
    situation, it is not uncommon to use a diffusion approximation
    \begin{align*}
      R_B(t) &:= u + (1+\eta)\nu m_1 t - (\nu m_1 t + \sqrt{\nu m_2}
      B(t))\\\\
      &\ = u + \eta \nu m_1 t - \sqrt{\nu m_2} B(t)
    \end{align*}
    that depends only on first and second moments of the claim size
    distribution $F.$ Here, the stochastic process
    $(B(t): 0 \leq t \leq T)$ denotes the standard Brownian
    motion, and 
        \begin{align*}
      \psi_B(u,T) := 
      \Pr \left\{ \sup_{t \in [0,T]} \big( \sqrt{\nu m_2} B(t) -
      \eta \nu m_1 t \big) \geq u \right\},
      \end{align*}
      is to serve as a diffusion approximation based substitute for
      ruin probabilities $\psi(u,T).$ See the seminal works of
      Iglehart and Harrison (\citet{Iglehart1969, HARRISON197767}) for
      a justification and some early applications of diffusion
      approximations in computing insurance ruin. Such diffusion
      approximations have enabled approximate computations of various
      path-dependent quantities, which may be otherwise
      intractable. However, as it is difficult to verify the accuracy
      of the Brownian approximation of ruin probabilities, in this
      example, we use the framework developed in Section
      \ref{Sec-Gen-Results} to compute worst-case estimates of ruin
      probabilities over all probability measures in a neighborhood
      around the baseline Brownian motion $B(t)$ driving the ruin
      model $R_B(t).$}

    \textnormal{For this purpose, we identify the Polish space where
      the stochastic processes of our interest live as the Skorokhod
      space $S = D([0,T], \mathbb{R})$ equipped with the $J_1$
      topology. In other words, $S = D([0,T], \mathbb{R})$ is simply
      the space of real-valued right-continuous functions with left
      limits (rcll) defined on the interval $[0,T],$ equipped with the
      $J_1$-metric $d_{_{J_1}}.$ Please refer Lemma
      \ref{Lem-Ruin-Prob-BM-App} in Appendix \ref{Appendix-Proofs} for
      an expression of $d_{_{J_1}}$ and Chapter 3 of \citet{MR1876437}
      for an excellent exposition on the space $D([0,T], \mathbb{R})$.
      Next, we take the transportation cost (corresponding to $p$-th
      order Wasserstein distance) as
      \[c(x,y) = \left(d_{_{J_1}}(x,y)\right)^p \quad\quad x, y \in
      S,\]
      for some $p \geq 1,$ and the baseline measure $\mu$ as the
      probability measure induced in the path space by the Brownian
      motion $B(t).$ In addition, if we let
      \[A_u := \bigg\{x \in S: \sup_{t \in [0,T]} \big( \sqrt{\nu m_2}
        x(t) - \nu m_1 t \big) \geq u \bigg\},\] then the following
      observations are in order: First, the Brownian approximation for
      ruin probability $\psi_B(u,T)$ equals $\mu(A_u).$ Next, the set
      $A_u$ is closed (verified in Lemma \ref{Lem-Au-Closed} in
      Appendix \ref{Appendix-Proofs}), and hence the function
      $\mathbf{1}_{A_u}(\cdot)$ is upper semicontinuous. Further, it
      is verified in Lemma \ref{Lem-Ruin-Prob-BM-App} in Appendix
      \ref{Appendix-Proofs} that
      \begin{align}
        c(x,A_u) := \inf\{ c(x,y): y \in A_u\} = \left(u - \sup_{t \in
        [0,T]} \big(\sqrt{\nu m_2} x(t) - 
        \eta \nu m_1 t\big) \right)^p, \quad \text{ for } x \notin A_u.
        \label{Proj-Ruin-Prob}
        \end{align}
        As $h(s) = E_{\mu}[c(X,A); c(X,A_u) \leq s]$ is continuous,
        for any given $\delta > 0,$ it follows from Theorem
        \ref{Thm-Prob-Reform} that 
       \begin{align}
         \psirob(u,T) :=\sup \{ P(A_u): d_c(P, \mu) \leq \delta\} =
         \mu\left\{ x \in S: c(x,A_u) \leq \frac{1}{\lambda^\ast}
         \right\},  
         \label{Inter-Ruin-Prob}
      \end{align}
      where, the level $1/\lambda^\ast$ is identified as
      $h^{-1}(\delta) = \inf\{ s \geq 0: h(s) \geq \delta \}.$ 
    }
    \textnormal{ 
      As $c(x,A_u)$ admits a simple form as in \eqref{Proj-Ruin-Prob},
      following \eqref{Inter-Ruin-Prob}, the problem of computing
      $\psirob(u,T)$ becomes as elementary as
      \begin{align}
        \psirob(u,T) = \ \Pr\left\{ \sup_{t \in [0,T]} \big( \sqrt{\nu
                       m_2} B_t - \eta \nu m_1 t \big) > u -
                       \left(\frac{1}{\lambda^\ast}\right)^{1/p}\right\}
                        = \psi_B(\tilde{u},T),
        \label{Ruin-Prob-1dim-shift}
      \end{align}
      with $\tilde{u} := u - (\lambda^\ast)^{-1/p}.$ Thus, the
      computation of worst-case ruin probability remains the problem
      of evaluating the probability that a Brownian motion with
      negative drift crosses a positive level, $\tilde{u},$ smaller
      than the original level $u.$ In other words, the presence of
      model ambiguity has manifested itself only in reducing the
      initial capital to a new level $\tilde{u}.$ Apart from the
      tractability, such interpretations of the worst-case ruin
      probabilities in terms of level crossing a modified ruin set is
      an attractive feature of using Wasserstein distances to model
      distributional ambiguities.}

     \noindent
    \textnormal{\textbf{Numerical illustration:} To make the
      discussion concrete, we consider a numerical example where
      $T = 100, \nu = 1, p = 2,$ and the safety loading $\eta = 0.1.$
      The claim sizes are taken from a distribution $F$ that is not
      known to the entire estimation procedure. Our objective is to
      compute the Brownian approximation to the ruin probability
      $\psi_B(u,T) = \mu(A_u),$ and its worst-case upper bound
        $\psirob(u,T) := \sup \big\{ P(A_u) : d_c(\mu,P) \leq
        \delta\big\}.$ 
        The estimation methodology is data-driven in the following
        sense: we derive the estimates of moments $m_1, m_2$ and
        $\delta$ from the observed realizations $\{X_1,\ldots,X_N\}$
        of claim sizes. While obtaining estimates for moments $m_1$
        and $m_2$ are straightforward, to compute $\delta,$ we use the
        claim size samples $\{ X_1,\ldots,X_N\}$ to embed realizations
        of compound Poisson risk process in a Brownian motion, as in
        \citet{khoshnevisan1993}.  Specific details of the algorithm
        that estimates $\delta$ can be found in Appendix
        \ref{Appendix-Embeddings}. The estimate of $\delta$ obtained
        is such that the compensated compound Poisson process of
        interest that models risk in our setup lies within the
        $\delta$-neighborhood $\{P: d_c(\mu,P) \leq \delta\}$ with
        high probability.  Next, given the knowledge of Lagrange
        multiplier $\lambda^\ast,$ the evaluation of $\psi_B(u,T)$ and
        $\psirob(u,T) = \psi_B(\tilde{u},T)$ (with
        $\tilde{u} = u - 1/\lambda^\ast)$ are straightforward, because
        of the closed-form expressions for level crossing
        probabilities of Brownian motion. Computation of $\lambda^*$
        is also elementary, and can be accomplished in multiple
        ways. We resort to an elementary sample average approximation
        scheme that solves for $\lambda^*.$ The estimates of Brownian
        approximation to ruin probability $\psi_B(u,T)$ and the
        worst-case ruin probability $\psirob(u,T),$ for various values
        of $u,$ are plotted in Table \ref{Table-Ruinprob-BMapprox}. To
        facilitate comparison, we have used a large deviation
        approximation of the ruin probabilities $\psild(u,T)$ that
        satisfy $\psild(u,T) \sim \psi(u,T),$ as
        $u \rightarrow \infty,$ as a common denominator to compare
        magnitudes of the Brownian estimate $\psi_B(u,T)$ and the
        worst-case bound $\psirob(u,T).$ In particular, we have drawn
        samples of claim sizes from the Pareto distribution
        $1-F(x) = 1 \wedge x^{-2.2}.$ While the Brownian approximation
        $\psi_B(u,T)$ remarkably underestimates the ruin probability
        $\psi(u,T)$, the worst-case upper bound $\psirob(u,T)$ appears
        to yield estimates that are conservative, yet of the correct
        order of magnitude.}
\begin{table} [htb!]
  \centering
  \caption{Comparison of ruin probability estimate $\psi_B(u,T)$ and
    the worst-case upper bound $\psirob(u,T)$: Example
    \ref{Eg-Ruin-prob-BM-approx}. The denominator $\psild(u,T)$ is such
    that $\psild(u,t)/\psi(u,T) \rightarrow 1,$ as $u \rightarrow \infty.$} 
  \begin{tabular}[h!]{ | c | c | c | c | c | c|}
    \hline
            & & & & & \\
    $u$ & $\psi_B(u,T)$ & $\tilde{u}$ & $\psirob(u,T)$ &
                                                         $\frac{\psi_B(u,T)}{\psild(u,T)}$
 & $\frac{\psirob(u,T)}{\psild(u,T)}$\\ 
        & & & & & \\
    \hline 
            & & & & & \\
    $50$ & $5.18\times 10^{-2}$ & $28.23$ & $0.2294$ & $3.32$  & $14.71$\\
    $100$ & $4.26 \times 10^{-4}$ & $50.78$ & $4.88 \times 10^{-2}$ &
                                                                  $1.07
                                                                  \times
                                                                  10^{-1}$
 & $12.28$\\
    $150$ & $4.36 \times 10^{-7}$ & $62.76$ & $1.84 \times 10^{-2}$ &
                                                                  $2.52
                                                                  \times
                                                                  10^{-4}$&
                                                                            $10.65$\\
    $200$ & $5.05 \times 10^{-11}$ & $69.40$ & $1.02 \times 10^{-2}$ & $5.35 \times 10^{-8}$&$10.80$\\
    $250$ & $6.75 \times 10^{-16}$ & $74.29$ & $6.50 \times 10^{-3}$& $1.15 \times 10^{-12}$&$10.98$\\
\hline
 \end{tabular}
\label{Table-Ruinprob-BMapprox}
 \end{table}
 \textnormal{\newline For further discussion on this toy example, let
   us say that an insurer paying for the claims distributed according
   to $X_i$ has a modest objective of keeping the probability of ruin
   before time $T$ at a level below 0.01. The various combinations of
   initial capital $u$ and safety loading factors $\eta$ that achieve
   this objective are shown in Figure
   \ref{Fig-Loading-vs-Capital}. While the $(\eta,u)$ combinations
   that work for the Brownian approximation model is drawn in red, the
   corresponding $(\eta, u)$ combinations that keep
   $\psirob(u,T) \leq 0.01$ is shown in blue.}
 \begin{figure}[h!]
   \centering
   \includegraphics[scale=0.3]{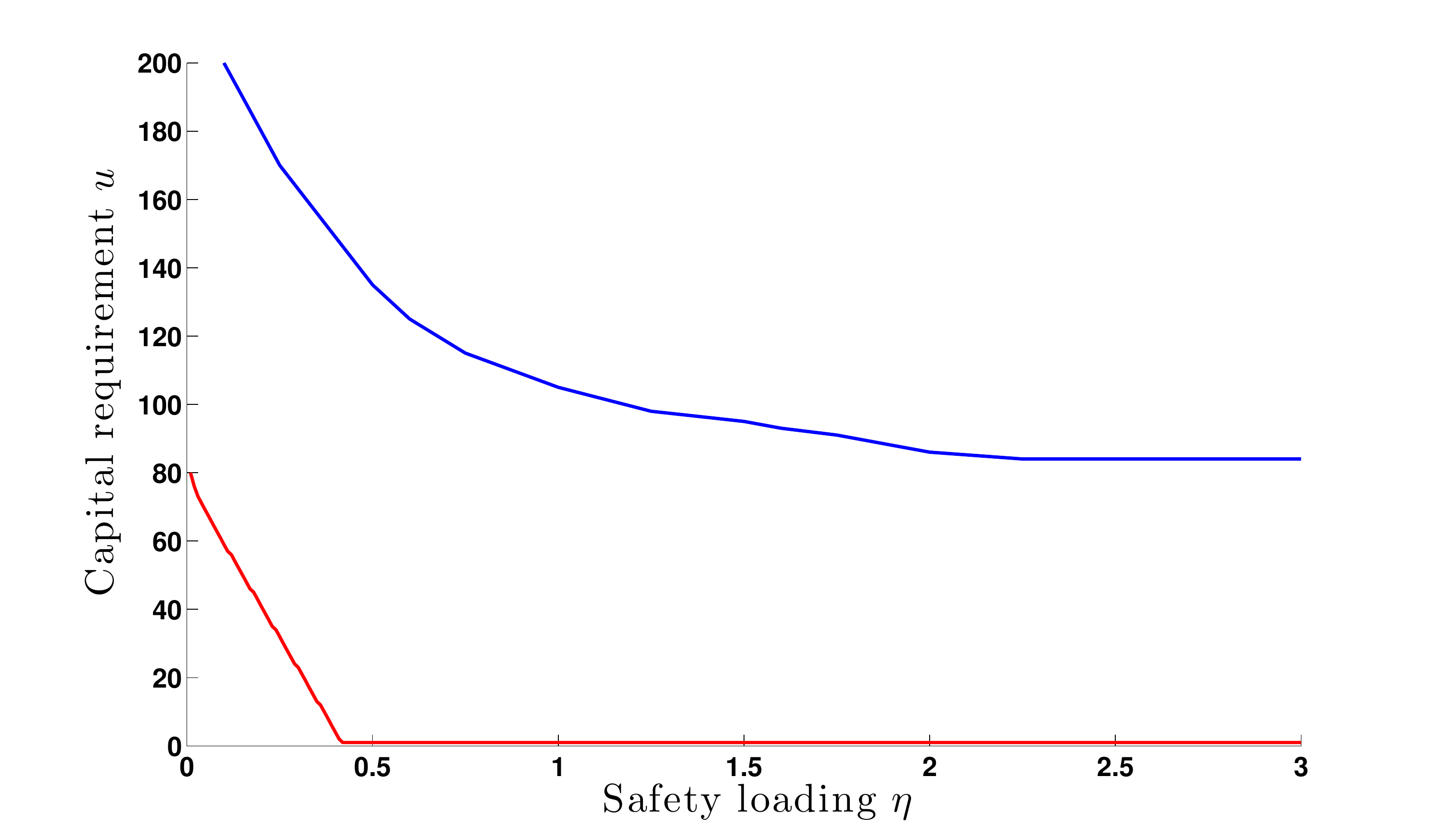}
   \caption{ Safety loading vs capital requirement for the Brownian
     model $R_B(t)$ (in red) and its robust counterpart (in blue) in
     Example \ref{Eg-Ruin-prob-BM-approx}. The objective is to keep
     the probability of ruin below 0.01.}
\label{Fig-Loading-vs-Capital}
 \end{figure}
 \textnormal{ \newline \newline \textbf{A discussion on regulatory
     capital requirement.}  For the safety loading $\eta = 0.1$ we
   have considered, the Brownian approximation model $R_B(t)$ requires
   that the initial capital $u$ be at least $60$ to achieve
   $\psi_B(u,T) \leq 0.01$. On the other hand, to keep the robust ruin
   probability estimates $\psirob(u,T)$ below $0.01,$ it is required
   that $u \approx 200,$ roughly 3 times more than the capital
   requirement of the Brownian model. From the insurer's point of
   view, due to the model uncertainty the contract is faced with, it
   is not uncommon to increase the premium income (or) initial capital
   by a factor of 3 or 4 (referred as ``hysteria factor'', see
   \citet{embrechts1998living}, \citet{Lewis2007}). This increase in
   premium can be thought of as a guess for the price for statistical
   uncertainty. Choosing higher premium and capital (larger $\eta$ and
   $u$), along the blue curve in Figure \ref{Fig-Loading-vs-Capital},
   as dictated by the robust estimates of ruin probabilities
   $\psirob(u,T),$ instead provides a mathematically sound way of
   doing the same. Unlike the Brownian approximation model, the blue
   curve demonstrates that one cannot decrease the probability of ruin
   by arbitrarily increasing the premium alone, the initial capital
   also has to be sufficiently large. This prescription is consistent
   with the behaviour of markets with heavy-tailed claims where, in
   the absence of initial capital, a few initial large claims that
   occur before enough premium income gets accrued are enough to cause
   ruin. The minimum capital requirement $u$ prescribed according to
   the worst-case estimates $\psirob(u,T)$ can be thought of as the
   regulatory minimum capital requirement.}
\label{Eg-Ruin-prob-BM-approx}
\end{example}

\begin{remark}
  \textnormal{ Defining the candidate set of ambiguous probability
    measures via KL-divergence (or) other likelihood based divergence
    has been the most common approach while studying distributional
    robustness (see, for example,
    \citet{Hans_Sarg,Ben_Tal,MAFI:MAFI12050}). This would not have
    been useful in this setup, as the compensated Poisson process that
    models risk is not absolutely continuous with respect to Brownian
    motion. In general, it is normal to run into absolute continuity
    issues when we deal with continuous time stochastic processes. In
    such instances, as demonstrated in Example
    \ref{Eg-Ruin-prob-BM-approx}, modeling via optimal transport costs
    (or) Wasserstein distances not only offers a tractable alternative,
    but also yields insightful equivalent reformulations.}
\label{Rem-Abs-Continuity}
\end{remark}

\section{Proof of the duality theorem.}
\label{Sec-Duality-Proof}
\subsection{Duality in compact spaces.}
\label{Sec-Duality-Compact}
We prove Theorem \ref{THM-STRONG-DUALITY} by first proving the
following progressively strong duality results in Polish spaces $S$
that are compact. Proofs of all the technicalities that are not
central to the argument are relegated to Appendix
\ref{Appendix-Proofs}.

\begin{proposition}
  \label{PROP-DUALITY-COMPACT-1}
  Suppose that $S$ is a compact Polish space, and
  $f: S \rightarrow \mathbb{R}$ satisfies Assumption (A2). In addition
  to satisfying Assumption (A1), suppose that
  $c: S\times S \rightarrow \mathbb{R}_+$ is continuous. Then $I = J,$
  and a primal optimizer $\pi^* \in \Phi_{\mu,\delta}$ satisfying
  $I(\pi^*) = I$ exists.
\end{proposition}

\begin{proposition}
  \label{PROP-DUALITY-COMPACT-2}
  Suppose that $S$ is a compact Polish space,
  $f: S \rightarrow \mathbb{R}$ satisfies Assumption (A2), and
  $c: S\times S \rightarrow \mathbb{R}_+$ satisfies Assumption (A1).
  Then $I = J,$ and a primal optimizer $\pi^* \in \Phi_{\mu,\delta}$
  satisfying $I(\pi^*) = I$ exists.
\end{proposition}


\noindent
As in the proof of the standard Kantorovich duality in compact spaces
in \citet{villani2003topics}, we use Fenchel duality theorem (see
Theorem 1 in Chapter 7, \citet{luenberger1997optimization}) to prove
Proposition \ref{PROP-DUALITY-COMPACT-1}. However, the similarity
stops there, owing to the reason that the set of primal feasible
measures $\Phi_{\mu,\delta}$ is not tight when $S$ is non-compact
(contrast this with Kantorovich duality, where the collection of
feasible measures $\Pi(\mu,\nu),$ the set of all joint distributions
between any two probability measures $\mu$ and $\nu,$ is tight).

\subsubsection*{Proof of Proposition \ref{PROP-DUALITY-COMPACT-1}.}
As a preparation for applying Fenchel duality theorem, we first let
$X = C_b(S \times S),$ and identify $X^\ast = M(S \times S)$ as its
topological dual. Here, the spaces $X$ and $X^\ast,$ representing the
vector space of bounded continuous functions and finite Borel measures
on $S \times S,$ respectively, are equipped with the supremum and
total variation norms. The fact that $X^\ast$ is the dual space of $X$
is a consequence of Riesz representation theorem (see, for example,
\citet{rudin1986real}). Next, let $C$ be the collection of all
functions $g$ in $X$ that are of the form
\begin{align*}
  g(x,y) = \phi(x) + \lambda c(x,y), \quad \text{ for all } x,y,
\end{align*}
for some $\phi \in C_b(S)$ and $\lambda \geq 0.$ In addition, let $D$
denote the collection of all functions $g$ in $X$ that satisfy
$g(x,y) \geq f(y),$ for every $x,y.$ Every $g$ in the convex subset
$C$ is defined by the pair $(\lambda, \phi),$ which, in turn, can be
uniquely identified by,
\begin{align*}
  \phi(x) = g(x,x) \quad  \text{ and } \quad  \lambda =
  \frac{g(x,y)-\phi(x)}{c(x,y)}, 
\end{align*}
for some $x,y$ in $S$ such that $c(x,y) \neq 0.$ Keeping this
invertible relationship in mind, define the functionals
$\Phi: C \rightarrow \mathbb{R}$ and
$\Gamma: D \rightarrow \mathbb{R}$ as below:
\begin{align*}
  \Phi(g) := \lambda \delta + \int \phi d\mu \quad \text{ and } \quad
  \Gamma(g) := 0. 
\end{align*}
Evidently, the functional $\Phi$ is convex, $\Gamma$ is concave, and
we are interested in
\begin{align}
  \inf_{g \in C \cap D} \big( \Phi(g) - \Gamma(g) \big) =
  \inf \big\{ J(\lambda, \phi): \lambda \geq 0, \ \phi
  \in C_b(S),\ \phi(x) + \lambda c(x,y) \geq f(y) \text{ for all } x,y
  \big\}.  
\label{RHS-Compact-Duality}
\end{align}
The next task is to identify the conjugate functionals
$\Phi^\ast, \Gamma^\ast$ and their respective domains $C^\ast,D^\ast$
as below:
\begin{align*}
  C^\ast = \left\{ \pi \in X^\ast: \sup_{g \in C} \left\{  \int g d\pi
  - \Phi(g)\right\} < \infty \right\} \text{ and } D^\ast = \left\{
  \pi \in X^\ast: \inf_{g \in D}  \int g d\pi > -\infty \right\}.
\end{align*}
The conjugate functionals $\Phi^\ast:C^\ast \rightarrow \mathbb{R}$
and $\Gamma^\ast: D^\ast \rightarrow \mathbb{R}$ are defined
accordingly as,
\begin{align*}
  \Phi^\ast(\pi) := \sup_{g \in C} \left\{  \int g d\pi
  - \Phi(g)\right\} \quad \text{ and } \quad \Gamma^\ast(\pi) := \inf_{g \in D}
  \int g d\pi. 
\end{align*}
First, to determine $C^\ast,$ we see that for every $\pi \in M(S
\times S),$
\begin{align*}
  \sup_{g \in C} \left\{  \int g d\pi - \Phi(g)\right\} &= 
  \sup_{(\lambda,\phi) \ \in \ \mathbb{R}_+ \times C_b(S)}
  \left\{ \int \big( \phi(x) + \lambda c(x,y) \big) d\pi(x,y) - \left( \lambda
  \delta + \int \phi(x) d\mu(x) \right)\right\}\\
&= \sup_{(\lambda,\phi) \ \in \ \mathbb{R}_+ \times C_b(S)}
  \left\{ \lambda \left( \int c(x,y) d\pi(x,y) - \delta
  \right)  + \int \phi(x) \big( d\pi(x,y) - d\mu(x) \big) \right\},\\\\
&=
  \begin{cases}
    0 \quad\quad\text{ if } \int cd\pi \leq \delta \text{ and } \pi(A
    \times S) = \mu(A) \text{ for all } A  \in \mathcal{B}(S),\\
    \infty \quad\quad\text{ otherwise}. 
  \end{cases}
\end{align*}
Therefore,
\begin{align*}
  C^\ast = \left\{ \pi \in M(S \times S): \int cd\pi \leq \delta, \ \pi(A
  \times S) = \mu(A) \text{ for all } A \in \mathcal{B}(S)\right\} 
\quad \text{ and } \quad \Phi^\ast(\pi) = 0
\end{align*}
{Next, to identify $D^\ast,$ we show in Lemma
  \ref{Lemma-Ruling-Out-Negative-Measures} in Appendix
  \ref{Appendix-Proofs} that $\inf_{g \in D} \int gd\pi = -\infty$ if
  a measure $\pi \in M(S \times S)$ is not non-negative. On the other
  hand, if $\pi \in M(S \times S)$ is non-negative, then
\begin{align*}
  \inf \left\{ \int g(x,y) d\pi(x,y): g(x,y) \geq f(y) \text{ for all } x,y\right\} = \int f(y) d\pi(x,y).
\end{align*}
This is because, $f$ being an upper semicontinuous function that is
bounded from above, it can be approximated pointwise by a
monotonically decreasing sequence of continuous functions\footnote{Let
  $f:X \rightarrow \mathbb{R}$ be an upper semicontinuous function,
  that is bounded from above, defined on a Polish space $X.$ If
  $d(\cdot,\cdot)$ is a function that metrizes the Polish space $X,$
  then, for instance, the choice
  $f_n(x) = \sup_{y \in X}\{ f(y) - nd(x,y)\}$ is continuous, and
  satisfies $f_n \downarrow f$ pointwise}.  Then the above equality
follows as a consequence of monotone convergence theorem.} As a
result,
\begin{align*}
  D^\ast = \left\{\pi \in M_+(S \times S): \int f(y)d\pi(x,y) >
  -\infty \right\} \quad \text{ and } \quad \Gamma^\ast(\pi) = 
  \int f(y)d\pi(x,y). 
\end{align*}
Then, $\Gamma^\ast(\pi) - \Phi^\ast(\pi) = \int f(y)d\pi(x,y)$ on
$C^* \cap D^* = \{ \pi \in \cup_{\nu \in P(S)}\Pi(\mu,\nu): \int cd\pi
\leq \delta, \int f(y)d\pi(x,y) > -\infty \}.$ As $I$ is defined to
equal
$\sup\{ \int fd\nu: d_c(\mu,\nu) \leq \delta, \int fd\nu > -\infty\},$
it follows that
\begin{align}
  \sup \big\{ \Gamma^\ast(\pi) - \Phi^\ast(\pi): \pi \in C^\ast \cap
  D^\ast \big\} = I.
\label{LHS-Compact-Duality}
\end{align}
As the set $C \cap D$ contains points in the relative interiors of $C$
and $D$ (consider, for example, the candidate
$h(x,y) = c(x,y) + \sup_{x \in S} f(x)$ where
$\sup_{x \in S} f(x) < \infty$ because $f$ is upper semicontinuous and
$S$ is compact) and the epigraph of the function $\Gamma$ has
non-empty interior, it follows as a consequence of Fenchel's duality
theorem (see, for example, Theorem 1 in Chapter 7,
\citet{luenberger1997optimization}) that
\begin{align*}
  \inf_{g \in C \cap D} \big( \Phi(g) - \Gamma(g) \big) = \sup \big\{
  \Gamma^\ast(\pi) - \Phi^\ast(\pi): \pi \in C^\ast \cap 
  D^\ast \big\},
\end{align*}
where the supremum in the right is achieved by some
$\pi^\ast \in \Phi_{\mu,\delta}.$ In other words (see
\eqref{RHS-Compact-Duality} and \eqref{LHS-Compact-Duality}), 
\begin{align*}
  \inf \big\{ J(\lambda, \phi): \lambda \geq 0, \ \phi
  \in C_b(S),\ \phi(x) + \lambda c(x,y) \geq f(y) \text{ for all } x,y
  \big\} = \max \big\{ I(\pi) : \pi \in 
  \Phi_{\mu,\delta}\big\} =: I.
\end{align*}
As $C_b(S) \subseteq \Mu,$ it follows from the definition of $J$ that
\[J \leq \inf \big\{ J(\lambda, \phi): \lambda \geq 0, \ \phi \in
  C_b(S),\ \phi(x) + \lambda c(x,y) \geq f(y) \text{ for all } x,y
  \big\} = I.\] However, due to weak duality \eqref{EQ-WEAK-DUALITY},
we have $J \geq I.$ Therefore,
$J = I = \max\{ I(\pi): \pi \in \Phi_{\mu,\delta}\}.$ \hfill$\Box$

\subsubsection*{Proof of Proposition \ref{PROP-DUALITY-COMPACT-2}.}
First, observe that $c$ is a nonnegative lower semicontinuous function
defined on the Polish space $S \times S.$ Therefore, we can write
$c = \sup_n c_n,$ where $(c_n: n \geq 1)$ is a nondecreasing sequence
of continuous cost functions
$c_n: S \times S \rightarrow \mathbb{R}_+.$\footnote{For instance, one
  can choose
  $c_n(x,y) = \inf_{\tilde{x},\tilde{y} \in S} \left\{ c(\tilde{x},
    \tilde{y}) + n d((x,y),(\tilde{x},\tilde{y}))\right\},$ where the
  function $d$ metrizes the Polish space $S \times S.$ Then
  $c_n(\cdot,\cdot)$ is continuous, and satisfies $c_n(x,y) = 0$ if
  and only if $x = y.$} With $c_n$ as cost function, define the
corresponding primal and dual problems,
\[ I_n = \sup_{\pi \in \Phi^n_{\mu,\delta}} I(\pi) \quad\quad \text{
  and } \quad\quad J_n = \inf_{(\lambda,\phi) \in \Lambda_{c{_n},f}}
J(\lambda,\phi).\]
Here, the feasible sets $\Phi_{\mu,\delta}^n$ and $\Lambda_{c,f}^n$
are obtained by suitably modifying the cost function in sets
$\Phi_{\mu,\delta}$ and $\Lambda_{c,f}$ (defined in
\eqref{Primal-Feas-Set} and \eqref{Dual-Feas-Set}) as below:
\begin{align*}
  &\hspace{80pt}\Phi^n_{\mu,\delta} := \bigg\{\pi \in \bigcup_{\nu \in P(S)}
    \hspace{-4pt}\Pi(\mu, \nu):  
    \ \int  c_n d\pi \leq \delta \bigg\} \text{
    and }\\
  \Lambda_{c_n,f} &:= \big\{ (\lambda, \phi): \lambda \geq 0,\  \phi
                    \in \Mu, \ 
                    \phi(x) + \lambda c_n(x,y) \geq f(y) 
                    \text{ for all } x, y \in S \big\}.
\end{align*}
As the cost functions $c_n$ are continuous, due to Proposition
\ref{PROP-DUALITY-COMPACT-1}, there exists a sequence of measures
$(\pi_n:n \geq 1)$ such that
\begin{align}
I(\pi_n) = I_n = J_n.
\label{EQ-INTER-DUALITY-PROP2}
\end{align}
As $S$ is compact, the set $(\pi_n: n \geq 1)$ is automatically tight,
and due to Prokhorov's theorem, there exists a subsequence
$(\pi_{n_k}: k \geq 1)$ weakly converging to some
$\pi^* \in P(S \times S).$ First, we check that $\pi^*$ is feasible:
\begin{align*}
  \int cd\pi^* = \int \lim_n c_n d\pi^* = \lim_n \int c_n d\pi^* = \lim_n
  \lim_k \int c_n d\pi_{n_{_k}},
\end{align*}
where the second and third equalities, respectively, are consequences
of monotone convergence and weak convergence. In addition, since
$(c_n: n \geq 1)$ is a nondecreasing sequence of functions,
\[\lim_k \int c_n d\pi_{n_{_k}} \leq \varliminf_k \int
  c_{n_{_k}}d\pi_{n_{_k}} \leq \delta.\] Therefore,
$\int cd\pi^* \leq \delta.$ Again as a simple consequence of weak
convergence, we have $\int g(x)d\pi^\ast(x,y) = \lim_n \int
g(x)d\pi_{n_k}(x,y) = \int g(x)d\mu(x),$ for every $g \in C_b(S).$ 
Therefore, $\pi^* \in \Phi_{\mu_,\delta},$ and
is indeed feasible. Next, as $f$ is upper semicontinuous and $S$ is
compact, $f$ is bounded from above. Then, due to the weak convergence
$\pi_{n_k} \Rightarrow \pi^*$, the objective function $I(\pi^*)$
satisfies
\begin{align*}
  I(\pi^*) = \int fd\pi^* \geq \ \varlimsup_{k} \int f d\pi_{n_{k}}  =
  \ \varlimsup_{k}  I(\pi_{n_{_k}}) = \ \varlimsup_{k}  J_{n_{_k}},
\end{align*}
because $ I(\pi_{n_{_k}}) = J_{n_{_k}}$ as in
\eqref{EQ-INTER-DUALITY-PROP2}. Further, as $c_n \leq c,$ it follows
that $\Lambda_{c_{_n}, f}$ is a subset of $\Lambda_{c,f}$ and hence
$J_{n} \geq J,$ for all $n.$ As a result, we obtain,
\[ I(\pi^*) \geq \ \varlimsup_{k} J_{n_{_k}} \geq J.\]
Since $I$ never exceeds $J$ (see \eqref{EQ-WEAK-DUALITY}), it follows
that $I(\pi^*) = I = J.$ \hfill$\Box$

\subsection{A note on additional notations and measurability.}
\label{Sec-Notations-Measurability}
Given that we have established duality in compact spaces, the next
step is to establish the same when $S$ is not compact. For this
purpose, we need the following additional notation. For any
probability measure $\pi \in P(S \times S),$ let
\[S_\pi := \spt(\pi_{_X}) \cup \spt(\pi_{_Y}),\] where
$\spt(\pi_{_X})$ and $\spt(\pi_{_Y})$ denote the respective supports
of marginals $\pi_{_X}(\cdot) := \pi(\ \cdot \times S)$ and
$\pi_{_Y}(\cdot) := \pi(S \times \cdot\ ).$ As every probability
measure defined on a Polish space has $\sigma$-compact support, the
set $S_\pi \times S_\pi,$ which is a subset of $S \times S,$ is
$\sigma$-compact. As we shall see in the proof of Proposition
\ref{PROP-DUALITY-IND-MEAS}, $S_\pi \times S_\pi$ can be written as
the union of an increasing sequence of compact subsets
$(S_n \times S_n: n \geq 1).$ It will then be easy to make progress
towards strong duality in noncompact sets such as $S_\pi \times S_\pi$
by utilizing the strong duality results derived in Section
\ref{Sec-Duality-Compact} (which are applicable for compact sets
$S_n \times S_n$) via a sequential argument; this is accomplished in
Section \ref{Sec-Duality-Noncompact} after introducing additional
notation as follows. 
{For any closed $K \subseteq S,$ let
\begin{align}
  \Lambda(K \times K) := \big\{ (\lambda,\phi): \lambda \geq 0, \
  \phi  \in \mathfrak{m}_{\mathcal{U}}(K;\overline{\mathbb{R}}), \ 
  \phi(x) + \lambda c(x,y) \geq f(y) \text{ for all } x,y \in K
  \big\}, 
\label{Notn-Lambda}
\end{align}
where $\mathfrak{m}_{\mathcal{U}}(K;\overline{\mathbb{R}})$ is used to
denote the collection of measurable functions
$\phi: (K, \mathcal{U}(K)) \rightarrow (\overline{\mathbb{R}},
\mathcal{B}(\overline{\mathbb{R}})),$ with $\mathcal{U}(K)$ denoting
the universal $\sigma-$algebra of the Polish space $K.$} With this
notation, the dual feasible set $\Lambda_{c,f}$ is nothing but
$\Lambda(S \times S).$ The next issue we address is the measurability
of functions of the form
$\phi_\lambda(x) = \sup_{y \in S} \big\{ f(y) - \lambda c(x,y)\big\}.$
For any $u \in \mathbb{R},$
\begin{align*}
  \big\{ \phi_\lambda(x) > u \big\} = \text{Proj}_1\left(\left\{
  (x,y) \in S \times S : f(y) - \lambda c(x,y) > u \right\}\right),
\end{align*}
which is analytic, because projections
$\text{Proj}_i(A) := \{ x_i \in S: (x_1,x_2) \in A\}, i = 1,2,$ of any
Borel measurable set $A$ are analytic (see, for example, Proposition
7.39 in \citet{bertsekas1978stochastic}). As analytic subsets of $S$
lie in $\mathcal{U}(S)$ (see Corollary 7.42.1 in
\citet{bertsekas1978stochastic}), the function
$\phi_\lambda(x): (S, \mathcal{U}(S)) \rightarrow
(\overline{\mathbb{R}}, \mathcal{B}(\overline{\mathbb{R}}))$ is
measurable (that is, universally measurable;
refer Chapter 7 of \citet{bertsekas1978stochastic} for an introduction
to universal measurability and analytic sets).

\subsection{Extension of duality to non-compact spaces.}
\label{Sec-Duality-Noncompact}
Proposition \ref{PROP-DUALITY-IND-MEAS}, presented below in this
section, is an important step towards extending the strong duality
results proved in Section \ref{Sec-Duality-Compact} to non-compact
sets.
\begin{proposition}
  \label{PROP-DUALITY-IND-MEAS}
  Suppose that Assumptions (A1) and (A2) are in force. Let $\pi \in
  P(S \times S)$ be a probability measure satisfying\\
  (a) $\int c(x,y)d\pi(x,y) < \infty,$\\
  (b) $\int f(y) d\pi(x,y) \in (-\infty, \infty),$ and \\
  (c) $\pi(A \times S) = \mu(A)$ for every $A \in \borel(S).$ Then
  \[ \inf_{(\lambda,\phi) \in \Lambda(S_\pi \times S_\pi)}
  J(\lambda,\phi) \leq \  I \]
\end{proposition}

\subsubsection*{Proof of Proposition \ref{PROP-DUALITY-IND-MEAS}.} 
As any Borel probability measure on a Polish space is concentrated on
a $\sigma$-compact set (see Theorem 1.3 of \citet{MR1700749}), the
sets $\spt(\pi_{_X})$ and $\spt(\pi_{_Y})$ are $\sigma$-compact, and
therefore $S_{\pi} \times S_{\pi}$ is $\sigma$-compact as well. Then,
by the definition of $\sigma$-compactness, the set
$S_{\pi} \times S_\pi$ can be written as the union of an increasing
sequence of compact subsets $(C_n: n \geq 1)$ of
$S_{\pi} \times S_{\pi}.$ If we let $S_n$ to be the union of the
projections of $C_n$ over its two coordinates, then it follows that
$(S_n \times S_n: n \geq 1)$ is an increasing sequence of subsets of
$S_\pi \times S_\pi$ satisfying
$\spt(\pi) \subseteq S_\pi \times S_\pi = \cup_{n \geq 1} (S_n \times
S_n).$ As $\int |f(y)|d\pi(x,y)$ and $\int c(x,y)d\pi(x,y)$ are
finite, 
one can pick the increasing
sequence $(S_n \times S_n: n \geq 1)$ with
$S_\pi \times S_\pi = \cup_{n \geq 1} (S_n \times S_n)$ to be
satisfying, 
\begin{subequations}
\begin{align}
  p_n := \pi (S_n \times S_n) \ \geq 1&-\frac{1}{n}, \label{Sn-Choice-1}\\  \int
  c(x,y)\mathbf{1}_{(S_n \times S_n)^c} d \pi(x,y) \ &\leq \
                                                    \frac{\delta}{n} \ \  \text{ and } \label{Sn-Choice-2}\\
  \int |f|(y)\mathbf{1}_{(S_n \times S_n)^c}d\pi(x,y) &\leq
                                                      \frac{1}{n}, \label{Sn-Choice-3}
\end{align}
\end{subequations}
where $(S_n \times S_n)^c := (S \times S) \setminus (S_n \times S_n),$
and $\delta \in (0,\infty)$ is introduced while defining the primal
feasibility set $\Phi_{\mu,\delta}$ (recall that
$I := \sup\{ I(\pi): \pi \in \Phi_{\mu,\delta}\}$) and the dual
objective $J(\lambda, \phi) = \lambda \delta + \int \phi d\mu$ in
Section \ref{Sec-primal-prob}.
\noindent Having chosen the compact subsets $(S_n: n \geq
1),$ define a joint probability measure $\pi_{n} \in P(S_n \times
S_n)$ and its corresponding marginal $\mu_n \in P(S_n)$ as below:
\[ \pi_{n} (\cdot) = \frac{\pi \big(\ \cdot \ \cap \ (S_n \times S_n)
    \big) }{p_n} \quad \text{ and } \quad \mu_n(\cdot) = \pi_{n}(\
  \cdot \times S_n),\] for every $n \geq
1.$ Additionally, let $\delta_n := \delta(1-1/n).$ For every $n \geq
1,$ these new definitions can be used to define ``restricted'' primal
and dual optimization problems $I_n$ and
$J_n$ supported on the compact set $S_n \times S_n:$
\begin{align}
  I_n := \sup &\bigg\{ \int f(y) \gamma(dx,dy) : \gamma \in \hspace{-4pt}
                \bigcup_{\nu \in P(S_n)} \hspace{-8pt}\Pi(\mu_n,\nu), \int
                c(x,y)d\gamma(x,y) \leq \delta_n 
        \bigg\} \text{ and }  \label{Res-Primal}\\
 & J_n := \inf \left\{ \lambda \delta_n + \int \phi d\mu_n:  \
  (\lambda, \phi)   \in \Lambda (S_n \times S_n)  \right\}.
  \label{Res-Dual}
\end{align}
To comprehend \eqref{Res-Dual}, refer the definition of
$\Lambda(\cdot)$ in \eqref{Notn-Lambda}.
As $S_n$ is compact, due to the duality result in Proposition
\ref{PROP-DUALITY-COMPACT-2}, we know that there exists a
$\gamma_n^* \in P(S_n \times S_n)$ that is feasible for optimization of
$I_n$ and satisfies
\begin{align}
  \label{INTER-DUALITY-THM1}
  \int f(y) \gamma_n^*(dx,dy) = I_n = J_n.
\end{align}
Using this optimal measure $\gamma_n^*,$ one can, in turn, construct a
measure $\tilde{\pi} \in P(S \times S)$ by stitching it together with
the residual portion of $\pi$ as below:
\begin{align*}
  \tilde{\pi}(\cdot) = p_n \gamma^*_n \big(\ \cdot \ \cap\  (S_n
  \times S_n) \big) +
  \pi \big( \ \cdot \ \cap \ (S_n \times S_n)^c\big). 
\end{align*}
{Having tailored a candidate measure $\tilde{\pi} \in
  P(S \times S),$ next we check its feasibility that $\tilde{\pi} \in
  \Phi_{\mu,\delta}$ as follows: Since $\gamma^*_n \in
  \Pi(\mu_n,\nu)$ for some $\nu \in
  P(S_n),$ it is easy to check, as below, that $\tilde{\pi}$ has
  $\mu$ as its marginal for the first component: first, it follows
  from the definition of $\tilde{\pi}$ that
\begin{align*}
  \tilde{\pi}(\ \cdot \times S) 
  &= p_n \gamma_n^{\ast} ((\ \cdot \ \cap  S_n)\times S_n)  +
    \pi ( (\ \cdot \times S) \ \cap \ (S_n \times S_n)^c)\\
  &=  p_n \mu_n(\ \cdot \ \cap S_n) + \pi ( (\ \cdot \times S) \
    \cap \ (S_n  \times S_n)^c). 
\end{align*}
As $\mu_n(\cdot) = \pi(\ \cdot \times S_n )/p_n,$ it follows that $p_n
\mu_n(\ \cdot \ \cap S_n)  = \pi((\ \cdot \times S) \cap (S_n \times
S_n )),$ and consequently,  
\begin{align*}
  \tilde{\pi}(\ \cdot \times S) = \pi((\ \cdot \times S) \cap (S_n  \times S_n )) + \pi ( (\
  \cdot \times S) \ \cap \ (S_n  \times S_n)^c)
  = \pi(\ \cdot \times S) = \mu(\cdot).
\end{align*}
Therefore, $\tilde{\pi} \in \Pi(\mu,\nu)$ for some $\nu \in P(S).$}
Further, 
as $\gamma_n^*$ is a feasible solution for the optimization problem in
\eqref{Res-Primal}, 
\begin{align*}
  \int cd \tilde{\pi} = p_n \int_{_{S_n \times S_n}} \hspace{-20pt} c
                        d\gamma_n^* + \int_{_{(S_n \times S_n)^c}} \hspace{-22pt} c d\pi
                      \ \leq \  \delta_n + \frac{\delta}{n},
\end{align*}
which does not exceed $\delta.$ To see this, refer \eqref{Sn-Choice-2}
and recall that $\delta_n := \delta (1-1/n).$ Therefore, the stitched
measure $\tilde{\pi}$ is primal feasible, that is
$\tilde{\pi} \in \Phi_{\mu,\delta}.$
Consequently,
\begin{align*}
  I \geq I(\tilde{\pi}) = \int f(y) d\tilde{\pi}(x,y) = p_n\int f(y)d\gamma_n^*(x,y) +
  \int f(y)\mathbf{1}_{(S_n \times S_n)^c}   d\pi(x,y).
\end{align*}
As $S_n$ is chosen to satisfy \eqref{Sn-Choice-3}, we have
$I \geq p_n I_n - n^{-1}.$ Therefore, it is immediate from
\eqref{INTER-DUALITY-THM1} that
\begin{align}
  \label{EQ-INTER-DUAL-BND}
    J_n \leq \left( 1 - \frac{1}{n}\right)^{-1} \left( I + \frac{1}{n}\right), 
\end{align}
which, if $I < \infty,$ is a useful upper bound for all of
$J_n, n \geq 1.$ See that the statement of Proposition
\ref{PROP-DUALITY-IND-MEAS} already holds when $I$ equals $\infty.$
Therefore, let us take $I$ to be finite. Next, given $n \geq 1$ and
$\epsilon > 0,$ take an $\epsilon$-optimal solution
$(\lambda_{n},\phi_n)$ for $J_n:$ that is,
$(\lambda_{n},\phi_n) \in \Lambda(S_n \times S_n)$ and
\[ \lambda_n \delta_n + \int \phi_n d\mu_n \leq J_n + \epsilon.\]
For any pair $(\lambda_n, \phi)$ that belongs to
$\Lambda(S_n \times S_n ),$ we have from the definition of
$\Lambda(S_n \times S_n )$ that $\phi (x)$ is larger than
$\sup_{z \in S_n} \{ f(z) - \lambda_n c(x,z) \},$ for every
$x \in S_n.$  
\begin{align*}
  \lambda_n \delta_n + \int \sup_{z \in S_n}
  \big\{  f(z) - \lambda_n   c(x,z) \big\} d\mu_n(x) \leq J_n + \epsilon,
\end{align*}
for every $n.$
{As $\mu_n(\cdot) = \pi(\cdot \times S_n)/p_n,$ combining the above
expression with \eqref{EQ-INTER-DUAL-BND} results in 
\begin{align}
  \label{EQ-DUAL-BND}
  \varlimsup_{n \rightarrow \infty} \left(\lambda_n \delta_n + \int
  \sup_{z \in S_n} 
  \big\{  f(z) - \lambda_n   c(x,z) \big\} \frac{\mathbf{1}_{S_n}(x)
  \mathbf{1}_{S_n}(y)}{p_n} d\pi(x,y) \right) \leq I +   \epsilon. 
\end{align}
Since $c(x,x) = 0,$ the integrand admits
$f(x)$ as a lower bound on $S_n \times S_n;$ further, as $f \in
L^1(d\mu),$
$-f^-(x)$ serves as a common integrable lower bound for all $n \geq
1.$} This has two consequences: First, because of the finite upper
bound in \eqref{EQ-DUAL-BND}, $\varliminf_n
\lambda_n$ and $\varlimsup_n
\lambda_n$ are finite (recall that $\lambda_n \geq
0$ as well), and there exists a subsequence $(\lambda_{n_k}: k \geq
1)$ such that $\lambda_{n_k} \rightarrow \lambda^\ast,$ as $k
\rightarrow \infty,$ for some $\lambda^\ast \in
[0,\infty).$ The second consequence of the existence of a common
integrable lower bound is that we can apply Fatou's lemma in
\eqref{EQ-DUAL-BND} along the subsequence $({n_k}: k \geq
1)$ to obtain
\begin{align*}
  I + \epsilon \ \geq \ \varliminf_{k \rightarrow \infty} 
  &\left(\lambda_{n_k} \delta_{n_k} + \int 
    \sup_{z \in S_{n_k}} \big\{  f(z) - \lambda_n   c(x,z) \big\}
    \frac{\mathbf{1}_{S_{n_k}}(x)   \mathbf{1}_{S_{n_k}}(y)}{p_{n_k}}
    d\pi(x,y) \right)\\
  &\quad\quad\quad\quad\quad\quad \geq \lambda^\ast \delta +
    \int_{S_\pi \times S_\pi}
    \sup_{z \in \cup_k  S_{n_k}} \big\{ f(z) - \lambda^* 
    c(x,z)\big\} d\pi(x,y)\\
    &\quad\quad\quad\quad\quad\quad = \lambda^\ast \delta +
    \int_S  \ \sup_{z \in S_\pi} \big\{ f(z) - \lambda^* 
    c(x,z)\big\} d\mu(x).
\end{align*}
Here, we have used that $\delta_n \rightarrow \delta,$ $p_n
\rightarrow 1$ as $n \rightarrow \infty,$ $\cup_{n \geq 1}S_n =
\cup_{k \geq 1}S_{n_k} = S_\pi,$ and the fact that
$\pi$ is supported on a subset of $S_\pi \times
S_\pi.$ The fact that $\varliminf_k \sup_{z \in S_{n_k}} \{ f(z) -
\lambda_{n_k} c(x,z)\}$ is at least $\sup_{z \in \cup_k S_{n_k}}\{
f(z) - \lambda^\ast c(x,z)\},$ for every $x \in
S,$ is carefully checked in Lemma \ref{Lem-Lim-Int-Interchange},
Appendix \ref{Appendix-Proofs}. 
If we let $\phi^*(x) := \sup_{y \in S_\pi} \big\{ f(y) -
\lambda^* c(x,y)\big\},$ then $(\lambda^*, \phi^*) \in \Lambda(S_\pi
\times S_\pi),$ and as $\epsilon >
0$ is arbitrary, it follows from the above inequality that
$J(\lambda^*,\phi^*) \leq I.$ This completes the proof. \hfill$\Box$

\subsubsection*{\textbf{Proof of Theorem  \ref{THM-STRONG-DUALITY}(a)}.}
If $I = \infty,$ then as $I$ never exceeds $J,$ $J$ equals $\infty$ as
well, and there is nothing to prove. Next, consider the case where $I$
is finite: that is, $I \in [\int fd\mu, \infty).$ Let $E$ denote the
convex set of probability measures $\pi \in P(S \times S)$ that
satisfy conditions (a)-(c) of Proposition \ref{PROP-DUALITY-IND-MEAS}.
Then, due to Proposition \ref{PROP-DUALITY-IND-MEAS}, for every
$\pi \in E,$
\begin{align}
  I  \geq \inf_{(\lambda,\phi) \in \Lambda (S_\pi \times S_\pi )} \left\{\lambda \delta
       + \int \phi(x) d\mu(x) \right\} \geq \ \inf_{\lambda \geq 0} \left\{ \lambda \delta  +
       \int \sup_{y \in S_\pi} \big\{ f(y)
       -  \lambda c(x,y) \big\} d\mu(x) \right\}. \nonumber
\end{align}
For any $\pi \in E$ and $\lambda \geq 0,$
define
\begin{align*}
  T(\lambda,\pi) :=
  \lambda \delta + \int \sup_{y \in S_\pi} \big\{
  f(y) - \lambda c(x,y)\big\} d\mu(x).  
\end{align*}
As $c(x,x) = 0,$ it is easy to see that
$T(\lambda,\pi) \geq \lambda \delta + \int fd\mu.$ 
Since $T(\lambda,\pi) > I$ for every
$\lambda > \lambda_{\max}:= (I - \int fd\mu)/\delta,$ one can rather
restrict attention to the compact subset $[0,\lambda_{\text{max}}],$
as follows:
\begin{align}
  I  &\geq \inf_{\lambda \in [0,\lambda_{\text{max}}]} T(\lambda,\pi). 
      \label{EQ-THM1-INTER1}
\end{align}
Being a pointwise supremum of a family of affine functions,
$\sup_{y \in S_\pi} \{ f(y) - \lambda c(x,y) \}$ is a
lower semicontinuous with respect to the variable $\lambda.$ Further,
recall that $T(\lambda, \pi) \geq \lambda \delta + \int fd\mu.$
Thus, for any $\lambda_n \rightarrow \lambda,$ due to Fatou's
lemma,
\begin{align*}
  \varliminf_{n \rightarrow \infty} T(\lambda_n ,\pi) &\geq \lambda
                                                        \delta + \int \varliminf_n \sup_{y \in S_\pi} \left\{ f(y) - \lambda_n
                                                        c(x,y)\right\} d\mu(x)\\
                                                      &\geq \lambda \delta + \int \sup_{y \in S_\pi} \left\{ f(y) - \lambda
                                                        c(x,y)\right\}
                                                        d\mu(x) =
                                                        T(\lambda,\pi),
\end{align*}
which means that $T(\lambda,\pi)$ is lower semicontinuous in
$\lambda.$ Along with this lower semicontinuity, for every fixed
$\pi,$ $T(\lambda,\pi)$ is also convex in $\lambda.$ In addition, for
any $\alpha \in (0,1)$ and $\pi_1, \pi_2 \in E,$
\begin{align*}
  T(\lambda, \alpha \pi_1 + (1-\alpha) \pi_2) &= \lambda \delta + \int \sup_{y
                                                \in S_{\alpha \pi_1 + (1-\alpha) \pi_2}} \big\{ f(y)
                                                - \lambda c(x,y)\big\}
                                                d\mu(x)\\ 
                                              &\geq \max_{i = 1,2}
                                                \left\{\lambda \delta
                                                + \int \sup_{y \in S_{\pi_i}} \big\{ f(y) - \lambda
                                                c(x,y)\big\} d\mu(x)
                                                \right\} \geq \ \alpha
                                                T(\lambda, \pi_1) +
                                                (1-\alpha) T(\lambda,
                                                \pi_2).  
\end{align*}
Therefore, $T(\lambda, \pi)$ is a concave function in $\pi$ for every
fixed $\lambda.$ One can apply a standard minimax theorem (see,
for example, \citet{sion1958}) to conclude 
\begin{align*}
  \sup_{\pi \in E} \inf_{\lambda \in [0,\lambda_{\text{max}}]}
  T(\lambda, \pi) = \inf_{\lambda \in [0,\lambda_{\text{max}}]}
  \sup_{\pi \in E} T(\lambda,\pi).
\end{align*}
This observation, in conjunction with \eqref{EQ-THM1-INTER1}, yields
\begin{align*}
  I \geq \sup_{\pi \in E} \inf_{\lambda \in [0,\lambda_{\text{max}}]}
  T(\lambda,\pi)
  = \inf_{\lambda \in [0,\lambda_{\text{max}}]} \left\{ \lambda
    \delta + \sup_{\pi \in E} \int \sup_{y
                    \in S_\pi} \big\{ f(y) - \lambda
    c(x,y)\big\} d\mu(x) \right\}
\end{align*}
Lemma \ref{LEM-SUPP-EQUIV} below is the last piece of technicality
that completes the proof of Part (a) of Theorem
\ref{THM-STRONG-DUALITY}.
\begin{lemma}
  Suppose that Assumptions (A1) and (A2) are in force. Then, for any
  $\lambda \geq 0,$
\begin{align*}
  \sup_{\pi \in E} \int \sup_{y
  \in S_\pi} \big\{ f(y) - \lambda
  c(x,y)\big\} d\mu(x)  = \int \sup_{y \in S} \big\{ f(y) - \lambda
  c(x,y)\big\} d\mu(x).
\end{align*}
\label{LEM-SUPP-EQUIV}  
\end{lemma}
\noindent If Lemma \ref{LEM-SUPP-EQUIV} holds, then it is automatic that
\begin{align*}
  I   &\geq \inf_{\lambda \in [0,\lambda_{\text{max}}]} \left\{ \lambda
    \delta + \int \sup_{y \in S} \big\{ f(y) - \lambda
    c(x,y)\big\} d\mu(x) \right\} \geq J,
\end{align*}
which will complete the proof of Part (a) of Theorem
\ref{THM-STRONG-DUALITY}. 

\begin{proof}
  \noindent \textbf{Proof of Lemma \ref{LEM-SUPP-EQUIV}.}  For
  brevity, let $g(x,y) :=f(y) - \lambda c(x,y)$ and
  $\phi_\lambda(x) := \sup_{y \in S} g(x,y).$ For any $n \geq 1$ and
  $k \leq n^2,$ define the sets,
  \begin{align*}
    A_{k,n} := \left\{ (x,y) \in S \times S: (k-1)/n \leq g(x,y) \leq
    k/n \right\} \ \  \text{ and }  \ \ 
    B_{k,n} := \text{Proj}_1\left( A_{k,n} \right) \setminus \cup_{j
              > k} \text{Proj}_1 \left( A_{j,n}\right). 
  \end{align*}
  Here, recall that for any $A \subset S \times S,$
  $\text{Proj}_1(A) = \{x_1: (x_1,x_2) \in A\}.$ The sequence
  $(B_{k,n} : k \leq n^2)$ also admits the following equivalent
  backward recursive definition,
  \begin{align*}
    B_{n^2,n} = \text{Proj}_1 \left( A_{n^2,n}\right), \text{ and } B_{k,n} =
    \text{Proj}_1 \left( A_{k,n}\right) \setminus \cup_{j > k} B_{j,n}
    \text{ for } k < n^2,
  \end{align*}
  that renders the collection $(B_{k,n}: k \leq n^2)$ disjoint.  As
  the function $g$ is upper semicontinuous, it is immediate that the
  sets $A_{k,n}$ are Borel measurable and their corresponding
  projections $B_{k,n}$ are analytic subsets of $S.$ Then, due to
  Jankov-von Neumann selection theorem\footnote{See, for example,
    Chapter 7 of \citet{bertsekas1978stochastic} for an introduction
    to analytic subsets and Jankov-von Neumann measurable selection
    theorem}, there exists, for each $k \leq n^2$ such that
  $A_{k,n} \neq \emptyset$, an universally measurable function
  $\gamma_k(x): \text{Proj}_1(A_{k,n}) \rightarrow S$ such that
  $(x,\gamma_k(x)) \in A_{k,n},$ or, in other words,
  \begin{align}
    \frac{k-1}{n} \leq g(x, \gamma_k(x)) \leq \frac{k}{n}.
    \label{Selector-Construction}
  \end{align}
  Next, let us define the function $\Gamma_n: S \rightarrow S$ as
  below:
  \begin{align*}
    \Gamma_n(x) =
    \begin{cases}
      \gamma_k(x), \quad &\text{ if } x \in B_{k,n} \text{ for some } k \leq n^2\\
       \ \ x, &\text{ otherwise }.
    \end{cases}
  \end{align*}
  This definition is possible because
  $B_{k,n} \subseteq \text{Proj}_1(A_{k,n}),$ and the collection
  $(B_{k,n}: k \leq n^2)$ is disjoint. As each $\gamma_k$ is
  universally measurable, the composite function $\Gamma_n(x)$ is
  universally measurable as well, and satisfies that
  \begin{align*}
    \sup \left\{ g(x,y): g(x,y) \leq  n, y \in S \right\} -
    \frac{1}{n} 
    \leq
    g(x, \Gamma_n(x)) \leq f(x) \vee n,
  \end{align*}
  for each $x.$ Both sides of the inequality above can be inferred
  from \eqref{Selector-Construction} after recalling that
  $g(x,x) = f(x) - \lambda c(x,x) = f(x).$ Letting
  $n \rightarrow \infty,$ we see that
  \begin{align}
    \phi_\lambda(x) := \sup \{ g(x,y) :  y \in S \} \ \leq \  \varliminf_n g(x,
    \Gamma_n(x) ) \ \leq \  \infty.
    \label{Near-Optimality-of-Selectors}
  \end{align}
  Next, define the family of probability measures $(\pi_n: n \geq 1)$
  as below:
  \begin{align*}
    d\pi_n(x,y) := d\mu(x) \times d\delta_{\Gamma_n(x)} (y), 
  \end{align*}
  with $\delta_a(\cdot)$ representing the dirac measure centred at
  $a \in S.$ As $g(x,y) \leq f(x) \vee n,$ $\pi_n$ almost surely, we
  have that the functions $c(x,y)$ and $f(y)$ are integrable with
  respect to $\pi_n.$ This means that $\pi_n \in E$ (here, recall that
  $E$ is the set of probability measures satisfying conditions (a)-(c)
  in the statement of Proposition
  \ref{PROP-DUALITY-IND-MEAS}). Further, as
  $S_{\pi_n} \supseteq \spt (\mu),$
  \begin{align*}
    \sup_{y \in S_{\pi_n}} g(x,y) = \sup_{y \in S_{\pi_n}} \big\{ f(y) - \lambda c(x,y)
    \big\} \geq f(x) - \lambda c(x,x) =f(x), \quad \text{ for all } x,
    \ \mu \text{ a.s., }
  \end{align*}
  which is integrable with respect to the measure $\mu.$ Therefore,
  due to Fatou's lemma,
\begin{align*}
  \varliminf_{n \rightarrow \infty} \int \sup_{y \in
  S_{\pi_{n}}} g(x,y) d\mu(x) 
  &\geq \int \varliminf_{n \rightarrow \infty} \sup_{y \in
    S_{\pi_{n}}} g(x,y)d\mu(x) \\
  &\geq \int \varliminf_{n \rightarrow \infty} g(x,\Gamma_n(x))d\mu(x) 
  \geq \int \phi_\lambda(x) d\mu(x),
\end{align*}
where the last inequality follows from
\eqref{Near-Optimality-of-Selectors}. Since $\pi_n$ is a member of set
$E,$ for every $n,$ it is then immediate that
\begin{align*}
  \sup_{\pi \in E} \int \sup_{y \in S_\pi} g(x,y) d\mu(x) \geq   \varliminf_{n \rightarrow \infty} \int \sup_{y \in
  S_{\pi_{n}}}
  g(x,y)
  d\mu(x) \geq \int \phi_\lambda(x) d\mu(x).\\
\text{ Thus, }  \sup_{\pi \in E} \int \sup_{y \in S_\pi} \big\{ f(y)
  -\lambda c(x,y) \big\}
  d\mu(x) \geq \int
    \sup_{y \in S} \big\{ f(y) - \lambda c(x,y) \big\} d\mu(x).
\end{align*}
As $S_\pi \subseteq S,$ for every $\pi \in E,$ the inequality in the
reverse direction is trivially true. This completes the proof of Lemma
\ref{LEM-SUPP-EQUIV}.
\end{proof}

\subsubsection*{\textbf{Proof of Theorem
    \ref{THM-STRONG-DUALITY}(b)}.}
To summarize, in the proof of Part (a) of Theorem
\ref{THM-STRONG-DUALITY}, we established that
\begin{align*}
  I = \inf_{\lambda \geq 0} \left\{ \lambda \delta + \int
  \phi_{\lambda}(x) d\mu(x) \right\},
\end{align*}
where $\phi_\lambda(x) := \sup_{y \in S} \{f(y) - \lambda c(x,y)\}$
is lower-semicontinuous in $\lambda,$ and is bounded by $f(x)$ from
below. Recall that $\int f d\mu$ is finite. If we let
\begin{align*}
  g(\lambda) := \lambda \delta + \int \phi_{\lambda}(x) d\mu(x),
\end{align*}
then due to a routine application of Fatou's lemma, we have that
$\varliminf_{n \rightarrow \infty} g(\lambda_n) \geq g(\lambda),$ 
whenever $\lambda_n \rightarrow \lambda.$ In other words, $g(\cdot)$
is lower semicontinuous. In addition, as
$g(\lambda) \geq \lambda \delta + \int f(x) d\mu(x),$ we have that
$g(\lambda) \rightarrow \infty$ if $\lambda \rightarrow \infty.$ This
means that the level sets $\{\lambda \geq 0: g(\lambda) \leq u\}$ are compact
for every $u,$ and therefore, $g(\cdot)$ attains its infimum. In other
words, there exists a $\lambda^\ast  \in [0,\infty)$ such that
\begin{align*}
  I = \lambda^\ast \delta + \int \phi_{_{\lambda^\ast}}(x) d\mu(x).
\end{align*}
Thus, we conclude that a dual optimizer of the form
$(\lambda^\ast, \phi_{_{\lambda^\ast}})$ always exists.  Next, if the
primal optimizer $\pi^\ast$ satisfying $I = I(\pi^\ast)$ also exist,
then
\begin{align*}
  \int f(y) d\pi^\ast(x,y) = \lambda^\ast \delta + \int \phi_{_{\lambda^\ast}} (x)
  d\mu(x) ,
\end{align*}
which gives us complementary slackness conditions \eqref{Comp-Slack-1}
and \eqref{Comp-Slack-2}, as a consequence of equality getting
enforced in the following series of inequalities:
\begin{align*}
  \int f(y) d\pi^\ast(x,y) &= \int \big( f(y) - \lambda^\ast c(x,y)
                             \big) d\pi^\ast(x,y) + \lambda^\ast \int c(x,y) d\pi^\ast(x,y)\\
                           &\leq \int \phi_{_{\lambda^\ast}}(x)
                             d\pi^\ast(x,y) + \lambda^\ast \int c(x,y)
                             d\pi^\ast(x,y)\\
                           &\leq \int \phi_{_{\lambda^\ast}}(x) d\mu(x) + \lambda^\ast \delta.
 \end{align*}
 Alternatively, if the complementary slackness conditions
 \eqref{Comp-Slack-1} and \eqref{Comp-Slack-2} are satisfied by any
 $\pi^\ast \in \Phi_{\mu,\delta}$ and
 $(\lambda^\ast, \phi_{\lambda}^\ast) \in \Lambda_{c,f},$ then
 automatically,
 \begin{align*}
   \int f(y)d\pi^\ast(x,y) &= \int \big(f(y) - \lambda^\ast c(x,y)\big)
                             d\pi^\ast(x,y) + \lambda^\ast \int c(x,y)d\pi^\ast(x,y)\\
                           &= \int \sup_{z \in S} \big\{f(z) -
                             \lambda^\ast c(x,z) \big\} d\pi^\ast(x,y)
                             + \lambda^\ast \delta\\ 
                           &= \int \phi_{_{\lambda^\ast}}(x) d\mu(x) +
                             \lambda^\ast \delta = J(\lambda^\ast, \phi_{\lambda^\ast}),
 \end{align*}
 thus proving that $\pi^\ast$ and
 $(\lambda^\ast, \phi_{_{\lambda^\ast}})$ are, respectively, the
 primal and dual optimizers.  This completes the proof of Theorem
 \ref{THM-STRONG-DUALITY}. \hfill$\Box$

 \section{{On the existence of a primal optimal transport plan.}}
 \label{Sec-Primal-Opt-Existence}
 Unlike the well-known Kantorovich duality where an optimizer that
 attains the infimum in \eqref{Defn-Transport-Cost} exists if the
 transport cost function $c(\cdot,\cdot)$ is nonnegative and lower
 semicontinuous (see, for example, Theorem 4.1 in
 \citet{villani2008optimal}), a primal optimizer $\pi^\ast$ satisfying
 $I(\pi^\ast) = I = J$ need not exist for the primal problem,
 $I = \sup\{ I(\pi): \pi \in \Phi_{\mu,\delta}\},$ that we have
 considered in this paper.  The feasible set for the problem
 \eqref{Defn-Transport-Cost} in Kantorovich duality, which is the set
 of all joint distributions $\Pi(\mu,\nu)$ with given $\mu$ and $\nu$
 as marginal distributions, is compact in the weak topology. On the
 other hand, the primal feasibility set $\Phi_{\mu,\delta},$ for a
 given $\mu$ and $\delta > 0,$ that we consider in this paper is not
 necessarily compact, and the existence of a primal optimizer
 $\pi^\ast \in \Phi_{\mu,\delta},$ satisfying $I(\pi^\ast) = I,$ is
 not guaranteed.  Example \ref{Eg-No-Primal-Opt} below demonstrates a
 setting where there is no primal optimizer
 $\pi^\ast \in \Phi_{\mu,\delta}$ satisfying $I(\pi^\ast) = I$ even if
 the transport cost $c(\cdot,\cdot)$ is chosen as a metric defined on
 the Polish space $S.$
 \begin{example}
   Consider $S = \mathbb{R}.$ Let $\mu = \delta_{\{0\}},$ the dirac
   measure at $0,$ be the reference measure defined on
   $(\mathbb{R},\mathcal{B}(\mathbb{R})).$ Let $f(x) = (1-\exp(-x))_+$
   and $c(x,y) = \vert x - y \vert / (1 + \vert x - y \vert)$ for all
   $x,y \in \mathbb{R}.$ The primal problem of interest is
   $I = \sup\{ I(\pi): \pi \in \Phi_{\mu,\delta}\},$ where
   $I(\pi) := \int f(y)d\pi(x,y),$ for the choice $\delta = 2.$ We
   first argue that $I = 1$ here: As $f(x) \leq 1$ for every $x \in
   \mathbb{R},$ we have $I \leq 1.$ 
   For $\lambda \geq 0,$ recall the definition
   $\phi_{\lambda}(x) := \sup_{y \in \mathbb{R}}\{ f(y) - \lambda
   c(x,y)\}.$ As Assumptions (A1) and (A2) hold, an application of
   Theorem \ref{THM-STRONG-DUALITY}(a) results in,
   \begin{align*}
     I &= \inf_{\lambda \geq 0} \left\{ \lambda \delta + \int
         \phi_{\lambda}(x) d\mu(x)\right\} = \inf_{\lambda \geq 0} \big\{
         2\lambda + \phi_{\lambda}(0) \big\}\\ 
       & = \inf_{\lambda \geq  0} \left\{ 2\lambda + \sup_{y \geq 0} \left\{ 1 - e^{-y} -
         \lambda \frac{y}{1+y} \right\} \right\} \geq \sup_{y \geq 0}
         \inf_{\lambda \geq 0}  \left\{ 1 - e^{-y} + 2\lambda - \lambda
         \frac{y}{1+y} \right\}  = 1.
   \end{align*}
   Since we had already verified that $I \leq 1,$ the above lower
   bound results in $I = 1,$ with the infimum being attained at
   $\lambda^\ast = 0.$ However, as $f(x) < 1$ for every
   $x \in \mathbb{R},$ it is immediate that
   $I(\pi) = \int f(y)d\pi(x,y)$ is necessarily smaller than 1 for all
   $\pi \in \Phi_{\mu,\delta}.$ Therefore, there does not exist a
   primal optimizer $\pi^\ast \in \Phi_{\mu,\delta}$ such that
   $I(\pi^\ast) = I.$ A careful examination of Part (b) of Theorem
   \ref{THM-STRONG-DUALITY} also results in the same conclusion: A
   primal optimizer, if it exists, is concentrated on
   $\{(x,y): x \in S, y \in \arg \max\{ f(y) - \lambda^\ast c(x,y)
   \}\},$ which is empty in this example, because
   $\arg \max\{ f(y) - \lambda^\ast c(x,y)\} = \arg \max_{y \geq 0}\{
   1 - \exp(-y) - 0\times c(x,y)\} = \emptyset,$ for every $x \in S.$
   \hfill$\Box$
   \label{Eg-No-Primal-Opt}
 \end{example}

 Recall from Proposition \ref{PROP-DUALITY-COMPACT-2} that a primal
 optimizer always exists whenever the underlying Polish space $S$ is
 compact. In the absence of compactness of $S,$ Proposition
 \ref{Prop-Suff-Cond} below presents additional topological properties
 (\textsc{P-Compactness}) and (\textsc{P-USC}) under which a primal
 optimizer exists.  Corollary \ref{Cor-Suff-Cond-I}, that follows
 Proposition \ref{Prop-Suff-Cond}, discusses a simple set of
 sufficient conditions for which these properties listed in
 Proposition \ref{Prop-Suff-Cond} are easily verified.

 \subsubsection*{Additional notation.} Define
 $\Phi^\prime_{\mu,\delta} := \left\{ \pi \in \Phi_{\mu,\delta}: \pi
   \big((x,y) \in S \times S: f(x) \leq f(y)\big) = 1\right\}.$
 For every $\pi \in \Phi_{\mu,\delta},$ it follows from
 Lemma~\ref{Lem-Notn-Clar} in Appendix \ref{Appendix-Proofs} that
 there exists
 $\pi^\prime \in \Phi^\prime_{\mu,\delta}$ 
 such that $I(\pi^\prime) \geq I(\pi),$ whenever $I(\pi)$ is
 well-defined. As a result,
 \[I := \sup\{ I(\pi): \pi \in \Phi_{\mu,\delta}\} = \sup\{ I(\pi):
   \pi \in \Phi^\prime_{\mu,\delta}\}.\]
 
 Throughout this section, we assume that $(\lambda^\ast,
 \phi_{_{\lambda^\ast}}) \in
 \Lambda_{c,f}$ is a dual-optimal pair satisfying $I = J =
 J(\lambda^\ast, \phi_{_{\lambda^\ast}}) <
 \infty.$ As $J(\lambda^\ast, \phi_{_{\lambda^\ast}}) = \lambda^\ast
 \delta + \int \phi_{_{\lambda^\ast}}d\mu <
 \infty,$ it follows that $\phi_{_{\lambda^\ast}}(x) <
 \infty,$ $\mu-$almost surely.

\begin{proposition}
  Suppose that the Assumptions (A1) and (A2) are in
    force.  In addition, suppose that the functions $c(\cdot,\cdot)$
    and $f(\cdot)$ are such that the following properties are
    satisfied: 
 \begin{property}[P-Compactness]
    \textnormal{For any $\epsilon > 0,$ there exists a compact
    $K_\epsilon \subseteq S$ such that
    $\mu(K_\epsilon) > 1 - \epsilon,$ and the closure
    $\textnormal{cl}(\Gamma_{\epsilon})$ of the set,
    $\Gamma_{\epsilon} := \left\{ (x,y) \in K_\epsilon \times S : f(y)
      - \lambda^\ast c(x,y) \geq \phi_{_{\lambda^\ast}}(x) - \gamma
    \right\},$
    is compact, for some $\gamma \in (0,\infty).$}
  \end{property}
  \begin{property}[P-USC]
   \textnormal{ For every collection
    $\{\pi_n: n \geq 1\} \subseteq \Phi^\prime_{\mu,\delta}$ such that
    $\pi_n \Rightarrow \pi^\ast$ for some $\pi^\ast \in
    \Phi_{\mu,\delta},$ we have $\varlimsup_n I(\pi_n) \leq I(\pi^\ast).$}
  \end{property}
  Then there exists a primal optimizer
  $\pi^\ast \in \Phi_{\mu,\delta}$ satisfying
  $I(\pi^\ast) = I = J = J(\lambda^\ast,\phi_{_{\lambda^\ast}}).$
  \label{Prop-Suff-Cond}
\end{proposition}
It is well-known that upper semicontinuous functions defined on
compact sets attain their supremum.  As Property 1
(\textsc{P-Compactness}) and Property 2 (\textsc{P-USC}),
respectively, enforce a specific type of compactness and upper
semicontinuity requirements that are relevant to our setup, we use the
suggestive labels (\textsc{P-Compactness}) and (\textsc{P-USC}),
respectively, to identify Properties 1 and 2 throughout the rest of
this section.
\begin{proof}
  \noindent \textit{Proof of Proposition \ref{Prop-Suff-Cond}.}  As
  $I = \sup\{ I(\pi): \pi \in \Phi^\prime_{\mu,\delta}\},$ consider a
  collection $\{\pi_n : n \geq 1\} \subseteq \Phi^\prime_{\mu,\delta}$
  such that $I(\pi_n) \geq I - 1/n.$ We first show that Property
  (\textsc{P-Compactness}) guarantees the existence of a weakly
  convergent subsequence $(\pi_{n_k}: k \geq 1)$ of
  $(\pi_n : n \geq 1)$ satisfying $\pi_{n_k} \Rightarrow \pi^\ast,$
  for some $\pi^\ast \in \Phi^\prime_{\mu,\delta}.$

  \textit{Step 1 $($Verifying tightness of $(\pi_n: n \geq 1))$}: As
  $I(\pi_n) \geq I - n^{-1},$ it follows from \eqref{Inter-Eps-Opt}
  that 
  \begin{align}
    \int \left( \phi_{_{\lambda^\ast}}(x) - f(y) + \lambda^\ast
    c(x,y)\right)d\pi_n(x,y) \leq \frac{1}{n} \quad \text{ and } \quad
    \lambda^\ast \left( \delta -\int cd\pi_n \right) \leq
    \frac{1}{n}. 
    \label{Eps-Optimzers}
  \end{align}
  for every $n \geq 1.$ For any $a > 0,$ Markov's inequality results
  in,
  \begin{align*}
    \pi_n&\left( \big\{ (x,y) \in S \times S: f(y) -
           \lambda^\ast c(x,y) < 
           \phi_{_{\lambda^\ast}}(x) - a \big\}\right) \leq \frac{1}{a}\int \left( \phi_{_{\lambda^\ast}}(x) -
           f(y) + \lambda^\ast c(x,y)\right) d\pi_n(x,y),
  \end{align*}
  which is at most $1/(na).$ Then, given any $\epsilon > 0$ and for
  the choice $a = \gamma$ specified in (\textsc{P-Compactness}), it
  follows from union bound that
  $\pi_n((S \times S) \setminus \Gamma_{\epsilon/2})$ is at least,
  \begin{align*}
    \pi_n\big( \left\{ (x,y): x \notin K_{\epsilon/2}, y \in S
    \right\}\big) + \pi_n \left(\left\{ (x,y) \in K_{\epsilon/2}
    \times S: \ f(y) - \lambda^\ast c(x,y) <
    \phi_{_{\lambda^\ast}}(x) - \gamma \right\} \right) \leq 
      \frac{\epsilon}{2} + \frac{1}{n\gamma}, 
  \end{align*}
  because 
  $\pi_n( \left\{ (x,y): x \notin K_{\epsilon/2}, y \in S \right\}) =
  \mu (S \setminus K_{\epsilon/2}) \leq \epsilon/2.$ Consequently,
  $\pi_n(\Gamma_{\epsilon/2}) \geq 1 - \epsilon/2 - 1/(n\gamma).$ As
  the closure of $\Gamma_{\epsilon/2}$ is compact for the choice of
  $\gamma$ in (\textsc{P-Compactness}), we have
  $\text{cl}(\Gamma_{\epsilon/2})$ is a compact subset of $S \times S$
  for given $\epsilon > 0.$ Therefore, for all
  $n \geq 2\epsilon^{-1} \gamma^{-1},$ we obtain that
  $\pi_n(\text{cl}(\Gamma_{\epsilon})) \geq 1 - \epsilon$ for any
  given $\epsilon > 0,$ thus rendering that the collection
  $\{\pi_{n}: n \geq 1\}$ is tight.

  Then, as a consequence of Prokhorov's theorem, we have a subsequence
  $(\pi_{n_k}: k \geq 1)$ such that $\pi_{n_k} \Rightarrow \pi^\ast$
  for some $\pi^{\ast} \in P(S \times S).$ 

  \textit{Step 2 $($Verifying $\pi^\ast \in \Phi_{\mu,\delta})$}: As
  $\pi_{n_k}(\ \cdot \times S) = \mu(\cdot),$ we obtain
  $\pi^\ast(\cdot \times S) = \mu(\cdot)$ as a consequence of the weak
  convergence $\pi_{n_k} \Rightarrow \pi^\ast$ (see that
  $\int g(x)d\pi^\ast(x,y) = \lim_k \int g(x)d\pi_n(x,y) = \int
  g(x)d\mu(x)$ for all $g \in C_b(S)$). Further, as $c$ is a
  nonnegative lower semicontinuous function, it follows from a version
  of Fatou's lemma for weakly converging probabilities (see Theorem
  1.1 in \citet{feinberg2014fatou} and references therein) that
  $\delta \geq \varliminf_k \int cd\pi_{n_k} \geq \int cd\pi^\ast.$
  Therefore, $\pi^\ast \in \Phi_{\mu,\delta}.$

  \textit{Step 3 $($Verifying optimality of $\pi^\ast)$}: As Property (\textsc{P-USC}) guarantees that
  $\varlimsup_k I(\pi_{n_k}) \leq I(\pi^\ast)$, we obtain
  $I = \varlimsup_{k} I(\pi_{n_k}) \leq I(\pi^\ast),$ thus yielding
  $I(\pi^\ast) = I.$ The fact that
  $I = J = J(\lambda^\ast,\phi_{\lambda^\ast})$ follows from Theorem
  \ref{THM-STRONG-DUALITY}. 
\end{proof}
The following assumptions imposing certain growth conditions on the
functions $c(\cdot,\cdot)$ and $f(\cdot)$ provide a simple set of
sufficient conditions required to verify properties
(\textsc{P-Compactness}) and (\textsc{P-USC}) stated in Proposition
\ref{Prop-Suff-Cond}.  
    \begin{assumption}[A3]
      There exist a nondecreasing function
      $g: \mathbb{R}_+ \rightarrow \mathbb{R}_+$ satisfying
      $g(t) \uparrow \infty$ as $t \rightarrow \infty,$ and a positive
      constant $C$ such that $c(x,y) \geq g(\Vert x -y \Vert)$
      whenever $\Vert x - y \Vert > C.$
    \end{assumption}
        \begin{assumption}[A4]
      There exist an increasing function
      $h: \mathbb{R}_+ \rightarrow \mathbb{R}_+$ satisfying
      $h(t) \uparrow \infty$ as $t \rightarrow \infty,$ and a positive
      constant $K$such that
      $\sup_{x,y \in S} \frac{f(y) - f(x)}{1+ h(\Vert x - y \Vert)}
      \leq K.$ In addition, given $\epsilon > 0,$ there exists a
      positive constant $C_\epsilon$ such that 
      $f(y) - f(x) \leq \epsilon(1 + c(x,y)),$  for every $x,y \in S$
      such that $\Vert x - y \Vert > C_\epsilon.$ 
    \end{assumption}
    The growth condition imposed in Assumption (A4) is similar to the
    requirement that the growth rate parameter $\kappa,$ defined in
    \citet{Gao_Kleyweget}, be equal to 0.

\begin{corollary}
  {Let $S$ be locally compact when equipped with the topology induced
    by a norm $\Vert \cdot \Vert$ defined on $S.$ Suppose that the
    functions $c$ and $f$ satisfy Assumptions (A1) - (A4).  Then,
    whenever the dual optimal pair
    $(\lambda^\ast,\phi_{_{\lambda^\ast}})$ satisfying
    $J(\lambda^\ast, \phi_{_{\lambda^\ast}}) = J < \infty$ is such
    that $\lambda^\ast > 0,$ there exists a primal optimizer
    $\pi^\ast \in \Phi_{\mu,\delta}$ satisfying
    $I(\pi^\ast) = I = J = J(\lambda^\ast, \phi_{_{\lambda^\ast}}).$}
  \label{Cor-Suff-Cond-I}
\end{corollary}

\begin{proof}
  We verify the properties (\textsc{P-Compactness}) and
  (\textsc{P-USC}) in the statement of Proposition
  \ref{Prop-Suff-Cond} in order to establish the existence
  of a primal optimizer.\\
  \textit{Step 1 \textnormal{(}Verification of}
  (\textsc{P-Compactness})): As any Borel probability measure on a
  Polish space is tight, there exists a compact
  $K_\epsilon \subseteq S$ such that
  $\mu(K_\epsilon) \geq 1-\epsilon,$ for any $\epsilon > 0.$ Let
  $a_\epsilon := \sup_{x \in K_\epsilon} \Vert x \Vert.$ 
  Given $\epsilon^\prime > 0,$ it follows from Assumption (A4) that
  there exists $C_{\epsilon^\prime}$ large enough satisfying
  $f(y) - f(x) \leq \epsilon^\prime c(x,y)$ for all $x,y$ such that
  $\Vert y - x \Vert > C_{\epsilon^\prime}.$ Then, for any
  $x \in K_\epsilon$ and $\gamma > 0,$
  $f(y) - \lambda^\ast c(x,y) < f(x) - \gamma$ when
  $\Vert y - x \Vert > C \vee C_{\epsilon^\prime}$ is large enough
  such that
  $g(\Vert x - y \Vert) \geq \gamma/(\lambda^\ast - \epsilon^\prime);$
  this is because,
  $f(y) - f(x) -\lambda^\ast c(x,y) < (\epsilon^\prime -
  \lambda^\ast)c(x,y) \leq (\epsilon^\prime - \lambda^\ast)g(\Vert x -
  y \Vert)$ when $\epsilon^\prime < \lambda^\ast$ and
  $ \Vert x - y \Vert > C \vee C_{\epsilon^\prime}.$ In particular, if
  $\Vert y \Vert > c_\epsilon := C \vee C_{\epsilon^\prime} +
  a_\epsilon + g^{-1}(\gamma/(\lambda^\ast - \epsilon^\prime)),$ we
  have $f(y) - \lambda^\ast c(x,y) < f(x) - \gamma$ for every
  $x \in K_\epsilon.$ As $\phi_{_{\lambda^\ast}}(x) \geq f(x),$ it
  follows that,
    \begin{align*}
    \Gamma_{\epsilon} := \left\{ (x,y) \in K_\epsilon \times S : f(y)
    - \lambda^\ast c(x,y) \geq \phi_{_{\lambda^\ast}}(x) - \gamma
    \right\} \subseteq \left\{ (x,y): x \in K_\epsilon, \Vert y \Vert
    \leq 
    c_\epsilon\right\}, 
  \end{align*}
  is a compact subset of $S \times S.$ Therefore,
  $\text{cl}(\Gamma_\epsilon)$ is compact as well, thus verifying
  Assumption (A3) in the statement of Proposition
  \ref{Prop-Suff-Cond}.\\
  \textit{Step 2 \textnormal{(}Verification of } \textsc{P-USC}). Let
  $\{\pi_n: n \geq 1\} \subseteq \Phi^\prime_{\mu,\delta}$ be such
  that $\pi_n \Rightarrow \pi^\ast$ for some
  $\pi^\ast \in \Phi_{\mu,\delta}.$ Our objective is to show that
\begin{align}
  \varlimsup_n I(\pi_n) =
  \varlimsup_n \int f(y)d\pi_n(x,y) \leq \int f(y)d\pi^\ast(x,y) =
  I(\pi^\ast)
  \label{To-Show}
\end{align}
While \eqref{To-Show} follows directly from the properties of weak
convergence when $f$ is bounded upper semicontinuous, an additional
asymptotic uniform integrability condition that,
\[\varlimsup_{M \rightarrow \infty} \sup_n \int_{\vert f(y) \vert > M}
\vert f(y) \vert d\pi_n(x,y) = 0,\] is sufficient to guarantee
\eqref{To-Show} in the absence of boundedness (see, for example,
Corollary 3 in \citet{Zapała2008698}). In order to demonstrate this
asymptotic uniform integrability property, we proceed as follows:
Given $\epsilon > 0,$ Assumption (A4) guarantees that
$f(y) - f(x) \leq \epsilon c(x,y)$ for every $x,y$ satisfying
$\Vert x - y \Vert > C_\epsilon.$ 
  Let $A_1, A_2$ and $B_{(M)}$ be subsets of $S \times S$ defined as
  follows:
    \begin{align*}
      A_1 := \{ (x,y): \Vert x - y \Vert > C_\epsilon\}, \  A_2 :=
        (S \times S) \setminus A_1, \text{ and }
      B_{(M)} := \{ (x,y) : \vert f(y) \vert > M\}  \text{ for  } M > 0. 
  \end{align*}
  As $\pi_n((x,y):f(x) \leq f(y)) = 1$ for every $n$ (recall that
  $\pi_n \in \tilde{\Phi}_{\mu,\delta}),$ it follows from the growth
  conditions in Assumption (A4) that
  \begin{align*}
    \int \vert f(y) - f(x) \vert d\pi_n(x,y) \leq \epsilon
    \int_{A_1}(1 + c(x,y)) d\pi_n(x,y) + K 
    \int_{A_2} (1+ h(\Vert x - y \Vert)d\pi_n(x,y).
  \end{align*}
  Further, as any $(x,y) \in A_2$ satisfies
  $\Vert x -y \Vert \leq C_\epsilon,$ it follows from the
  nondecreasing nature of $h(\cdot)$ that
  $h(\Vert x - y\Vert) \leq h(C_\epsilon) < \infty,$ for every for
  every $(x,y) \in A_2.$ Combining this observation with the fact that
  $\int cd\pi_n \leq \delta,$ we obtain,
  \begin{align*}
    \int_{B_{(M)}} \vert f(y) - f(x) \vert d\pi_n(x,y) \leq 
    \epsilon (1 + \delta) + K 
    (1+h(C_\epsilon)) \pi_n(A_3 \cap B_{(M)}).
  \end{align*}
  As $\vert f \vert = f + 2f^-,$ it follows from Markov's inequality
  that, $\pi_n(B_{(M)}) = \pi_n({y: \vert f(y) \vert > M})$ is at
  most,
  \[\frac{1}{M}\int \vert f(y) \vert
    d\pi_n(x,y) = \frac{1}{M} \left( \int f(y)d\pi_n(x,y) + 2\int
      f^-(y)d\pi_n(x,y)\right).\] Since
  $\pi_n((x,y): f(x) \leq f(y))=1,$ we have $f^-(x) \geq f^-(y),$
  $\pi_n$-almost surely, for every $n.$ Consequently,
  $\int f^-(y)d\pi_n(x,y) \leq \int f^-(x)d\pi_n(x,y) = \int
  f^-(x)d\mu(x),$ and therefore,
  \begin{align*}
    \sup_n \pi_n \left(B_{(M)} \right) \leq \frac{1}{M}\left(I + 2\int
    f^-d\mu\right),
  \end{align*}
 where $I = J < \infty,$ and
  $\int f^-d\mu < \infty.$ As a result, we obtain
\begin{align}
  \sup_n \int_{B_{(M)}} \vert f(y) - f(x) \vert d\pi_n(x,y) \leq 
  \epsilon (1 + \delta) + \frac{K}{M} 
  (1+h(C_\epsilon))\left( I + 2\int f^-d\mu\right).
\label{Inter-UI-UB}
\end{align}
Further, as
$\int \vert f(x) \vert d\pi_n(x,y) = \int \vert f(x) \vert d\mu(x)$ is
finite and $\sup_n \pi_n(B_{(M)}) \leq M^{-1}(I + 2\int f^-d\mu)
\rightarrow 0$ when 
$M \rightarrow \infty,$ we have
$\int_{B_{(M)}} \vert f(x) \vert d\mu(x) \rightarrow 0$ as
$M \rightarrow \infty.$ Consequently, letting $M \rightarrow \infty$
and $\epsilon \rightarrow 0$ in \eqref{Inter-UI-UB}, we obtain
\begin{align*}
  \varlimsup_{M \rightarrow \infty} \sup_n \int_{\{\vert f(y) \vert >
  M\}} \hspace{-30pt} \vert f(y) 
  \vert d\pi_n(x,y)  \leq  \varlimsup_{M \rightarrow \infty} \sup_n \int_{B_{(M)}} \vert f(y) - f(x)
  \vert d\pi_n(x,y)  + \varlimsup_{M \rightarrow \infty}
  \int_{B_{(M)}} \vert f(x) \vert d\mu(x) = 0,
\end{align*}
thus verifying the desired uniform integrability property. This
observation, in conjunction with the upper semicontinuity of $f$ and
weak convergence $\pi_{n} \Rightarrow \pi^\ast,$ results in
\eqref{To-Show}.
As all the requirements stated in Proposition \ref{Prop-Suff-Cond} are
verified, a primal optimizer satisfying $I(\pi^\ast) = I$
exists. 
\end{proof}

\begin{remark}
  \textnormal{ In addition to the assumptions made in Corollary
    \ref{Cor-Suff-Cond-I}, suppose that $c(x,y)$ is a convex function
    in $y$ for every $x \in S,$ and $f$ is a concave function. Then,
    due to Corollary \ref{Cor-Suff-Cond-I}, a primal optimizer
    $\pi^\ast \in \Phi_{\mu,\delta}$ exists. Further, as the supremum
    in $\sup_{y \in S} \{f(y) - \lambda^\ast c(x,y)\},$ is attained at
    a unique maximizer for every $x \in S,$ it follows from Remark
    \ref{Rem-Uniqueness-Primal-Opt} that the primal optimizer
    $\pi^\ast$ is unique.} 
  \label{Rem-Uniqueness-Primal-Opt-II}
\end{remark}
\vspace{-20pt}

\section{A few more examples.}
\label{Sec-App}
\subsection{Applications to computing general first passage
  probabilities.}
\label{Sec-App-FPP}
Computing probabilities of first passage of a stochastic process into
a target set of interest has been one of the central problems in
applied
probability. 
The objective of this section is to demonstrate that, similar to the
one dimensional level crossing example in Section \ref{Sec-App-LC},
one can compute general worst-case probabilities of first passage into
a target set $B$ by simply computing the probability of first passage
of the baseline stochastic process into a suitably inflated
neighborhood of set $B.$ The goal of such a demonstration is to show
that the worst-case first passage probabilities can be computed with
no significant extra effort.

\begin{example}
\label{Eg-2D-FPP}
  \textnormal{Let ${R}(t) = (R_1(t),R_2(t)) \in \mathbb{R}^2$ model
    the financial reserve, at time $t,$ of an insurance firm with two
    lines of business. Ruin occurs if the reserve process ${R}_t$ hits
    a certain ruin set $B$ within time $T > 0.$ In the univariate
    case, the ruin set is usually an interval of the form $(-\infty,0)$ or
    its translations (as in Example
    \ref{Eg-Ruin-prob-BM-approx}). However, in multivariate cases, the
    ruin set $B$ can take a variety of shapes based on rules of
    capital transfers between the different lines of businesses. For
    example, if capital can be transferred between the two lines
    without any restrictions, a natural choice is to declare ruin when
    the total reserve $R_1(t) + R_2(t)$ becomes negative. On the other
    hand, if no capital transfer is allowed between the two lines,
    ruin is declared immediately when either $R_1(t)$ or $R_2(t)$
    becomes negative. In this example, let us consider the case where
    the regulatory requirements allow only $\beta \in [0,1]$ fraction
    of reserve, if positive, to be transferred from one line of
    business to the other.  Such a restriction will result in a ruin
    set of the form
    \begin{align*}
      B := \left\{ (x_1,x_2) \in \mathbb{R}^2: \beta x_1 + x_2 \leq 0
      \right\} \cup \left\{ (x_1,x_2) \in \mathbb{R}^2: \beta x_2 +
      x_1 \leq 0 \right\}. 
    \end{align*}
    It is immediately clear that capital transfer is completely
    unrestricted when $\beta = 1$ and altogether prohibited when
    $\beta = 0.$ The intermediate values of $\beta \in (0,1)$ softly
    interpolates between these two extreme cases.  Of course, one can
    take the fraction of capital transfer allowed from line $1$ to
    line $2$ to be different from that allowed from line $2$ to line
    1, and various other modifications. However, to keep the
    discussion simple we focus on the model described above, and refer
    the readers to \citet{hult2006heavy} for a general specification of
    ruin models allowing different rules of capital transfers between
    $d$ lines of businesses.  }

  \textnormal{As in Example \ref{Eg-Ruin-prob-BM-approx}, we take the
    space in which the reserve process $({R}(t): 0 \leq t \leq T)$
    takes values to be $S = D([0,T], \mathbb{R}^2),$ the set of all
    $\mathbb{R}^2$-valued right continuous functions with left limits
    defined on the interval $[0,T].$ The space $S,$ again, as in
    Example \ref{Eg-Ruin-prob-BM-approx}, is equipped with the
    standard $J_1$-metric (see Chapter 3 of \citet{MR1876437}), and
    consequently, the cost function
  \begin{equation}
    c(x,y) = \inf_{\lambda \in \Lambda} \left(\sup_{t \in [0,T]} \vert \vert
    x(t) - y(\lambda(t)) \vert
    \vert_2^2 + \sup_{t \in [0,T]} \vert \lambda(t) - t \vert \right),
    \quad x,y \in S 
    \label{Cost-2dim-Ruin}
  \end{equation}
  is continuous. Here, $||x||_2 = (|x_1|^2 + |x_2|^2|)^{1/2}$ is the
  standard Euclidean norm for $x = (x_1, x_2)$ in $\mathbb{R}^2,$ and
  $\Lambda$ is the set of all strictly increasing functions
  $\lambda:[0,T] \rightarrow [0,T]$ such that both $\lambda$ and
  $\lambda^{-1}$ are continuous. Let $\mu \in \mathcal{P}(S)$ denote a
  baseline probability measure that models the reserve process
  ${R}(t)$ in the path space. {Given $\delta > 0,$ our
    objective is to characterize the worst-case first passage
    probability,
    $\sup\{ P(x(t) \in B \text{ for some } t \in [0,T]): d_c(\mu,P)
    \leq \delta\},$ of the reserve process hitting the ruin set $B.$
As the set $\{ x \in S: x(t)
\in B \text{ for some } t \in [0,T]\}$ is not closed, we consider its
topological closure  
\begin{align*}
  A := \left\{ x \in S: \inf_{t \in [0,T]} (\beta x_1(t) + x_2(t)) \leq
  0\right\} \cup \left\{ x \in S: \inf_{t \in [0,T]} (x_1(t) + \beta
  x_2(t)) \leq 0\right\},
\end{align*}
which, apart from the paths that hit ruin set $B,$ also contains the
paths that come arbitrarily close to $B$ without hitting $B$ due to
the presence of jumps; the fact that $A$ is the desired closure is
verified in Lemma \ref{Lem-2D-Closure} and Corollary
\ref{Cor-2D-Closure} in Appendix \ref{Appendix-Proofs}. Further, it is
verified in Lemma \ref{Lem-Sym-2dim-Ruin} in Appendix
\ref{Appendix-Proofs} that
  \begin{align*}
    c(x,A) = \frac{1}{1+\beta^2}\left[\inf_{t \in [0,T]} \big( \beta
    x_1(t) + x_2(t) \big)^2 \wedge \inf_{t \in [0,T]} \big( \beta x_2(t) + x_1(t)
    \big)^2\right], 
  \end{align*} 
  If the reference distribution $\mu(\cdot)$ is such that $h(u) :=
  E_\mu[c(X,A); c(X,A) \leq u]$ is continuous, then 
  \begin{align}
    \sup\{ P(x(t) \in B \text{ for some } t \in [0,T]): d_c(\mu,P)
    \leq \delta\} &\leq \sup \big\{  P(A)  : d_c(\mu,P) \leq
                    \delta \big\} \label{Closure-Inequality}\\
                  &= \mu \left\{ x \in S: c(x,A) \leq
                    \frac{1}{\lambda^\ast}\right\},  \nonumber
  \end{align}
  as an application of Theorem \ref{Thm-Prob-Reform};
  here, as in Section \ref{Eg-Rob-Probabilities}, 
  $1/\lambda^\ast = h^{-1}(\delta) := \inf\{u \geq 0: h(u)
  \geq \delta\}.$
  Following the expression for $c(x,A)$ in Lemma
  \ref{Lem-Sym-2dim-Ruin}, we also have that 
  $\{ x \in S: c(x,A) \leq {1}/{\lambda^\ast}\}$ equals
  \begin{align}
    \left\{ x \in S: \inf_{t \in [0,T]}(\beta x_1 (t) + x_2(t)) \leq
    \sqrt{\frac{1+\beta^2}{{\lambda^\ast}}}\right\} \cup    \left\{ x \in S:
    \inf_{t \in [0,T]}(x_1 (t) + \beta x_2(t)) \leq 
    \sqrt{\frac{1+\beta^2}{{\lambda^\ast}}}\right\}. 
    \label{2D-Soln}
   \end{align}
   Next, if we let $c^\ast = \sqrt{(1+\beta^2)/\lambda^\ast}$ and
   \begin{align}
     A_{(c)} = \left\{ x \in S: \inf_{t \in [0,T]}(\beta x_1 (t) +
     x_2(t)) \leq  c \right\} \cup    \left\{ x \in S: \inf_{t \in [0,T]}(x_1
     (t) + \beta x_2(t)) \leq c\right\} 
     \label{Defn-Ruin-Paths}
   \end{align}
   as a family of sets parameterized by $c \in \mathbb{R},$ then it
   follows from \eqref{2D-Soln} that 
   \begin{align}
     \sup\{P\left( A_{(0)}\right): d_c(\mu,P) \leq \delta\} =
     \mu\left( A_{\left(c^\ast\right)}\right);
     \label{2dim-Wcaseprob-Char}
   \end{align}
   in words, the worst-case probability that the reserve process hits
   (or) comes arbitrarily close to the ruin set $B$ is simply equal to
   the probability under reference measure $\mu$ that the reserve
   process hits (or) comes arbitrarily close to an inflated ruin set
   $B_{(c^\ast)},$ where for any $c \in \mathbb{R}$ we define 
   \begin{align*}
     B_{(c)} := \left\{ (x_1,x_2) \in \mathbb{R}^2: \beta x_1 + x_2 \leq c
    \right\} \cup \left\{ (x_1,x_2) \in \mathbb{R}^2: \beta x_2 +
    x_1  \leq c \right\}.
   \end{align*}
   This conclusion is very similar to \eqref{Ruin-Prob-1dim-shift}
   derived for 1-dimensional level crossing in
   Section~\ref{Sec-App-LC}.  
  The original ruin set $B = B_{(0)}$ and the suitably inflated ruin
  set $B_{(c^\ast)}$ are respectively shown in the Figures
  \ref{Fig-2dim-Sym-Ruin}(a) and (b).}  \setcounter{subfigure}{0}
  \begin{figure}[h!]
    \caption{Comparison of computation of ruin under baseline measure
      (in Fig(a)) and worst-case ruin (in Fig(b))}
    \begin{center}
      \subfigure[Computation of ruin under baseline
      measure]{\includegraphics[scale=0.4, page=1]{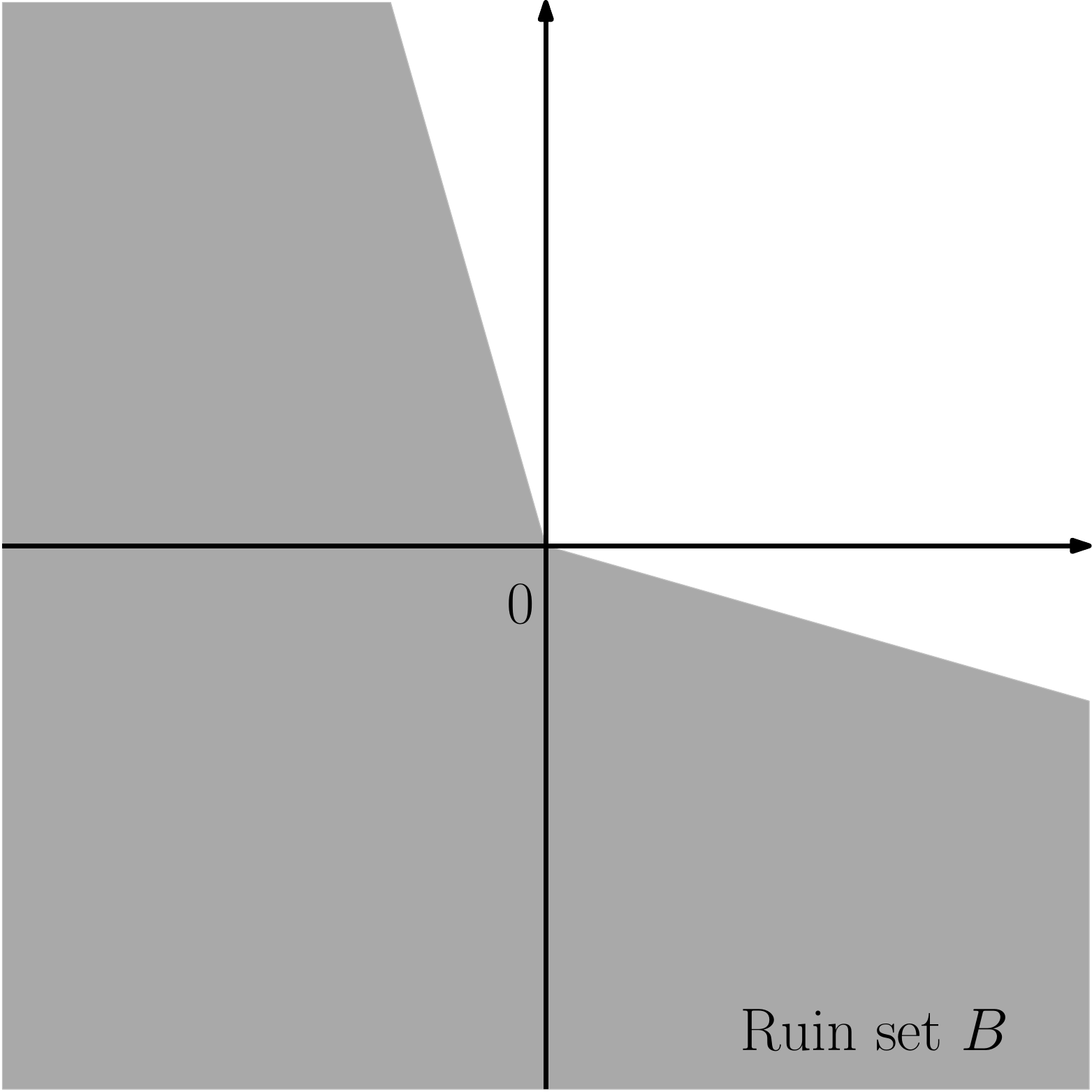}}
      \hspace{50pt} \subfigure[Computation of worst-case ruin using
      the baseline measure]{\includegraphics[scale=0.4,
        page=2]{Sym_2dim_Ruin.pdf}}
   \end{center}
  \label{Fig-2dim-Sym-Ruin}
  \end{figure}} 

\textnormal{Next, as an example, let us take the reserve process
  satisfying the following dynamics as our baseline model:
  \begin{align*}
    d{R}(t) = u{b} + m dt + \Sigma d{B}(t),
  \end{align*}
  where $m \in \mathbb{R}_-^2$ is the drift vector with negative
  components, $\Sigma$ is a 2 $\times$ 2 positive definite covariance
  matrix, and $({B}(t): 0 \leq t \leq T)$ a 2-dimensional standard
  Brownian motion. The real number $u$ denotes the total initial
  capital and ${b} = (b_1, b_2) \in \mathbb{R}^2_+$ such that
  $b_1 + b_2 = 1$ denotes the proportion of initial capital set aside
  for the two different lines of businesses. The probability measure
  induced in the path space by the process ${R}(t)$ is taken as the
  baseline measure $\mu.$ Further, for purposes of numerical
  illustration, we take
  $m = [-0.1, -0.1], {b} = [0.5, 0.5] \text{ and } \Sigma = I_2,$ the
  $2 \times 2$ identity matrix. Our aim is to find total initial
  capital $u$ such that
  $\sup\{P(A_{(0)}): d_c(\mu,P) \leq \delta\} \leq 0.01$ (recall the
  definition of the ruin event $A_{(0)}$ in \ref{Defn-Ruin-Paths}).
  This is indeed possible due to the equivalent characterization in
  \eqref{2dim-Wcaseprob-Char}, and the resulting capital requirement
  for various values of $\delta$ are displayed in Figure
  \ref{Fig-Cap-Req-2dim-Ruin}.}
  \begin{figure}[h!]
    \centering
    \caption{Capital requirement for various values of $\delta.$ The
      capital requirement is calculated to keep the worst-case
      probability of ruin under 0.01}
    \includegraphics[width = 9cm, height = 6cm]{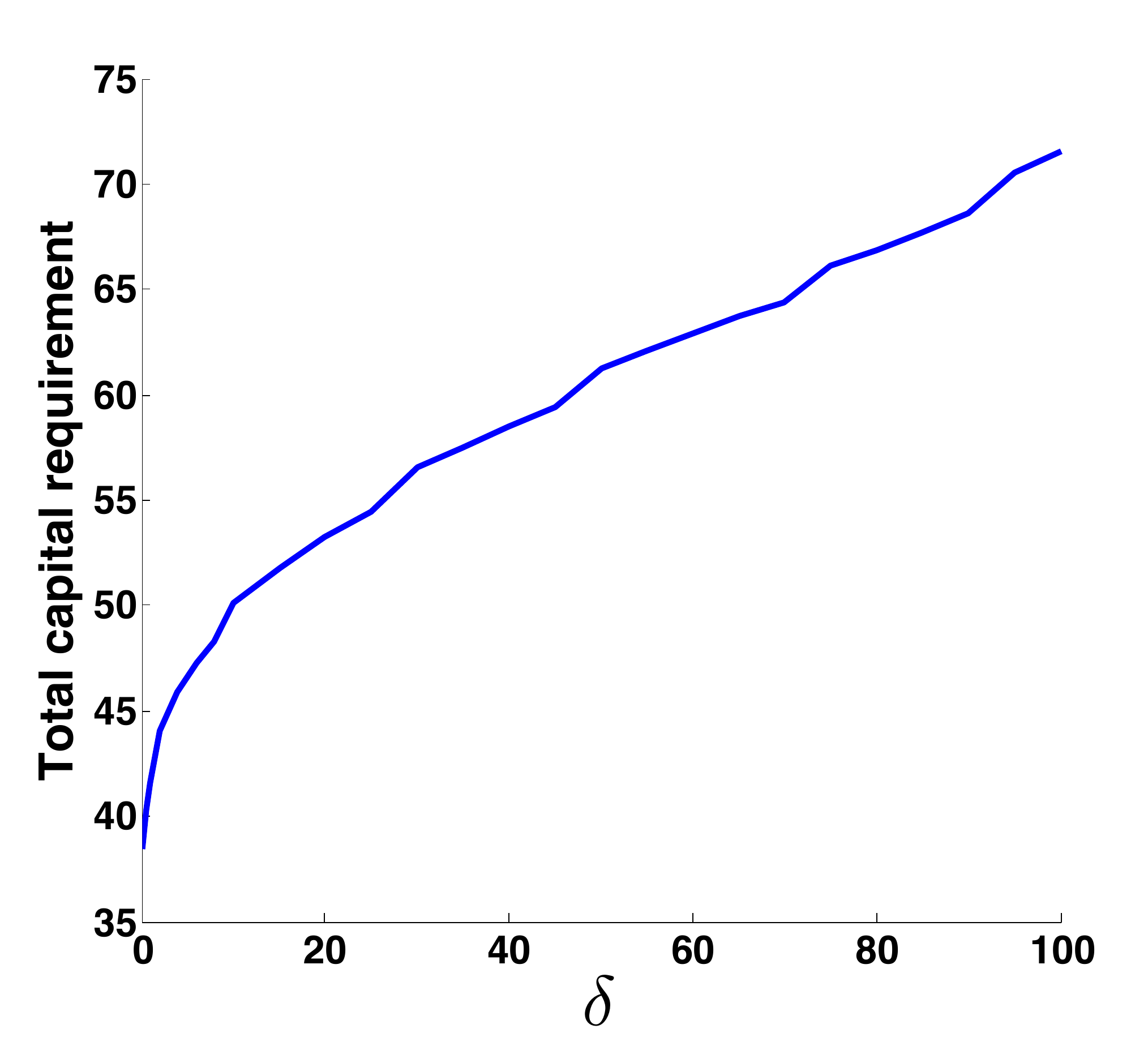}
    \label{Fig-Cap-Req-2dim-Ruin}
  \end{figure}

  \textnormal{{It may also be useful to note that one
      may as well choose $S = C([0,T],\mathbb{R}^2),$ the space of
      continuous functions taking values in $\mathbb{R}^2,$ as the
      underlying space to work with if the modeler decides to restrict
      the distributional ambiguities to continuous stochastic process;
      in that case, the first passage set
      $\{x \in S: x(t) \in B \text{ for some } t \in [0,T]\}$ itself
      is closed, and the inequality in \eqref{Closure-Inequality}
      holds with equality.}}


  \textnormal{As a final remark, if the modeler believes that model
    ambiguity is more prevalent in one line of business over other,
    she can perhaps quantify that effect by instead choosing
    $\vert \vert x \vert \vert_2 = (\vert x_1 \vert^2 + \alpha \vert
    x_2 \vert^2)^{1/2},$
    for $x = (x_1,x_2) \in \mathbb{R}^2$ in the cost function $c(x,y)$
    defined in \eqref{Cost-2dim-Ruin}. This would penalise moving
    probability mass in one direction more than the other and result
    in the ruin set being inflated asymmetrically along different
    directions. We identify studying the effects of various choices of
    cost functions, the corresponding $\delta$ and the appropriateness
    of the resulting inflated ruin set for real world ruin problems as
    an important direction towards applying the proposed framework in
    quantitative risk management. For example, suppose that the
    regulator has interacted with the insurance company, and both the
    company and the regulator have negotiated a certain level of
    capital requirement multiple times in the past. Then the insurance
    company can calibrate $\delta$ from these previous interactions
    with the regulator. That is, find the value of $\delta$ which
    implies that the negotiated capital requirement in a given past
    interaction is necessary for the bound on ruin probability to be
    lesser than a fixed (regulatory driven) acceptance level (for
    instance, we use 0.01 as the acceptance level for ruin probability
    in our earlier numerical illustration in this example). An
    appropriate quantile of these `implied' $\delta$ values may be
    used to choose $\delta$ based on risk preference.}
\end{example}

\subsection{Applications to ambiguity-averse decision making.}
\label{Sec-App-St-Opt}
In this section, we consider a stochastic optimization problem in the
presence of model uncertainty. It has been of immense interest
recently to search for distributionally robust optimal decisions, that
is, to find a decision variable $b$ that solves
\begin{align*}
  \textnormal{OPT} = \inf_{b \in B} \sup_{P \in \mathcal{P}} E \left[ f(X,b) \right].
\end{align*}
Here, $f$ is a performance/risk measure that depends on a random
element $X$ and a decision variable $b$ that can be chosen from an
action space $B.$ The solution to the above problem minimizes
worst-case risk over a family of ambiguous probability measures
$\mathcal{P}.$ Such an ambiguity-averse optimal choice is also
referred as a distributionally robust choice, because the performance
of the chosen decision variable is guaranteed to be better than OPT
irrespective of the model picked from the family $\mathcal{P}.$

There has been a broad range of ambiguity sets $\mathcal{P}$ that has
been considered: For examples, refer \citet{MR2680566, MR2683483,
  MR3294550} for moment-based uncertainty sets, \citet{Hans_Sarg,
  iyengar2005robust, NE:05, lim2007relative, Jain2010, Ben_Tal,
  Wang2015, Jiang2015, hu2012kullback, doi:10.1287/educ.2015.0134} for
KL-divergence and other likelihood based uncertainty sets,
\citet{MR2354780, MR2874755, esfahani2015data, Zhao_Guan,
  Gao_Kleyweget} for Wasserstein distance 
based neighborhoods, \citet{MR2216800} for neighborhoods based on
Prokhorov metric, \citet{Bandi_rob_queues, MR3247335} for uncertainty
sets based on statistical tests and \citet{MR2546839, MR2066239} for a
general overview. As most of the works mentioned above assume the
random element $X$ to be $\mathbb{R}^d$-valued, it is of our interest
in the following example to demonstrate the usefulness of our
framework in formulating and solving distributionally robust
optimization problems that involve stochastic processes taking values
in general Polish spaces as well.

\begin{example}
  \label{Eg-Choosing-Reinsurance-Prop}
  \textnormal{We continue with the insurance toy example considered in
    Section \ref{Sec-App-LC}. In practice, as the risk left to the
    first-line insurer is too large, reinsurance is usually
    adopted. In proportional reinsurance, one of the popular forms of
    reinsurance, the insurer pays only for a proportion $b$ of the
    incoming claims, and a reinsurer pays for the remaining $1-b$
    fraction of all the claims received. In turn, the reinsurer
    receives a premium at rate $p_r = (1+\theta)(1-b)\nu cm_1$ from
    the insurer. Here, $\theta > \eta,$ otherwise, the insurer could
    make riskless profit by reinsuring the whole portfolio. The
    problem we consider here is to find the reinsurance proportion $b$
    that minimises the expected maximum loss that happens within
    duration $T.$ In the extensive line of research that studies
    optimal reinsurance proportion, diffusion models have been
    particularly recommended for tractability reasons (see, for
    example \citet{Hojgaard_Taksar,schmidli2001}). As in Example
    \ref{Eg-Ruin-prob-BM-approx}, if we take
    $\sqrt{\nu m_2} B(t) + \nu m_1 t$ to be the diffusion process that
    approximates the arrival of claims, then
    \begin{align*}
      L_b(t) &:= p_rt + b\big(\sqrt{\nu m_2} B(t) + \nu m_1 t \big) -
               pt \\
             &=  b\sqrt{\nu  m_2} B(t)- \big( b\theta - (\theta - \eta) \big) \nu
               m_1t  
    \end{align*}
    is a suitable model for losses made by the firm. Here,
    $p_r = (1 + \theta) (1-b) \nu m_1$ is the rate of payout for the
    reinsurance contract, and $p = (1 + \eta) \nu m_1$ is the rate at
    which a premium income is received by the insurance firm. The
    quantity of interest is to determine the reinsurance proportion
    $b$ that minimises the maximum expected losses,
    \begin{equation}
      L := \inf_{b \in [0,1]}  E \left[ \max_{t \in [0,T]}  L_b(t)
      \right].
     \label{Nonrob-Opt-Losses}
    \end{equation}
    However, as we saw in Example \ref{Eg-Ruin-prob-BM-approx},
    conclusions based on diffusion approximations can be
    misleading. Following the practice advocated by the rich
    literature of robust optimization, we instead find a reinsurance
    proportion $b$ that performs well against the family of models
    specified by $\mathcal{P} := \{P: d_c(\mu,P) \leq \delta\}.$ In
    other words, we attempt to solve for
    \begin{align*}
      L' := \inf_{b \in [0,1]} \sup_{P \in \mathcal{P}} E_P\left[ \sup_{t
      \in [0,T]}  \big( a_2(b)  X(t) - a_1(b)t  \big) \right],
    \end{align*}
    where $X$ is a random element in space $D[0,T]$ following measure
    $P,$ $a_2(b) := b\sqrt{\nu m_2}$ and
    $a_1(b) :=(b\theta - (\theta - \eta)) \nu m_1.$ Here, as in
    Section \ref{Eg-Rob-Probabilities}, we have taken $S = D[0,T]$ and
    \begin{align*}
      c(x,y) = \left(d_{_{J_1}}(x,y)\right)^2.
    \end{align*}
    For $x \in D[0,T],$ if we take
    \begin{align*}
      f(x,b) := \sup_{t \in [0,T]}\big(a_2(b)x(t) - a_1(b)t \big) \quad
               \text{ and } \quad 
      \phi_{\lambda,b}(x) := \sup_{y \in S} \left\{ f(y,b) - \lambda
      c(x,y) \right\},
    \end{align*}
    then due to the application of Theorem \ref{THM-STRONG-DUALITY},
    we obtain
    \begin{align*}
      L' = \inf_{b \in [0,1]} \sup\left\{ E_P[f(X,b)]: d_c(\mu,P) \leq
      \delta \right\}
      = \inf_{b \in [0,1]} \inf _{\lambda \geq 0} \left\{ \lambda \delta + \int
      \phi_{\lambda,b}(x) d\mu(x) \right\},
    \end{align*}
    To keep this discussion terse, it is verified in Lemma
    \ref{Lem-Exp-Max-BM-App} in Appendix \ref{Appendix-Proofs} that the
    inner infimum evaluates simply to
    \begin{align*}
      E \left[ \sup_{t \in [0,T]} \big( a_2(b)B(t) - a_1(b) t \big)
      \right] + a_2(b) \sqrt{\delta}.
    \end{align*}
   As a result, 
   \begin{align}
     L' = \inf_{b \in [0,1]} \left\{  E \left[ \sup_{t \in [0,T]}
     \big( a_2(b)B(t) - a_1(b) t \big)
     \right] +  a_2(b) \sqrt{\delta} \right\},
     \label{Rob-Opt-Losses}
   \end{align}
   which is an optimization problem that involves the same effective
   computational effort as the non-robust counterpart in
   \eqref{Nonrob-Opt-Losses}. For the specific numerical values
   employed in Example \ref{Eg-Ruin-prob-BM-approx}, if we
   additionally take the new parameter $\theta = 0.3,$ the Brownian
   approximation model evaluates to the optimal choice $b = 0.66$ and
   the corresponding loss $L = 17.63,$ whereas the robust counterpart
   in \eqref{Rob-Opt-Losses} evaluates to worst-case $L' = 28.86$ for
   the ambiguity-averse optimal choice $b = 0.42.$ For a collection of
   various other examples of using Wasserstein based ambiguity sets in
   the context of distributionally robust optimization, refer
   \citet{esfahani2015data, Zhao_Guan} and \citet{Gao_Kleyweget}.}
\end{example}

\section*{Acknowledgments.}
The authors would like to thank the anonymous referees whose valuable
suggestions have been immensely useful in supplementing the proof of
the strong duality theorem with crisp arguments at various
instances. The authors would also like to thank Garud Iyengar and
David Goldberg for helpful discussions, and Rui Gao for providing a
comment that led us to add Remark \ref{Rem-Primal-Opt-Cbar-Not-equal}
in the paper. The authors gratefully acknowledge support from Norges
Bank Investment Management and NSF grant CMMI 1436700.

\bibliography{Transport}

\appendix 
\section{Brownian embeddings for the estimation of $\delta$ in Example
  \ref{Eg-Ruin-prob-BM-approx}}
\label{Appendix-Embeddings}
\noindent
Recall that the definition of optimal transport cost between two
probability measures $\mu$ and $\nu,$ denoted by $d_c(\mu,\nu),$
involves computing the minimum expected cost over all possible
couplings between $\mu$ and $\nu$ (see Section \ref{Sec-Notn-Assump}
for a definition). Though it may not always be possible to identify
the optimal coupling (with the lowest expected cost) between $\mu$ and
$\nu,$ one can perhaps employ a `good' coupling to derive an upper
bound on $d_c(\mu,\nu).$ In this section, we describe one such
coupling, popularly referred as Skorokhod embedding, between the risk
process $R(t)$ and its Brownian motion based diffusion approximation
$R_B(t)$ in Example \ref{Eg-Ruin-prob-BM-approx}. More specifically,
we `embed' the compensated compound Poisson process
\[Z(t) = \frac{1}{\sqrt{m_2}} \left(\sum_{i = 1}^{N_t} X_i - m_1 t
\right)\]
in a Brownian motion $B(t)$ to obtain a coupling between the risk
processes $R(t)$ and $R_B(t)$ in order to choose a $\delta$ in Example
\ref{Eg-Ruin-prob-BM-approx}.  Here, the symbols $m_1$ and $m_2$
denote the first and second moments of claim sizes $X_i,$ and $N_t$ is
a unit rate Poisson process.  Please refer Example
\ref{Eg-Ruin-prob-BM-approx} in Section \ref{Sec-App-LC} for a
thorough review of notations. The procedure is data-driven in the
sense that we do not assume the knowledge of the distribution of claim
sizes $X_i.$ Instead, we simply assume access to an oracle that
provides independent realizations of claim sizes $\{X_1,X_2,\ldots\}.$
Given this access to claim size information, Algorithm
\ref{Algo-Skorokhod-Embedding-CPP-BM} below specifies the coupling
that embeds the process $Z(t)$ in Brownian motion $B(t).$

\begin{algorithm}[h!]
  \caption{To embed the process $(Z(t) : t \geq 0)$ in Brownian motion
    $(B(t): t \geq 0)$\newline Given: Brownian motion $B(t),$ moment
    $m_1$ and independent realizations of claim sizes
    $X_1,X_2,\ldots$}
    \begin{algorithmic}
      \\
      \State Initialize $\tau_0 := 0$ and $\Psi_0 := 0.$ For
      $j \geq 1,$ recursively  define,
      \[\tau_{j+1} := \inf \bigg\{ s \geq \tau_j: \sup_{\tau_j \leq r \leq s}
      B_r - B_s = X_{j+1} \bigg\}, \text{ and } \Psi_j := \Psi_{j-1} +
      X_j.\]

      \State Define the auxiliary processes
      \begin{align*}
        \tilde{S}(t) := \sum_{j > 0} \sup_{\tau_j \leq s \leq t} B(s)
      \mathbf{1} \left(\tau_j \leq t < \tau_{j+1} \right) \text{ and }
        \tilde{N}(t) := \sum_{j \geq 0} \Psi_j \mathbf{1}(\tau_j \leq
      t < \tau_{j+1}).
      \end{align*}

     \State Let $A(t) := \tilde{N}(t) + \tilde{S}(t),$ and identify the
     time change $\sigma(t) := \inf\{ s : A(s) = m_1 t\}.$ Next,
     take the time changed version $Z(t) := \tilde{S}(\sigma(t)).$\\

     \State Replace $Z(t)$ by $-Z(t)$ and $B(t)$ by $-B(t).$
  \end{algorithmic}
\label{Algo-Skorokhod-Embedding-CPP-BM} 
\end{algorithm}

Algorithm \ref{Algo-Skorokhod-Embedding-CPP-BM} is a brief description
of the coupling developed in \citet{khoshnevisan1993}, where it is also
proved that the process $Z(t)$ output by the construction is indeed
the desired compound Poisson process `closely' coupled with the
Brownian motion $B(t).$ Figure \ref{Fig-Coupled-Paths-Sko-Emb} below
shows a typical coupled path output by Algorithm
\ref{Algo-Skorokhod-Embedding-CPP-BM}. Several independent
replications of such coupled paths are used to simulate the coupled
risk processes $R(t)$ and $R_B(t)$ in Example
\ref{Eg-Ruin-prob-BM-approx}, and $\delta$ is chosen as prescribed by
the confidence intervals (obtained due to a straightforward
application of central limit theorem) for the empirical average cost
of the simulated coupling.

\begin{figure}[h!]
  \centering
  \caption{A coupled path output by Algorithm 1}
   \includegraphics[width = 9cm, height = 6cm]{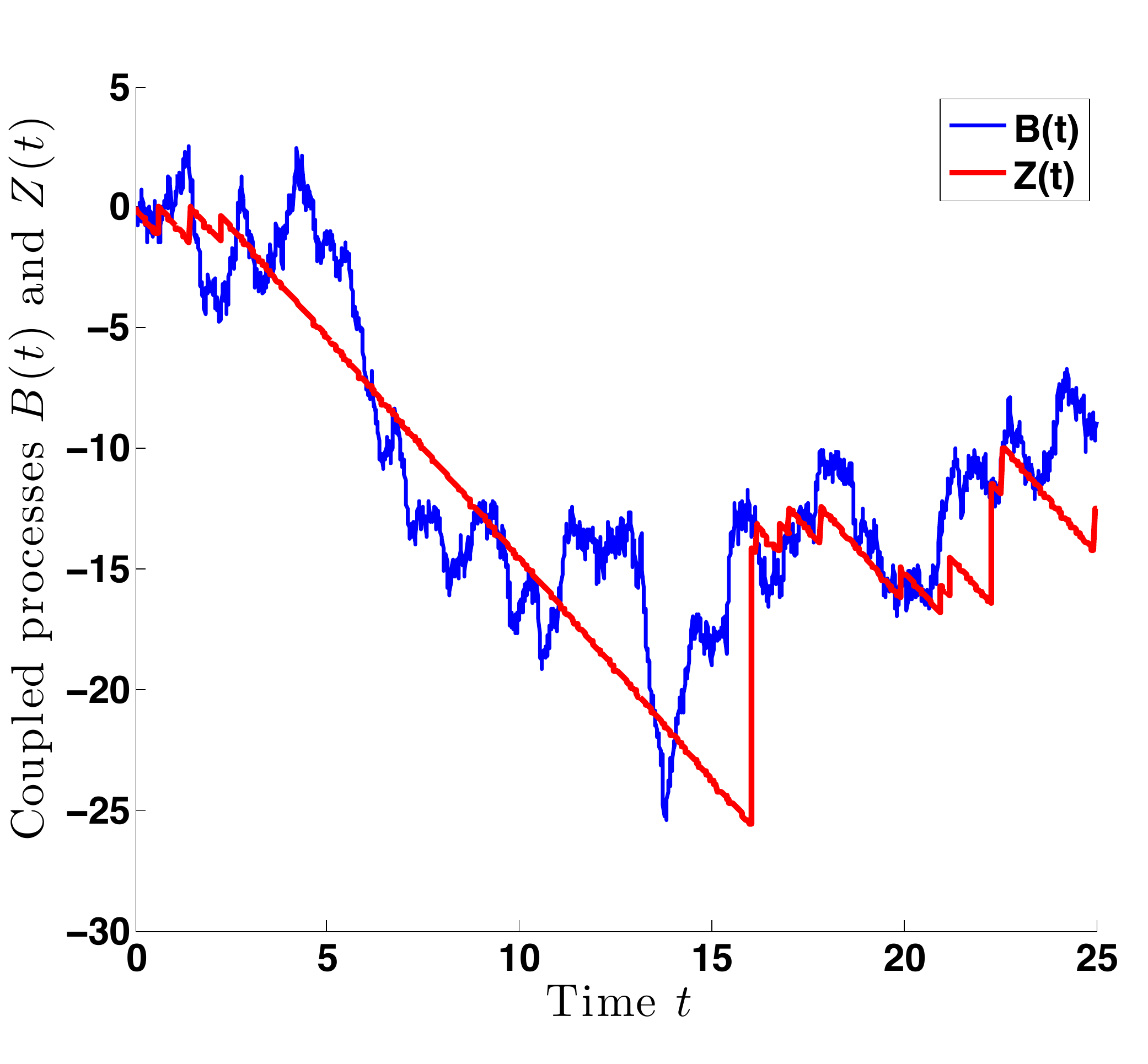} 
\label{Fig-Coupled-Paths-Sko-Emb}
\end{figure} 

\section{Some technical proofs}
\label{Appendix-Proofs}
In this section, we first state and prove all the technical results
that are utilized in Examples \ref{Eg-Ruin-prob-BM-approx},
\ref{Eg-2D-FPP} and \ref{Eg-Choosing-Reinsurance-Prop}. Then we prove
Lemma \ref{Lemma-Ruling-Out-Negative-Measures} and
\ref{Lem-Lim-Int-Interchange}, which are technical results used in
Section \ref{Sec-Duality-Proof} to complete the proof of Theorem
\ref{THM-STRONG-DUALITY}. We conclude the section with Lemma
\ref{Lem-Notn-Clar} and Corollary \ref{Cor-Notn-Clar} that add more
clarity to the notation,
  \begin{align}
    \sup\left\{ \int fd\nu: d_c(\mu,\nu) \leq \delta\right\} :=
    \sup\left\{ \int fd\nu: d_c(\mu,\nu) \leq \delta, \int f^-d\nu <
    \infty\right\}
    \label{Primal-Prob-Interp}
  \end{align}
  adopted while defining the primal problem $I$ in Section
  \ref{Sec-primal-prob}.


\subsection{Technical results used in Examples
  \ref{Eg-Ruin-prob-BM-approx} and
  \ref{Eg-Choosing-Reinsurance-Prop}.}  For the results (Lemma
\ref{Lem-Au-Closed}, \ref{Lem-Ruin-Prob-BM-App} and
\ref{Lem-Exp-Max-BM-App}) stated and proved in this section, let
$S = D([0,T], \mathbb{R})$ be equipped with the $J_1$-topology, and
define
\[A_u := \left\{ x \in D([0,T], \mathbb{R}): \sup_{t \in [0,T]} x(t)
    \geq u \right\}, \quad\quad u \in \mathbb{R}.\]

\begin{lemma}
  For every $u \in \mathbb{R},$ the set $A_u$ is closed.
  \label{Lem-Au-Closed}
\end{lemma}
\begin{proof}
  Pick any $x \notin A_u,$ and let
  $\epsilon := (u - \sup_{t \in [0,T]} x(t))/2.$ As
  $\sup_{t \in [0,T]}x(t) < u,$ $\epsilon$ is strictly positive. Then,
  for every $y \in S$ such that $d_{_{J_1}}(x,y) < \epsilon,$ we have
  that
  \begin{align*}
    \sup_{t \in [0,T]} y(t) \leq \sup_{t \in [0,T]} x(t) +
    d_{_{J_1}}(x,y) \leq \sup_{t \in [0,T]} x(t) + \epsilon, 
  \end{align*}
  which, in turn, is smaller than $u$ because of the way $\epsilon$ is
  chosen. Thus, $\{y \in S: d_{_{J_1}}(x,y) < \epsilon\}$ is a subset
  of $S \setminus A_u,$ and hence $S \setminus A_u$ is open. This
  automatically means that the set $A_u$ is closed. 
\end{proof}

\begin{lemma}
  Let $c(x,y) = d_{_{J_1}}(x,y),$ for $x,y$ in $D([0,T],\mathbb{R}).$
  For any $u \in \mathbb{R},$ $c(x,A_u) := \inf\{ c(x,y): y \in A_u\}$
  is given by
  \[ c(x,A_u) = u - \sup_{t \in [0,T]}x(t) , \]
  for every $x \notin A_u.$
 \label{Lem-Ruin-Prob-BM-App}
 \end{lemma}
 \begin{proof}
   Given $\epsilon > 0$, let
   $z(t) := x(t) + \inf_{t \in [0,T]} (u - x(t)) + \epsilon,$ for all
   $t \in [0,T].$ As $z \in A_u,$ it is immediate from the definition
   of $c(x,A_u)$ that
   \begin{align}
     c(x,A_u) \leq d_{_{J_1}}(x, z) \leq \sup_{t \in [0,T]} \vert
     z(t) - x(t) \vert
     = \inf_{t \in [0,T]} (u-x(t)) + \epsilon = u - \sup_{t \in [0,T]}
     x(t) + \epsilon. 
     \label{Inter-LC-Equiv}
   \end{align}
   Next, let $\Lambda$ be the set of strictly increasing functions
   $\lambda$ mapping the interval $[0,T]$ onto itself, such that both
   $\lambda$ and $\lambda^{-1}$ are continuous. Also, let $e$ be the
   identity map, that is, $e(t) = t,$ for all $t$ in $[0,T].$ Then,
   \begin{align*}
     d_{_{J_1}}(x,y) := \inf_{\lambda \in \Lambda} 
     \big\{\vert \vert x - y \circ \lambda \vert \vert_{\infty} + \vert \vert \lambda - e
     \vert \vert_\infty\big\},
   \end{align*}
   where
   $\vert \vert z \vert \vert_\infty = \sup_{t \in [0,T]}|z(t)|,$ for
   any $z$ in $S.$ Since $c(x,y) = d_{_{J_1}}(x,y),$ we have
   \begin{align*}
     c(x,A_u) = \inf_{y \in A_u} \inf_{\lambda \in \Lambda} 
     \big\{\vert \vert x - y \circ \lambda \vert \vert_{\infty} + \vert \vert \lambda - e
     \vert \vert_\infty \big\}.
   \end{align*}
   Next, as $y \circ \lambda \in A_u$ for any $y \in A_u,$ it is
   immediate that one can restrict to $\lambda = e$ without loss of
   generality, and subsequently, the inner infimum is not
   necessary. In other words, for any $\epsilon > 0,$ 
      \begin{align*}
        c(x,A_u) &= \inf_{y \in A_u}\vert \vert x - y \vert
        \vert_{\infty}  = \inf_{y \in A_u} \sup_{t \in [0,T]}\vert
        x(t) - y(t)\vert\\
        &\geq \inf_{y \in A_u} \left( \inf_{t \in [0,T]:  y_t > u - \epsilon}
          \hspace{-4pt} y(t) - \sup_{t
          \in [0,T]} x(t)\right)
        \geq u - \sup_{t \in [0,T]} x(t) -\epsilon.
   \end{align*}
   As $\epsilon$ is arbitrary, this observation, when combined with
   \eqref{Inter-LC-Equiv}, concludes the proof of Lemma
   \ref{Lem-Ruin-Prob-BM-App}. 
 \end{proof}

\begin{lemma}
  Let $c(x,y) = d_{_{J_1}}^2(x,y)$ for $x,y$ in
  $D([0,T], \mathbb{R}).$ Given nonnegative constants $a_1,a_2$ and
  $\lambda,$ define the functions,
  \begin{align*}
    f(x) = \sup_{t \in [0,T]} \big( a_2 x(t) - a_1t \big) \text{ and }
    \phi_{\lambda}(x) = \sup_{y \in S} \left\{ f(y) - \lambda c(x,y) \right\},
  \end{align*}
  for every $x \in S.$ Then, for any $\delta > 0,$
  \begin{align*}
    \inf_{\lambda \geq 0}\left\{ \lambda \delta + \int
    \phi_{\lambda}(x) d\mu(x) \right\} = \int f(x) d\mu(x) + a_2 \sqrt{\delta}.
  \end{align*}
  \label{Lem-Exp-Max-BM-App}
\end{lemma}
\begin{proof}
  Fix any $x \in S$ and $\lambda > 0.$ For any positive constant $b,$
  as $y(t) = x(t) + b$ is also a member of $S,$ it follows from the
  definition of $\phi_\lambda(x)$ that
  \begin{align*}
    \phi_{\lambda}(x) \geq \sup_{\underset{b \ \geq \ 0}{y: y = x + b}} \left\{
    \sup_{t \in [0,T]} \big( a_2 y(t) - a_1 t\big) - \lambda
   d_{_{J_1}}^2(x,y)\right\}
    =\sup_{b \ \geq \ 0}\left\{
    \sup_{t \in [0,T]} \big( a_2 (x(t) + b) - a_1 t\big) - \lambda
    b^2\right\}.
  \end{align*}
  As the function $a_2 b - \lambda b^2$ is maximized at $b =
  a_2/2\lambda,$ the above lower bound simplifies to 
  \begin{align}
    \label{Inter-Max-Exp-Lem}
    \phi_{\lambda}(x) \ \geq \ f(x) + \frac{a_2^2}{4\lambda}. 
  \end{align}
  To obtain an upper bound, we first observe from the definition of
  the metric $d_{_{J_1}}$ that
  \begin{align*}
    \phi_{\lambda}(x) = \sup_{y \in S} \inf_{\nu \in \Lambda}\left\{ f(y) - \lambda
    \big( \vert \vert  x -  y \circ \nu \vert \vert_\infty + \vert
    \vert \nu - e \vert \vert_\infty\big)^2 \right\},
  \end{align*}
  where the quantities $\Lambda$ and $e$ are defined as in the proof
  of Lemma \ref{Lem-Ruin-Prob-BM-App}. Again, as $y \circ \nu$ lies in
  $S,$ for every time change $\nu \in \Lambda,$ one can take $\nu = e$
  without loss of generality. Consequently,
\begin{align*}
  \phi_{\lambda}(x) = \sup_{y \in S} \left\{ f(y) - \lambda
  \vert \vert  x - y \vert \vert_\infty^2 \right\} = \sup_{z \in S}
  \left\{ f(x + z) - \lambda \vert \vert z \vert \vert_\infty^2 \right\},
  \end{align*}
  where we have also changed the variable from $y-x$ to $z.$ In
  particular, 
  \begin{align*}
    \phi_{\lambda}(x) &= \sup_{z \in S} \left\{ \sup_{t \in [0,T]}
    \big( a_2(x(t) + z(t)) - a_1t\big) - \lambda \Vert  z
                        \Vert_\infty^2 \right\}\\ 
    &\leq \sup_{z \in S} \left\{ \sup_{t \in [0,T]} \big( a_2 x(t) -
      a_1 t\big) + a_2 \hspace{-5pt}\sup_{t \in [0,T]} z(t) - \lambda \Vert  z
       \Vert_\infty^2\right\}\\
    &= f(x) + \sup_{z \in S} \left( a_2  \Vert z 
      \Vert_\infty - \lambda \Vert  z  \Vert_\infty^2
      \right),
  \end{align*}
  which, in turn, is maximized for $||z|| = a_2/2\lambda.$ Combining
  this upper bound with the lower bound in \eqref{Inter-Max-Exp-Lem},
  we obtain that $\phi_\lambda(x) = f(x) + {a_2^2}/{4\lambda}.$ 
Next, it is a straightforward exercise in calculus to verify that the
function $\lambda \delta + \int \phi_\lambda d\mu$ is minimized at
$\lambda^\ast = a_2/2\sqrt{\delta},$ and subsequently, 
\begin{align*}
  \inf_{\lambda \geq 0} \left\{ \lambda \delta + \int \phi_\lambda
  d\mu\right\} = \int fd\mu + a_2\sqrt{\delta}. 
\end{align*}
This completes the proof. \hfill$\Box$
\end{proof}

\subsection{Proofs of results used in Section \ref{Sec-App-FPP}}
For the results (Lemma \ref{Lem-2D-Closure}, \ref{Lem-Sym-2dim-Ruin}
and Corollary \ref{Cor-2D-Closure}) stated and proved in this section,
let $S = D([0,T], \mathbb{R}^2)$ be equipped with the $J_1$-topology
induced by the metric
  \begin{align*}
    d_{J_1}(x,y) = \inf_{\lambda \in \Lambda}\left\{ \sup_{t \in
    [0,T]}\max_{i=1,2}\vert 
    x_i(t) - y_i(\lambda(t))\vert + \sup_{t \in [0,T]} \vert
    \lambda(t) - t \vert \right\},
  \end{align*}
  where the set $\Lambda,$ as in Lemma 3, is the set of strictly
  increasing functions $\lambda$ mapping the interval $[0,T]$ onto
  itself, such that both $\lambda$ and $\lambda^{-1}$ are continuous. 
  {
\begin{lemma}
  For a given vector $a = (a_1, a_2)$ with $a_1, a_2 \geq 0$ and
  $a_1a_2 > 0, $
  \begin{align*}
    \textnormal{cl}\left(\{x \in S: a^Tx(t) \leq 0 \text{ for some } t
    \in [0,T]\}\right) = \left\{x \in
    S: \inf_{t \in [0,T]} a^Tx(t) \leq 0\right\}. 
  \end{align*}
  \label{Lem-2D-Closure}
\end{lemma}
\begin{proof}
  We first show that the set
  $C := \{ x \in S: \inf_{t \in [0,T]}a^Tx(t) \leq 0\}$ is closed by
  showing that its complement is open. Given any $x \notin C,$ let
  $\epsilon := \inf_{t\in [0,T]}a^Tx(t) > 0.$ Then, for every
  $y \in S$ such that
  $d_{_{J_1}} \hspace{-3pt}(x,y) < \epsilon/(2a_1 + 2a_2),$ we have
  that
  \begin{align*}
    \inf_{t \in [0,T]} y_i(t) \geq \inf_{t \in [0,T]} x_i(t) -
    d_{_{J_1}}\hspace{-3pt} (x,y) >  \inf_{t \in [0,T]} x_i(t) -
    \frac{\epsilon}{2(a_1 + a_2)}, 
  \end{align*}
  for $i=1,2.$ Further, as $a_1,a_2$ are non-negative, we obtain that
  \begin{align*}
    \inf_{t \in [0,T]} a^Ty(t) \geq \inf_{t \in [0,T]} a^Tx(t)  - (a_1
    + a_2) \frac{\epsilon}{2(a_1 + a_2)} > 0.
  \end{align*}
  Thus, for every $x \in S \setminus C,$ we can find an $\epsilon > 0$
  such that
  $\{ y \in S: d_{_{J_1}}\hspace{-3pt} (x,y) < \epsilon/(2a_1+2a_2)\}$
  is a subset of $S \setminus C,$ and hence $S \setminus C$ is
  open.  Therefore, $C$ is closed. \\
  \indent Letting
  $D := \{x \in S: a^Tx(t) \leq 0 \text{ for some } t \in [0,T]\},$
  the remaining task is to show that for every $\epsilon > 0$ and
  $x \in C,$ there exists $y \in D$ such that
  $d_{_{J_1}}\hspace{-3pt}(x,y) < \epsilon.$ For any given
  $\epsilon > 0$ and $x \in C,$ we first see that there exists
  $t_0 \in [0,T]$ such that $a^Tx(t_0) < \epsilon(a_1+a_2)/2;$ such a
  $t_0$ exists for every $x \in C$ because
  $\inf_{t \in [0,T]}a^Tx(t) \leq 0$. Then
  $(y(t) = (y_1(t), y_2(t)): t \in [0,T])$ defined as
  $y_i(t) = x_i(t) - \epsilon, i = 1,2$ lies in $D$ because
  $a^Ty(t_0) = a^Tx(t_0) - (a_1 + a_2)\epsilon < 0.$ Therefore, every
  $x$ in $C$ is a closure point of $D.$ Therefore, $\text{cl}(D) = C.$
  \hfill$\Box$ 
\end{proof}}

{
\begin{corollary}
  Let
  $B = \{ (x_1,x_2) \in \mathbb{R}^2: \beta x_1 + x_2 \leq 0\} \cup \{
  (x_1,x_2) \in \mathbb{R}^2: x_1 + \beta x_2 \leq 0\}$ for some
  $\beta \geq 0.$ Then
  $\textnormal{cl}\left(\{x \in S: x(t) \in B \text{ for some } t \in
    [0,T]\}\right)$ equals $A_1 \cup A_2,$ where 
  \begin{align}
    A_1 := \left\{ x \in S: \inf_{t \in [0,T]} (\beta x_1(t)
    + x_2(t)) \leq 0\right\} \text{ and }  A_2 := \left\{ x
    \in S: \inf_{t \in [0,T]} (x_1(t)  + \beta x_2(t)) \leq 0\right\}.  
    \label{A1-A2-Defn}
  \end{align}
\label{Cor-2D-Closure}
\end{corollary}
\begin{proof}
  It follows from Lemma \ref{Lem-2D-Closure} that
  $\text{cl}\left( \{ x \in S: \beta x_1(t) + x_2(t) \leq 0 \text{ for
      some } t \in [0,T]\}\right) = A_1$ and
  $\text{cl}\left( \{ x \in S: x_1(t) + \beta x_2(t) \leq 0 \text{ for
      some } t \in [0,T]\}\right) = A_2.$
  Then the statement to verify follows from the fact that the closure
  of the union of two sets equals the union of closures of those two
  sets.
\end{proof}}

{
\begin{lemma}
  For $x,y \in D([0,T], \mathbb{R}^2),$ let $c(x,y)$ be defined as in
  \eqref{Cost-2dim-Ruin}. Let $A = A_1 \cup A_2,$ where $A_1$ and $A_2$
  are defined as in \eqref{A1-A2-Defn}, and $\beta \geq 0.$ Then for
  $x \in S \setminus A,$
  \begin{align*}
    c(x,A) = \frac{1}{1+\beta^2} \left[\inf_{t \in [0,T]} \big( \beta
    x_1(t) + x_2(t) \big)^2 \wedge \inf_{t \in [0,T]} \big( x_1(t) +
    \beta x_2(t) \big)^2\right],
  \end{align*}
  and for $\lambda^\ast \geq 0,$
  $\left\{ x \in S: c(x,A) \leq 1/\lambda^\ast \right\}$ equals
   \begin{align*}
     \left\{ x \in S: \inf_{t \in [0,T]} (\beta x_1(t) + x_2(t)) \leq
     \sqrt{\frac{1+\beta^2}{\lambda^\ast}}\right\} \cup  \left\{ x \in
     S: \inf_{t \in [0,T]} ( x_1(t) + \beta x_2(t)) \leq
     \sqrt{\frac{1+\beta^2}{\lambda^\ast}}\right\}
   \end{align*}
\label{Lem-Sym-2dim-Ruin}
\end{lemma}
\begin{proof}
  Let us focus on determining $c(x,A_1) = \inf_{y \in A_1} c(x,y)$ for
  $x \notin A_1.$ Given $x \notin A_1$ and $\epsilon > 0,$ we first
  define $\tilde{x}_{\epsilon}(t) := x(t) - (b+\epsilon)u,$ where
  $b := \inf_{t \in [0,T]} (\beta x_1(t) + x_2(t))/\sqrt{1+\beta^2} >
  0$ and $u$ is the unit vector (in $\ell_2$-norm) along the direction
  $[\beta,1].$ As
  \[\beta \tilde{x}_{\epsilon,1}(t) + \tilde{x}_{\epsilon,2}(t) = \beta
    x_1(t) + x_2(t) - (b + \epsilon)\sqrt{1 + \beta^2} \leq 0 \quad
    \text{ for some } t \leq T,\] we have
  $\tilde{x}_\epsilon := (\tilde{x}_{\epsilon,1},
  \tilde{x}_{\epsilon,2})\in A_1.$ As a result, 
  \begin{align}
    c(x,A_1) \leq c(x,\tilde{x}_\epsilon) = \Vert (b + \epsilon)
    u\Vert_2^2 = \left( \inf_{t \in [0,T]}\frac{\beta
    x_1(t) + x_2(t)}{\sqrt{1+\beta^2}} + \epsilon \right)^2
    \label{Inter-2dim-hyperplane}
  \end{align}
  for $x \notin A_1.$ Next, using the same line of reasoning as in the
  proof of Lemma \ref{Lem-Ruin-Prob-BM-App}, one can restrict to time
  changes $\lambda(t) = t$ in the computation of lower bound.
  Then, 
  \begin{align*} 
    c(x,A_1) = \inf_{y \in A_1} \sup_{t \in [0,T]} \vert \vert x(t) -
  y(t) \vert \vert_2^2  
    \geq \inf_{y \in A_1} \sup_{t: \beta y_1(t) + y_2(t) \leq
    \epsilon} \vert \vert x(t) - y(t) \vert \vert_2^2, 
\end{align*}
for any $\epsilon > 0.$ For $z \in \mathbb{R}^2,$ if we let
$g_\epsilon(z) = \inf_{\{y \in \mathbb{R}^2: \beta y_1 + y_2 \leq
  \epsilon\}} \vert \vert z - y \vert \vert_2^2,$ then
\begin{align*}
  c(x,A_1)  \geq \inf_{y \in A_1} \sup_{t: \beta y_1(t) + y_2(t) \leq
  \epsilon } g_\epsilon(x(t)). 
\end{align*}
Next, for $z = (z_1,z_2)$ such that $\beta z_1 + z_2 > \epsilon,$ observe that
$g_\epsilon(z) = (\beta z_1 + z_2-\epsilon)^2/(1+\beta^2).$ Therefore,
\begin{align*}
  c(x,A_1)  \geq \inf_{y \in A_1} \sup_{t: \beta y_1(t) + y_2(t) \leq
  \epsilon} \frac{(\beta x_1(t) + x_2(t) - \epsilon)^2}{1+\beta^2} \geq \inf_{t
  \in [0,T]} \frac{(\beta x_1(t) + x_2(t) - \epsilon)^2}{1+\beta^2}
\end{align*}
for every $\epsilon$ small enough. As $\epsilon$ can be arbitrarily
small, combining the upper bound in \eqref{Inter-2dim-hyperplane} with
the above lower bound results in
\begin{align*}
  c(x,A_1) = \inf_{t \in [0,T]} \frac{\big( \beta x_1(t) + x_2(t)
  \big)^2}{1+\beta^2} \quad\text{ for } x \notin A_1.
\end{align*}
Consequently, $\{x \in S: c(x,A_1) \leq 1/\lambda^\ast\}$ =
$A_1 \cup \{x \in S: c(x,A_1) \in (0,1/\lambda^\ast)\},$ which is
simply,
\begin{align*}
  \left\{ x \in S: \inf_{t \in [0,T]} \big( \beta x_1(t) + x_2(t) 
  \big) \leq \sqrt{\frac{1 + \beta^2}{\lambda^\ast}}\right\} \quad \text{
  for } x \notin A_1.
\end{align*}
Similarly, one can show that
$c(x,A_2) = \inf_{t \in [0,T]} (x_1(t) + \beta x_2(t))^2/(1+\beta^2)$
and a similar expression as above for
$\{x \in S: c(x,A_2) \leq 1/\lambda^\ast\}.$ As $A = A_1 \cup A_2,$ it
is immediate that $c(x,A) = c(x,A_1) \wedge c(x,A_2)$ and
$\{x \in S: c(x,A) \leq \frac{1}{\lambda^\ast}\}$ is the union of two
sets $\{x \in S: c(x,A_i) \leq {1}/{\lambda^\ast}\}, i = 1,2,$ thus
verifying both the claims.  \hfill$\Box$
\end{proof}}


\subsection{Technical results used in Section \ref{Sec-Duality-Proof}
  to complete the proof of Theorem \ref{THM-STRONG-DUALITY}.}
{\begin{lemma} Let $S \times S$ be a Polish space that is compact, and
    $f:S \rightarrow \mathbb{R}$ is upper semicontinuous. Suppose that
    $\pi \in M(S \times S),$ with the Jordan decomposition
    $\pi = \pi^+ - \pi^-$ comprising positive measures $\pi^+$ and
    $\pi^-,$ is such that $\pi^+(A) = 0 < \pi^-(A) < \infty$ for some
    $A \in \mathcal{B}(S \times S).$ Then, \begin{align*} \inf \left\{
        \int gd\pi: g \in C_b(S \times S), g(x,y) \geq f(y) \right\} =
      -\infty.  \end{align*} \label{Lemma-Ruling-Out-Negative-Measures} \end{lemma} \begin{proof}
We first observe that if $\pi$ is not
    As any finite Borel measure on a Polish space is regular, there
    exists a compact set $K_\epsilon$ and an open set $O_\epsilon$
    such that $K_\epsilon \subseteq A \subseteq O_\epsilon$
    and
    \[\pi^-(O_\epsilon) -\epsilon \ \leq \ \pi^-(A) \ \leq \ \pi^-(K_\epsilon)
      + \epsilon, \quad \text{ and } \quad \pi^+(O_\epsilon) \ \leq \ 
      \epsilon,\] for any given $\epsilon > 0$ (see Lemma 18.5 in
    \citet{aliprantis1998principles}). Further, as $S \times S$ is
    compact, Urysohn's lemma (see, for
    example, Theorem 10.8 in \citet{aliprantis1998principles})
    guarantees us the existence of a continuous function
    $h: S \times S \rightarrow [0,1]$ such that $h(x,y) = 1$ for all
    $x \in K_\epsilon$ and $h(x,y) = 0$ for all $x \notin O_\epsilon.$
    In that case, choosing $\epsilon < \pi^-(A)/2,$ we have
    $\inf_{n \geq 1}\int g_nd\pi = -\infty$ for the sequence of
    continuous functions $g_n(x,y) = nh(x,y) + \sup_{x \in S} f(x).$
    This is because, $\sup_{x \in S} f(x) < \infty$ (recall that $f$
    is upper semicontinuous and $S$ is compact), and
  \begin{align*}
    \int hd\pi = \int hd\pi^+ - \int hd\pi^- \  \leq \ \pi^+(O_\epsilon) -
    \pi^-(K_\epsilon) \ \leq \ 2\epsilon - \pi^-(A) \ < \ 0 
  \end{align*}
  when $\epsilon < \pi^-(A)/2.$  
  As $g_n(x,y) \geq f(y)$ for all
  $x,y$ in S, it follows that
  $\inf_{g \in D} \int gd\pi \leq \inf_n \int g_nd\pi = -\infty$ as
  well, whenever there exists $A$ such that $\pi^-(A) > 0.$
  \hfill$\Box$
\end{proof}}
\begin{lemma}
  Suppose that Assumptions (A1) and (A2) are in force. Let
  $(S_n: n \geq 1)$ be an increasing sequence of subsets of $S,$ and
  $(\lambda_n: n \geq 1)$ be a real-valued sequence satisfying
  $\lambda_n \rightarrow \lambda^\ast,$ for some
  $\lambda^\ast \geq 0,$ as $n \rightarrow \infty.$ Then for any
  $x \in S,$
  \begin{align*}
    \varliminf_n \sup_{y \in S_n} \big\{ f(y) - \lambda_n c(x,y)\big\}
    \geq \sup_{y \in \cup_n S_n} \left\{ f(y) - \lambda^* c(x,y)
    \right\}. 
  \end{align*}
  \label{Lem-Lim-Int-Interchange}
\end{lemma}
\begin{proof}
  {Pick any $x \in S.$ For brevity, let
  $g(y,\lambda) := f(y) - \lambda c(x,y)$ (hiding the dependence on
  the fixed choice $x \in S$) and
  $\bar{\lambda}_m := \sup_{k \geq m} \lambda_k.$ The observations
  that\\
    (A) $g(y,\lambda)$ is a non-increasing function in $\lambda,$
    and\\ (B)
    $ \sup_{y \in S_n}g(y,\lambda) \leq \sup_{y \in S_{n+1}}
    g(y,\lambda)$ for $n \geq 1,$\\
    will be used repeatedly throughout this proof. While Observation
    (A) follows from the fact that $S_n \subseteq S_{n+1},$
    Observation (B) is true because $c(\cdot,\cdot)$ is
    non-negative. The quantity of interest, 
  \begin{align*}
    \varliminf_n \sup_{y \in S_n} \big\{ f(y) - \lambda_n c(x,y)\big\}
    = \sup_{n > 0} \inf_{m \geq n} \sup_{y \in S_m} g(y,\lambda_m)
    \geq \sup_{n > 0} \inf_{m \geq n} \sup_{y \in S_n}  g(y,\lambda_m)
    \geq \sup_{n > 0} \sup_{y \in S_n} \inf_{m \geq n} g(y,
    \lambda_m), 
  \end{align*}
  where the first inequality follows from Observation (B), and the
  second inequality from the exchange of inf and sup operations. As
  $\inf_{m \geq n}g(y,\lambda_m) = f(y) - (\sup_{m \geq n}\lambda_m)
  c(x,y) = g(y,\bar{\lambda}_n),$ we obtain
  \begin{align}
    \varliminf_n \sup_{y \in S_n} \big\{ f(y) - \lambda_n c(x,y)\big\}
    \geq \sup_{n > 0} \sup_{y \in S_n} g(y,\bar{\lambda}_n) = \sup_{y
    \in \cup_n S_n} g(y,\lambda^\ast),
    \label{Liminf-Lemma-Inter}
  \end{align}
  where the rest of this proof is devoted to justify the last equality
  in \eqref{Liminf-Lemma-Inter}: As
  $\sup_{y \in S_n}g(y,\bar{\lambda}_n)$ is non-decreasing in $n,$ it
  is immediate that
  $\sup_n \sup_{y \in S_n} g(y,\bar{\lambda}_n) \leq \bar{g} :=
  \sup_{y \in \cup_n S_n} g(y,\lambda^\ast).$ To show that
  $\sup_n \sup_{y \in S_n} g(y,\bar{\lambda}_n)$ indeed equals
  $\bar{g},$ we use $\hat{S}$ to denote $\hat{S} := \cup_n S_n,$
  and consider distinct cases:\\
  \textsc{Case 1: $\bar{g} < \infty.$} If $\bar{g}$ is finite, then
  there exists $y_\epsilon \in \hat{S}$ such that
  $g(y_\epsilon,\lambda^\ast) \geq \bar{g} - \epsilon/2$ for any
  $\epsilon > 0.$ Further, as $g(y_\epsilon,\lambda)$ is continuous in
  $\lambda,$ there exists $n_0$ large enough such that
  $g(y_\epsilon,\bar{\lambda}_n) \geq g(y_\epsilon,\lambda^\ast) -
  \epsilon/2,$ and consequently,
  $g(y_\epsilon,\bar{\lambda}_n) \geq \bar{g} - \epsilon,$ for all
  $n \geq n_0.$ Therefore, for all $n > n_0$ satisfying
  $y_\epsilon \in S_n,$ it follows that
  $\sup_{y \in S_n} g(y,\bar{\lambda}_n) \geq g(y_\epsilon,
  \bar{\lambda}_n) \geq \bar{g} - \epsilon.$ Since $\epsilon$ is
  arbitrary, we obtain
  $\sup_n \sup_{y \in S_n} g(y,\bar{\lambda}_n) = \bar{g}$ whenever
  $\bar{g}$ is finite. \\
  \textsc{Case 2: $\bar{g} = \infty.$} If
  $\bar{g} = \sup_{y \in \hat{S}} g(y,\lambda^\ast) = \infty,$ then
  there exists a strictly increasing subsequence $(n_k: k \geq 1)$ of
  natural numbers such that
  $\sup_{y \in \hat{S}} g(y,\bar{\lambda}_{n_k}) > k$ for all $k$
  (this is because, being a pointwise supremum of family of continuous
  functions of $\lambda,$ $\sup_{y \in \hat{S}} g(y,\lambda)$ is a
  lower semicontinuous function of $\lambda$). Next, as
  $\sup_{y \in \hat{S}} g(y,\bar{\lambda}_{n_k}) > k,$ there exists a
  sequence $(y_{k}: k \geq 1)$ comprising elements of $\hat{S}$ such
  that $g(y_{k},\bar{\lambda}_{n_k}) \geq k/2$ for all $k \geq 1.$ As
  $(S_n: n \geq 1)$ is a sequence of sets increasing to $\hat{S}$, one
  can identify a strictly increasing subsequence $(m_k: k \geq 1)$ of
  natural numbers such that $y_{k} \in S_{m_k}$ and consequently,
  $\sup_{y \in S_{m_k}} g(y, \bar{\lambda}_{n_k}) \geq k/2$ for all
  $k \geq 1.$ Next, if we let $l_k = n_k \vee m_k,$ then
  $\bar{\lambda}_{l_k} \leq \bar{\lambda}_{n_k}, S_{m_k} \subseteq
  S_{l_k},$ and it follows from Observations (A) and (B) that
  $\sup_{y \in S_{l_k}} g(y, \bar{\lambda}_{l_k}) \geq k/2$ as
  well. This results in a non-decreasing sequence $(l_k: k \geq 1)$
  with $l_k \rightarrow \infty$ such that
  $\sup_{y \in S_{l_k}} g(y, \bar{\lambda}_{l_k}) \geq k/2,$ thus
  yielding
  $\sup_{n} \sup_{y \in S_n} g(y,\bar{\lambda}_n) = \infty =
  \bar{g}.$}

  The proof is complete because
  $\sup_{y \in \cup_n S_n}g(y,\lambda^\ast)$ in
  \eqref{Liminf-Lemma-Inter} is the desired right hand side.
  \hfill$\Box$
\end{proof}

\subsection{A technical result to add clarity to the notation
  \eqref{Primal-Prob-Interp}.} 
\begin{lemma}
  Suppose that Assumptions (A1) and (A2) are in force. Then, for every
  $\pi \in \Phi_{\mu,\delta},$ there exists a
  $\pi^\prime \in \Phi_{\mu,\delta}$ such that $\pi^\prime$ is
  concentrated on $\{(x,y) \in S \times S: f(y) \geq f(x)\},$ along
  with satisfying,
  \begin{align}
    \int f^+(y)d\pi^\prime(x,y) \geq \int f^+(y)d\pi(x,y) \quad \text{ and
    } \quad \int f^-(y)d\pi^\prime(x,y) \leq \int f^-(x)d\mu(x).
    \label{Modified-Coupling}
    \end{align}
    Further, $\int f(y)d\pi^\prime(x,y) \geq \int f(y)d\pi(x,y),$
    whenever the integral $\int f(y)d\pi(x,y)$ is well-defined.
  \label{Lem-Notn-Clar}
\end{lemma}
\begin{proof}
 Let $(X,Y)$ be jointly distributed
  according to $\pi.$ Using $(X,Y),$ let us define a new jointly
  distributed pair $(X',Y'),$ with joint distribution denoted by
  $\pi^\prime,$ as follows:
  \begin{align}
    X' := X \quad\text{ and } \quad Y' := Y I\left( f(X) \leq
    f(Y) \right) + X I\left(f(X) > f(Y) \right),
    \label{Coup-Defn}
  \end{align}
  where $I(\cdot)$ denotes the indicator function. As $X' = X,$ the
  marginal distribution of $X'$ is $\mu.$ Further, it follows from the
  definition of $(X',Y')$ in \eqref{Coup-Defn} that
  \begin{align*}
    \int cd\pi^\prime = \int_{\{f(x) \leq f(y)\}} \hspace{-20pt} c(x,y)d\pi(x,y) +
                        \int_{\{f(x) > f(y)\}} \hspace{-20pt} c(x,x)d\pi(x,y) =
    \int_{\{f(x) \leq f(y)\}} \hspace{-20pt} c(x,y)d\pi(x,y),
  \end{align*}
  because $c(x,x) = 0$ for every $x$ in $S.$ In addition, as
  $c(\cdot,\cdot)$ is non-negative, it follows that
  $\int cd\pi^\prime \leq \int cd\pi \leq \delta.$
  Therefore, $\pi \in \Phi_{\mu,\delta}.$ Next,
    \begin{align*}
      \int f^+(y)d\pi^\prime(x,y) 
      &= \ \ \int_{\{f(x) \leq f(y)\}} \hspace{-20pt}
        f^+(y)d\pi(x,y) + \int_{\{f(x) > f(y)\}} \hspace{-20pt}
        f^+(x)d\pi(x,y)\\
      &= \ \int_{\{f^+(x) \leq f^+(y)\}} \hspace{-20pt}
        f^+(y)d\pi(x,y) + \int_{\{f^+(x) > f^+(y)\}} \hspace{-20pt}
        f^+(x)d\pi(x,y)\\
      &\geq \int_{\{f^+(x) \leq f^+(y)\}} \hspace{-20pt}
        f^+(y)d\pi(x,y) + \int_{\{f^+(x) > f^+(y)\}} \hspace{-20pt}
        f^+(y)d\pi(x,y) = \int f^+(y)d\pi(x,y), \quad\quad \text{ and
        }\\
     \int f^-(y)d\pi^\prime(x,y) 
      &= \ \ \int_{\{f(x) \leq f(y)\}} \hspace{-20pt}
        f^-(y)d\pi(x,y) + \int_{\{f(x) > f(y)\}} \hspace{-20pt}
        f^-(x)d\pi(x,y)\\
      &= \ \int_{\{f^-(x) \geq f^-(y)\}} \hspace{-20pt}
        f^-(y)d\pi(x,y) + \int_{\{f^-(x) < f^-(y)\}} \hspace{-20pt}
        f^-(x)d\pi(x,y)\\
      &\leq \int_{\{f^-(x) \geq f^-(y)\}} \hspace{-20pt}
        f^-(x)d\pi(x,y) + \int_{\{f^-(x) < f^-(y)\}} \hspace{-20pt}
        f^-(x)d\pi(x,y) = \int f^-(x)d\pi(x,y) = \int f^-d\mu, 
  \end{align*}
  As $\int f^-d\mu < \infty,$ $\int f(y)d\pi^\prime(x,y)$ is always
  well-defined. Further, when $\int f(y)d\pi(x,y)$ is also well
  defined, it is immediate from the definition of $\pi^\prime$ that, 
    \begin{align*}
    \int f(y)d\pi^\prime(x,y) &= \ \ \int_{\{f(x) \leq f(y)\}} \hspace{-20pt}
    f(y)d\pi(x,y) + \int_{\{f(x) > f(y)\}} \hspace{-20pt}
                           f(x)d\pi(x,y)\\
    &\geq \int_{\{f(x) \leq f(y)\}} \hspace{-20pt}
    f(y)d\pi(x,y) + \int_{\{f(x) > f(y)\}} \hspace{-20pt}
      f(y)d\pi(x,y) = \int f(y)d\pi(x,y),
  \end{align*}
  thus verifying all the claims made in the statement. 
\end{proof}

{
  \begin{corollary}
    Suppose that Assumptions (A1) and (A2) are in force. Let
    $\nu \in {P}(S)$ be such that $\int f^- d\nu = \infty$ and
    $d_c(\mu,\nu) \leq \delta$ for a given reference probability
    measure $\mu \in {P}(S).$ Then there exists $\nu' \in P(S)$ such
    that $\int f^+ d\nu' \geq \int f^+ d\nu,$
    $\int f^-d\nu' < \infty = \int f^-d\nu,$ and
    $d_c(\mu,\nu') \leq \delta.$ 
  \label{Cor-Notn-Clar}
\end{corollary}
\begin{proof}
  Let $\pi \in \Pi(\mu,\nu)$ be an optimal coupling that attains the
  infimum in the definition of $d_c(\mu,\nu);$ in other words,
  $d_c(\mu,\nu) = \int cd\pi$
  (such an optimal coupling $\pi$ always exists because of the lower
  semicontinuity of $c(\cdot,\cdot),$ see Theorem 4.1 in
  \citet{villani2008optimal}). Since $d_c(\mu,\nu) \leq \delta$ and
  $\pi(\cdot \times S) = \mu(\cdot),$ we have that
  $\pi \in \Phi_{\mu,\delta}.$ Then, according to Lemma
  \ref{Lem-Notn-Clar}, there exists $\pi^\prime \in \Phi_{\mu,\delta}$
  satisfying \eqref{Modified-Coupling}. Consequently, the marginal
  distribution defined by
  $\nu^\prime(\cdot) := \pi^\prime(S \times \cdot)$ satisfies
  $d_c(\mu,\nu^\prime) \leq \delta,$ along with
  $\int f^+d\nu^\prime \geq \int f^+d\nu$ and
  $\int f^-d\nu^\prime \leq \int f^-d\mu < \infty = \int f^-d\nu.$
\end{proof}
}


\end{document}